\documentclass[11pt, letterpaper]{article}
\usepackage{arxiv}
\usepackage{booktabs,threeparttable,xcolor,colortbl}
\allowdisplaybreaks

\title{\Large\textbf{Optimal Online and Offline Algorithms for Contextual MNL\\ with Applications to Assortment and Pricing}}

\author{
\textbf{Yunfan Zhang}\thanks{Equal contribution, reverse-alphabetical ordered.}\\
{\small Stern School of Business, New York University}\\
{\small \texttt{yz11751@stern.nyu.edu}}
\and
\textbf{Yuxuan Han}\footnotemark[1]\\
{\small Stern School of Business, New York University}\\
{\small \texttt{yh6061@stern.nyu.edu}}
\and
\textbf{Hongyu Shan}\\
{\small Department of Finance, China Europe International Business School (CEIBS)}\\
{\small \texttt{hongyushan@ceibs.edu}}
\and
\textbf{Jose Blanchet}\\
{\small Department of Management Science \& Engineering, Stanford University}\\
{\small \texttt{jose.blanchet@stanford.edu}}
\and
\textbf{Zhengyuan Zhou}\\
{\small Stern School of Business, New York University}\\
{\small \texttt{zzhou@stern.nyu.edu}}
}

\date{}

\begin{document}

%%%%%%%%%%%%%%%%%%%% title/author/date/abstract/content %%%%%%%%%%%%%%%%%%%%

\maketitle

\begin{abstract}
Selecting which products to display and at what prices is a central decision in retail and e-commerce operations. In many applications, these two choices must be made jointly under limited display capacity and uncertain customer demand.
In this paper, we study the joint assortment and pricing problem under a price-based contextual multinomial logit model, where customer preferences depend on both product features and selling prices. Our analysis begins with the construction of a new confidence region for demand estimation under price-dependent features. Building on this result, we develop a pessimistic offline algorithm and SupCB-type online algorithms for joint assortment and pricing optimization. In the offline setting, we establish a suboptimality guarantee governed by local information around the optimal assortment-price pair, rather than by exact coverage of the optimal action. In the online setting, our SupCB-type algorithm improves the best previously known regret bound to $\widetilde{O}(W\sqrt{dT\log N}/L_0)$, and we also provide a computationally simpler Thompson-sampling alternative. When specialized to the assortment-only or pricing-only setting, our bounds recover the near-minimax-optimal rates established in those respective domains, thereby bridging the gap between the mature study on assortment optimization or dynamic pricing and the limited literature on their joint optimization.
\end{abstract}


%%%%%%%%%%%%%%%%%%%% input main text here %%%%%%%%%%%%%%%%%%%%

\section{Introduction}

In modern commerce, sellers managing a large inventory face two closely related decisions: which products to offer and at what prices. These correspond to two classical problems that have been extensively studied in isolation. In assortment optimization, prices are fixed and the seller selects a subset of products, typically subject to a capacity constraint, to maximize expected revenue. In dynamic pricing, the offered assortment is fixed, and the seller adjusts prices to maximize revenue while learning unknown demand. In practice, however, these two decisions are often made jointly.
For example, when a customer searches for ``laptops’’ on an e-commerce platform, limited screen space allows only a small number of products to be displayed out of thousands of candidates. At the same time, the platform must determine the price or promotional discount for each displayed item. This naturally leads to the study of \textit{joint assortment and pricing optimization} under \textit{cardinality constraints}, where the seller aims to simultaneously determine an assortment of at most $K$ out of $N$ items and their corresponding prices to maximize expected revenue.

A central difficulty is that the expected revenue of any assortment-pricing decision depends not only on posted prices but also on customer choice behavior, which is typically unknown and must be learned from data. Discrete-choice models provide a natural framework for capturing such behavior, and the multinomial logit (MNL) model is a particularly attractive benchmark because of its computational tractability and broad use in revenue management; see, for example, \citet{vanRyzinMahajan1999,rusmevichientong2010dynamic,wang2012capacitated,berbeglia2022comparative,feldman2022alibaba}. Motivated by these considerations, we study joint assortment and pricing under a price-based contextual MNL model.
Specifically, each item $i\in[N]$ is associated with a feature vector $\bx_i\in\mathbb{R}^d$ (e.g., color, texture, brand, quality, reviews), and its attraction under price $p_i$ is modeled as
$$v_i(p_i) := \exp\left( \langle \bm\psi^\star, \bm x_i \rangle - \langle \bm \phi^\star, \bm x_i \rangle p_i \right),$$
where $\bm\psi^\star$ captures the customer’s intrinsic preference and $\bm\phi^\star$ captures price sensitivity. Given an assortment $S$, the customer’s choice then follows the standard multinomial distribution as described in \eqref{eq-mnl-choice-prob}.

Under this formulation, the difficulty described above becomes a statistical learning problem: the seller must estimate the unknown parameter pair $(\bm\psi^\star,\bm\phi^\star)$ from observed purchase data. In other words, the seller must learn both customers’ baseline preferences and their feature-dependent price sensitivities while simultaneously making revenue-generating decisions. This formulation generalizes the standard linear MNL model as the special case of fixed prices, and it also recovers the dynamic pricing setting when the assortment is fixed throughout. Our framework therefore provides a unified model for assortment optimization, dynamic pricing, and their joint formulation.

Although assortment optimization and dynamic pricing have each been studied extensively, comparatively fewer papers consider their joint formulation, especially in contextual settings where observable item features influence both baseline preference and price sensitivity. In this paper, we study this problem under a price-based contextual MNL model in both online and offline learning settings.

The online setting has been the primary focus of the existing literature. It captures applications such as e-commerce and digital retail, where customers arrive sequentially and the seller can continuously adapt assortments and prices using the purchase data accumulated so far. For the joint problem, \citet{miao2021dynamic} initiated the study of joint assortment and pricing under the MNL model, and recent contextual work \citep{erginbas2025online} has established $\widetilde{\mathcal O}(d\sqrt{KT}/L_0)$ regret bounds for related online settings. By contrast, in the simpler contextual assortment-only MNL setting, the sharp benchmark regret is known to be $\widetilde{\mathcal O}(\sqrt{dT})$ \citep{oh2021multinomial,han2025improved}.
This naturally raises the question of whether the contextual joint problem can attain the same sharp benchmark. Our main result shows that this is indeed possible. We propose a SupCB-type algorithm that achieves the optimal $\widetilde{\mathcal O}(\sqrt{dT}/L_0)$ regret bound for contextual joint assortment and pricing, and the same analysis also yields the corresponding sharp guarantees for dynamic pricing and assortment optimization as special cases.

Despite its importance, the online setting captures only part of the practical picture. In many real-world applications, due to high experimentation costs, strict operational constraints, or the risk of degrading the customer experience, firms cannot experiment with live assortment-pricing policies. This motivates offline learning, where the seller must instead rely on historical transaction data collected under past decisions and learn a strong policy before deployment. Existing work on offline learning under MNL choice models is largely restricted to assortment-only problems \citep{dong2023pasta,han2025improved}. To the best of our knowledge, comparable theoretical guarantees for offline joint assortment and pricing have not been developed. Our work bridges this gap by extending the analysis to the offline setting, providing sharp instance-dependent guarantees for learning joint assortment-pricing policies from historical data.

Together, these results show that our price-based contextual MNL framework supports sharp guarantees for joint assortment and pricing, while recovering assortment-only and pricing-only problems as special cases. We compare our guarantees with the existing literature in Table~\ref{tab:comparison} and summarize our contributions below.
 %% ---- BEGIN linear_mnl_table_offline.tex ----

\begin{table*}[t]
\centering
\small
\setlength{\tabcolsep}{6pt}
\renewcommand{\arraystretch}{1.3}
\resizebox{\textwidth}{!}{%
\begin{tabular}{lcccccc}
\toprule
\textbf{Reference} & \textbf{Assortment} & \textbf{Pricing} & \textbf{Context} 
  & \textbf{Upper Bound} & \textbf{Lower Bound} \\
\midrule
\multicolumn{6}{c}{\textsc{Online Learning}} \\[2pt]
\midrule
\cite{chen2018note}       & \cmark & \xmark & \xmark 
  & --                                                        & $\Omega(\sqrt{NT})$ \\
\cite{agrawal2019mnl}     & \cmark & \xmark & \xmark 
  & $\widetilde{\mathcal{O}}(\sqrt{NT})$                     & $\Omega(\sqrt{NT/K})$ \\
\cite{chen2020dynamic}    & \cmark & \xmark & \cmark 
  & $\widetilde{\mathcal{O}}(d\sqrt{T})$                     & $\Omega(d\sqrt{T}/K)$ \\
\cite{oh2021multinomial}  & \cmark & \xmark & \cmark 
  & $\widetilde{\mathcal{O}}(\kappa^{-1}\sqrt{dT\log N})$    & -- \\
\cite{miao2021dynamic}    & \cmark & \cmark & \xmark 
  & $\widetilde{\mathcal{O}}(e^{W+1/L_0}\sqrt{NT})^{\ddagger}$ & -- \\
\cite{perivier2022dynamic}& \xmark & \cmark & \cmark 
  & $\widetilde{\mathcal{O}}(d\sqrt{T})$                     & -- \\
\cite{lee2024nearly}      & \cmark & \xmark & \cmark 
  & $\widetilde{\mathcal{O}}(d\sqrt{T})$                     & $\Omega(d\sqrt{T})$ \\
\cite{erginbas2025online} & \cmark & \cmark & \cmark 
  & $\widetilde{\mathcal{O}}(e^{W}d\sqrt{KT}/L_0)^{\S}$     & $\Omega(d\sqrt{T}/L_0)$ \\[2pt]
\rowcolor{gray!12}
\textbf{This paper}       & \cmark & \cmark & \cmark 
  & $\widetilde{\mathcal{O}}\!\left(W\sqrt{dT\log N}/L_0\right)$ 
  & $\Omega(\sqrt{dT})$ \\
\midrule[0.4pt]
\multicolumn{6}{c}{\textsc{Offline Learning}} \\[2pt]
\midrule
\cite{dong2023pasta}      & \cmark & \xmark & \cmark 
  & $\widetilde{\mathcal{O}}\!\left(\sqrt{d/(\kappa' n_{S^\star})}\right)^{\dagger}$ & -- \\
\cite{han2025optimal}     & \cmark & \xmark & \xmark 
  & $\widetilde{\mathcal{O}}(K/\sqrt{n^\star})^{\diamond}$              & $\Omega(K/\sqrt{n^\star})$ \\[2pt]
\rowcolor{gray!12}
\textbf{This paper}       & \cmark & \cmark & \cmark 
  & $\widetilde{\mathcal{O}}\!\left(\dfrac{W}{L_0}
      \sqrt{\displaystyle\sum_{j\in S^\star}q_0^\star q_j^\star
            \|\widetilde{\bx}_j(p_j^\star)\|^2_{\bH_{\widetilde{\mathcal{D}}}^{-1}}}
    \right)$
  & $\Omega\!\left(\dfrac{W}{L_0}
      \sqrt{\displaystyle\sum_{j\in S^\star}q_0^\star q_j^\star
            \|\widetilde{\bx}_j(p_j^\star)\|^2_{\bH_{\widetilde{\mathcal{D}}}^{-1}}}
    \right)$ \\[6pt]
\bottomrule
\end{tabular}%
}
\caption{Comparison of leading-order guarantees for online and offline learning in MNL assortment optimization, dynamic pricing, and joint assortment--pricing problems. Here $N$ denotes the number of items, $K$ the assortment size, $d$ the feature dimension, $L_0$ the price-sensitivity parameter, $W$ the parameter radius, and $q_j^\star$ denote the choice probabilities evaluated at the optimal solution $(S^\star,\bp^\star|\bm\vartheta^\star)$. $^{\ddagger}$~In \cite{miao2021dynamic}, the factor $e^{W+L_0^{-1}}$ is absorbed into the constant $c_0$ in Theorem~1. $^{\S}$~Although \cite{erginbas2025online} states the result for $W=1$, Lemma~C.1 implies an exponential dependence on $W$ for general $W>0$. Offline guarantees rely on different notions of data coverage: $^{\dagger}$~\citet{dong2023pasta} requires \emph{assortment-level coverage}, captured by $n_{S^\star}=\sum_{m=1}^n \mathbf{1}\{S_m=S^\star\}$; $^{\diamond}$~\citet{han2025optimal} establishes a sharp \emph{item-level coverage} benchmark, where $n^\star=\min_{j\in S^\star}\sum_{m=1}^n \mathbf{1}\{j\in S_m\}$. Our offline guarantee is instead characterized by a \emph{local information-weighted item-level} quantity that reduces to the $\widetilde{\Theta}(K/\sqrt{n^\star})$ benchmark under canonical features and balanced coverage.}
\label{tab:comparison}
\end{table*}
%% ---- END linear_mnl_table_offline.tex ----

\subsection{Our Contributions}

\paragraph{Improved Confidence Region for Price-Based MNL MLE. } 
Our first contribution is a sharper, variance-aware confidence region for the regularized MLE that avoids explicit dependence on $\kappa$. Specifically, for the MLE $\widehat{\bm \vartheta}_{\widetilde\cD}^\lambda$ computed on dataset $\widetilde{\cD}$, we establish that with high probability:
$$|\widetilde{\bm x}(p)^{\top}(\widehat{\bm \vartheta}_{\widetilde\cD}^\lambda-\bm \vartheta^\star)| =\widetilde{\mathcal O} \left(\|\widetilde{\bm x}( p)\|_{\bH_{\widetilde\cD}^\lambda\left(\bm\vartheta^\star\right)^{-1}} \right),\quad \forall \lVert \bx\rVert \leq 1,\ p\in[0,P],$$
where $\bH_{\widetilde\cD}^\lambda(\bm\vartheta^\star)$ is the log-likelihood Hessian. 
By exploiting the self-concordant-like geometry of the MNL likelihood \citep{perivier2022dynamic,agrawal2023tractable,lee2024nearly}, we obtain a local prediction bound governed by the inverse Hessian. This result is the key technical foundation for our subsequent analysis and enables sharper offline and online guarantees.

\paragraph{Offline Joint Assortment and Pricing via Pessimism.} 
Building on the improved confidence region, we develop a pessimism-based offline algorithm and establish a sharp instance-dependent guarantee for joint assortment and pricing under an item-level coverage condition. The leading term scales as $\widetilde{\mathcal{O}}\left(\frac{W}{L_0}\sqrt{\sum_{j\in S^\star}q_0^\star q_j^\star \| \widetilde{\bx}_j( p^\star_j)\|^2_{\bH_{\widetilde\cD}^{-1}}}\right)$.  
% We also show that a burn-in-free variant can be obtained by plugging in alternative confidence regions, at the cost of additional $\sqrt{dK\bar P^4W^3}$ multiplicative dependency, and under a fixed data-collection policy our result yields a clean asymptotic rate in terms of the limiting pricing Hessian.

\paragraph{Online Joint Assortment and Pricing.} 
By leveraging the improved confidence region, we develop a SupCB-style algorithm for the online problem and obtain the regret bound $\widetilde{\mathcal O}\left(W\sqrt{dT\log N}/L_0\right).$ This improves the dependence of dimension in the leading term over the $\widetilde{\mathcal O}(e^Wd\sqrt{KT}/L_0)$ guarantee of \citet{erginbas2025online} whenever $N=o(2^d)$. We further show that a computationally simpler Thompson-sampling variant achieves Bayesian regret $\widetilde{\mathcal O}\left(Wd\sqrt{T}/L_0\right).$

% An earlier conference version studied only the assortment-only linear MNL setting~\citep{han2025improved}. The present paper extends that work to the more general price-based contextual MNL model and establishes new offline and online guarantees for joint assortment and pricing. The linear MNL results are recovered as a special case.
An earlier conference version \citep{han2025improved} studied only the assortment-only linear MNL setting. The present paper extends that work to the more general price-based contextual MNL model and establishes new offline and online guarantees for joint assortment and pricing. The linear MNL results are recovered as a special case.

%-----------------------------related work-------------------------------%
\subsection{Related Works}

\paragraph{Dynamic Assortment optimization under MNL.}

The study of dynamic assortment optimization with demand learning was initiated by \citet{caro2007dynamic}. \citet{rusmevichientong2010dynamic} formulated the canonical capacitated assortment problem under the MNL model and established one of the earliest regret guarantees $O(N^2\log^2 T)$ in this setting. \citet{saure2013optimal} subsequently studied a more general dynamic demand-learning framework that includes the MNL model as a special case. The regret frontier for the classical $N$-item MNL problem was substantially sharpened by \citet{agrawal2017thompson,agrawal2019mnl}, who analyzed bandit algorithms for MNL-based assortment selection and achieved a regret of order $\widetilde{\mathcal O}(\sqrt{NT})$. Complementing these upper bounds, \citet{chen2018note} established a matching lower bound, so the canonical assortment-only MNL problem is now understood to have minimax regret $\widetilde{\Theta}(\sqrt{NT})$ up to logarithmic factors. Subsequent work has further refined the structural and computational understanding of this benchmark problem \citep{chen2021optimal,saha2024stop}.

A parallel line of work studies contextual (linear) MNL assortment optimization, which is closer to our setting because utilities are parameterized by item features. \citet{chen2020dynamic} initiated the study of dynamic assortment optimization under time-varying contextual information and obtained a regret bound of $\widetilde{\mathcal O}(d\sqrt{T})$, together with a lower bound of order $\Omega(d\sqrt{T}/K)$. Building on earlier Thompson-sampling ideas for multinomial logit contextual bandits \citep{oh2019thompson}, \citet{oh2021multinomial} showed that the linear structure can be exploited more aggressively to achieve regret of order $\widetilde{\mathcal O}(\kappa^{-1}\sqrt{dT}\log N)$, improving the dependence on the feature dimension $d$ at the cost of a multiplicative dependence on the problem-dependent constant $\kappa^{-1}$. Subsequent work revisited exactly this tradeoff through sharper confidence analyses for the MNL likelihood. \citet{perivier2022dynamic} improved the dependence on $\kappa$ by establishing a regret bound of order $\widetilde{\mathcal O}(d\sqrt{\kappa T}+\kappa^{-1}\log T)$. \citet{lee2024nearly} further improved the bound to $\widetilde{\mathcal O}(d\sqrt{T})$, a $\kappa$-free bound that is nearly minimax optimal relative to the known lower bound. Most recently, \citet{lee2025improved} further refined the online confidence analysis for multinomial logistic bandits. Taken together, these results highlight a fundamental trade-off in the linear MNL assortment problem: tighter dependence on the feature dimension $d$ has come at the cost of a multiplicative $\kappa^{-1}$ factor, and eliminating this factor without sacrificing the dimension dependence has been the central technical challenge.

In contrast, the offline assortment literature remains sparse. \citet{dong2023pasta} proposed a pessimistic method for linear MNL assortment optimization whose guarantee scales with the number of exact observations of the optimal assortment $S^\star$. More recently, \citet{han2025optimal} showed that such exact coverage is unnecessary in the canonical $N$-item model. By using an item-wise pessimistic principle, they established the minimax item-coverage rate $K/\sqrt{\min_{i\in S^\star} n_i}$ in the general reward setting. These papers are the closest offline baselines to our work, but both rely on discrete coverage notions, either exact coverage of $S^\star$ or item-wise coverage of the optimal items, rather than the finer local information exploited in our analysis.

\paragraph{Dynamic pricing with demand learning.}

Dynamic pricing with demand learning has generated a substantial literature, but the line most relevant to our paper is contextual pricing, where the seller learns a continuous price from observable features. 
Classical single-product models without contextual information establish the basic regret benchmarks. \citet{broder2012dynamic} show that under general parametric demand the regret is $\Theta(\sqrt{T})$, while stronger separation conditions can improve it to $\Theta(\log T)$. The more relevant development for our setting is contextual single-product pricing. In stationary stochastic models, \citet{zhao2024contextual} establish the nearly minimax benchmark $\widetilde{\Theta}(\sqrt{dT})$ under generalized linear demand. Under adversarial context sequences, the regret depends on the structural assumptions: with contextual elasticity, \citet{xu2023elasticity} obtain a $\widetilde{\mathcal O}(\sqrt{dT})$ guarantee up to logarithmic factors, while under stronger known-noise assumptions \citet{xu2021logarithmic} achieve $O(d\log T)$ regret. 

Moving from one product to many products introduces substitution effects and is therefore closer to our setting. In a contextual multi-product MNL pricing model with heterogeneous price sensitivity, \citet{javanmard2020multi} prove regret $O(\log(Td)(\sqrt{T}+d\log T))$. More closely related to our setting, \citet{perivier2022dynamic} study contextual MNL pricing under adversarial arrivals and obtain an $O(d\sqrt{T}\log T)$ regret bound.

\paragraph{Joint assortment and pricing under MNL.}
Compared with the two separate strands above, the literature on joint assortment and pricing is much more limited. Under known demand, a body of work studies the static joint optimization under MNL and related choice models \citep{hopp2005product,LiHuh2011,wang2012capacitated,Besbes2016Assortment,song2021demand,ke2025modeling}. Learning with unknown demand is much more recent. \citet{miao2021dynamic}proposed the first online learning algorithm for dynamic joint assortment and pricing under an MNL model in the canonical setting, establishing a Bayesian regret guarantee with explicit dependence on problem-specific constants. 
Subsequent work has extended this line in two directions. \citet{erginbas2025online} study the contextual joint problem, where customer preferences depend on observed features, and establish a regret bound of order $\widetilde{\mathcal O}(d\sqrt{KT}/L_0)$ together with a lower bound $\Omega(d\sqrt{T}/L_0)$.  \citet{kim2025dynamic} address a related formulation with censored demand feedback, in which customers first screen out overpriced products before making a multinomial choice among the remaining items. They combine LCB-based pricing with either UCB or Thompson sampling for assortment selection, obtaining regret bounds $\widetilde{\mathcal O}(d^{3/2}\sqrt{T/\kappa})$ and $\widetilde{\mathcal O}(d^{2}\sqrt{T/\kappa})$, respectively.

%---------------------Preliminary-----------------------------%
%% ---- BEGIN 2-Preliminary.tex ----
%---------------------Preliminary-----------------------------%

\paragraph{Notations.} 
Through out the paper, for any integer $N$ we denote $[N]:= \{1,\dots,N\}$. For any real numbers $a,b$ we use the notion $a \vee b$ to denote $\max\{a,b\}$, and $a \wedge b$ to denote $\min\{a,b\}$. 
% Moreover, we use the notation $a\lesssim b$ if there exists some absolute constant $c>0$ so that $a< c b$. 
 We also use the notation $a = \widetilde{\mathcal O}(b)$ to indicate that $a \leq Cb$ for some constant $C$ that depends only logarithmically on $T$ and $K$.
For any subset $S$ of $[N]$, we denote $S_+:= S\cup \{0\}$ . For matrices and vectors, given any vector $\bm x\in \mathbb{R}^d$ and $A\in \mathbb{R}^{d\times d}$, we denote $\lVert \cdot \rVert$ the $\ell_2$-norm, and $\lVert \bm x \rVert_{A}:= \sqrt{\bm x^\top A\bm x},$ we also denote $A^\dagger$ the pseudo inverse of $A$.

\section{Preliminary}
% \yf{Edited: format/$\cdot$, $x\to\bm x$, $\theta\to\bm\theta$, $p_i(S|\bm v)\to q_i(S|\bm v)$}\\
%\paragraph{Revenue Maximization under the Linear MNL Model.} 

%---------------------Notations-----------------------------%
%-------------------------------------------------------------%

% \yf{general confidence bound - MNL: online- offline - pricing : online, offline }

We study the assortment optimization problem, which models the interaction between a \emph{seller} and a \emph{customer}. Let $[N]$ denote the set of $N$ available products/items. An assortment $S \subseteq [N]$ represents the subset of products that the seller offers to the customer. When presented with assortment $S$, the customer chooses a product from the choice set $S_+ = \{0\} \cup S$, where $\{0\}$ represents the no-purchase option. In the $N$-item MNL model, each item has an attraction value $v_i\geq 0$. The no-purchase option is normalized to have an attraction value $v_0=1$. When a customer encounters assortment $S$, the probability that he/she will choose product $i\in S_+$ is given by
\begin{equation}\label{eq-mnl-choice-prob}
    q_i(S\lvert \bm v):= \frac{v_i}{1 + \sum_{j\in S} v_j}.
\end{equation}  
Each item $i$ generates a revenue $r_i$ when purchased, while the no-purchase option generates no revenue: $r_0 = 0$. The seller's goal is to maximize the expected revenue from the selected assortment, defined as
\begin{equation}\label{preliminary-eq: revenue-v}
R(S\lvert \bm v): = \sum_{i\in S} r_i q_i(S\lvert \bm v) = \frac{  \sum_{i\in S} r_i v_i}{1+\sum_{j\in S} v_j}.
\end{equation}
Following standard research conventions, we consider assortments $S$ of size at most $K$, i.e., $\mathcal{S}_K=\{S\in[N]:|S|\le K\}$. The optimal assortment is defined as $S^\star = \text{argmax}_{S\in \mathcal{S}_K}R(S\lvert \bv)$.
Although this step involves solving a maximization problem over an exponentially large set, it has been well-studied in the assortment optimization literature. In particular, several well-known algorithms \citep{rusmevichientong2010dynamic, davis2013assortment, avadhanula2016tightness} can solve it in polynomial time.

%---------------------linear mnl-----------------------------%
%-------------------------------------------------------------%
\subsection{Assortment Optimization under the Linear MNL Model}\label{preliminary-subsection: assortment-optimization}
In the $d$-dimensional linear MNL model, the attraction values are parameterized by feature vectors. Specifically, each item $i\in [N]$ is associated with a feature vector $\bm x_i \in \mathbb{R}^d,$ and there exists an underlying parameter $\bm\theta^\star\in \mathbb{R}^d$ such that $$v_i = \exp(\bm x_i^\top\bm\theta^\star).$$
The collection of all feature vectors forms the design matrix $\bm{X}:= (\bm x_1,\dots,\bm x_N)\in \mathbb{R}^{N\times d}$. Under this parameterization, the choice probabilities and the expected revenue become functions of $\bm\theta^\star$:
\begin{align}
\label{preliminary-eq: linear-choice}
q_i(S\lvert \exp(\bm{X}^\top \bm\theta^\star))= \frac{ \exp(\bm x_i^\top\bm\theta^\star)}{1+\sum_{j\in S} \exp(\bm x_j^\top\bm\theta^\star)},
\\
\label{preliminary-eq: linear-revenue}
R(S\lvert \exp(\bm{X}^\top \bm\theta^\star))= \frac{\sum_{i\in S} r_i\exp(\bm x_i^\top\bm\theta^\star)}{1+\sum_{j\in S} \exp(\bm x_j^\top\bm\theta^\star)}.
\end{align}
The optimal decision $S^\star$ is defined by $S^\star= \text{argmax}_{S \in \mathcal{S}_K} R(S\lvert \exp(\bm{X}^\top \bm\theta^\star))$. 
In addition, we impose the following assumption on the $\bX$ and $\bm\theta^\star:$

\begin{assumption}[Boundedness]\label{assump-bounded}
For some constants $M>0$ and $W>0$, we assume
$r_i\in[0,M]$, $\|\bm\theta^\star\|_2 \le {W}$, and
$\|\bm x_{i}\|_2 \le 1$ for all $i \in [N]$.
\end{assumption}

In particular, in the known $\bm\theta^\star$ setting, the assortment optimization problem is still computationally feasible as we discussed below~\eqref{preliminary-eq: revenue-v}. In this work, our focus is the data-driven setting, where instead of knowing $\bm\theta^\star,$ the seller can either have access to a pre-collected dataset or interact with customers for certain rounds, as described below:

%---------------------offline & online-----------------------------%
\paragraph{Offline Assortment Optimization.} In the offline setting, the seller does not know the underlying parameter $\bm\theta^\star$ but can access a pre-collected dataset $\{(i_m,S_m,\bX_m)\}_{m = 1}^n$ consisting of the choice-assortment pairs $(i_m,S_m)$ and time-varying item features $\bX_m$ satisfying Assumption~\ref{assump-bounded}, where for each given $S_m$, the corresponding $i_m$ is sampled independently from the linear MNL choice model with parameter $\bm\theta^\star$ and feature $\bX_m$. The seller's goal is to approximate the optimal assortment based on this dataset under a specific feature $\Xoff$. The learning objective is the sub-optimality gap, defined as 
    \begin{align*}
    \SubOpt(S, \Xoff):= R(S_{\text{off}}^\star\lvert \exp(\Xoff^\top \bm\theta^\star))- R(S\lvert \exp(\Xoff^\top \bm\theta^\star)),
    \end{align*}
which measures the difference between the revenue achieved by a given assortment $S$ and the maximum achievable revenue under the true model.
The optimal offline decision is given by $S_{\text{off}}^\star \in \text{argmax}_{S \in \mathcal{S}_K} R(S\lvert \exp(\Xoff^\top \bm\theta^\star))$. 

\paragraph{Online Assortment Optimization.} In the online setting, the seller interacts sequentially with a sequence of $T$ customers. At each round $t$, the seller observes a set of contextual vectors $\bm X_t=(\bm x_{t1},\dots,\bm x_{tN})$ satisfying Assumption~\ref{assump-bounded}. Based on this context and the history of past interactions, the seller provides an assortment $S_t\in \mathcal{S}_K$ to the customer, and then receives a feedback $i_t$ drawn according to the distribution specified in \eqref{eq-mnl-choice-prob}. The seller's objective is to design an adaptive policy $\pi = (\pi_1,\dots,\pi_T)$ that sequentially selects $S_t$ to minimize the cumulative regret over $T$ rounds, which is defined as the total sub-optimality of the policy $\pi$ over $T$ rounds, 
\begin{align*}
    \text{Reg}(T) &= \E\left[\sum_{t = 1}^T \left(R(S_t^\star\lvert \exp(\bm{X}_t^\top \bm\theta^\star))-R(S_t\lvert \exp(\bm{X}_t^\top \bm\theta^\star))\right)\right].
\end{align*}
where the optimal decision in round $t$ is given by $S_t^\star \in  \text{argmax}_{S \in \mathcal{S}_K} R(S\lvert \exp(\bm X_t^\top \bm\theta^\star))$.

%---------------------assumption-----------------------------%
% \noindent Those optimization analyses will work under the standard boundedness assumption.
%we assume that the attraction values $v_i$ and the underlying parameter norm $\lVert \bm\theta^\star \rVert$ are bounded, which is standard and appears in almost all the linear MNL bandit literature.

\vspace{2mm}
\noindent In the linear MNL model, we define the problem parameters 
\begin{align}
\label{eq-kappa-def-off}
\Koff:&= \min_{S\in \mathcal{S}_K, i\in S, \lVert \bm\theta \rVert_2 \leq W} q_i(S\lvert \exp(\Xoff^\top \bm\theta)) q_0(S\lvert \exp(\Xoff^\top \bm\theta)), \\
\label{eq-kappa-def-online}
\Konline :&= \min_{\substack{t\in [T], S\in \mathcal{S}_K\\  i\in S, \lVert \bm\theta \rVert_2 \leq W}}  q_i(S\lvert \exp(\bm{X}_t^\top \bm\theta)) q_0(S\lvert \exp(\bm{X}_t^\top \bm\theta)),    
\end{align}
corresponding to the offline and online settings, respectively. It can be observed that both $\Koff^{-1}$ and $\Konline^{-1}$ may scale as $\exp(W)$ in the worst case. This exponential dependency on $W$ has motivated a series of previous works on MNL bandits~\citep{chen2020dynamic, perivier2022dynamic, agrawal2023tractable, lee2024nearly}, which aim to mitigate the dependency on $\Koff^{-1}$ or $\Konline^{-1}$ in the sample complexity.

%---------------------mnl pricing-----------------------------%
%-------------------------------------------------------------%
\subsection{Joint Assortment and Pricing under Linear MNL}
The linear MNL framework can be extended to the joint assortment and pricing problems, where the seller needs to determine not only the assortment $S\in \mathcal{S}_K$ but also the price vector $\bm p \in \mathbb{R}_+^N$, representing the price for each items\footnote{Strictly speaking, the seller only needs to specify prices for the items included in $S$, since prices outside $S$ do not affect the revenue. We consider $\bp\in \mathbb{R}^N$ for notational simplicity.}.
In this setting, the revenue $r_i$ for item $i$ is its price $p_i$. We continue to model customer choice behavior using the linear MNL framework, but incorporate a price-dependent formulation of the attraction values. Specifically, for an item with feature vector $\bx_i \in \mathbb{R}^d$, the effect of price on its attraction value is then formulated by a log-linear model: 
\begin{equation}\label{preliminary-eq: utlity-with-price}
v_i(p_i) := \exp\left( \langle \bm\psi^\star, \bm x_i \rangle - \langle \bm \phi^\star, \bm x_i\rangle p_i \right) ,    
\end{equation}
where  $\bm\psi^\star\in \mathbb{R}^d$ and $\bm\phi^\star \in \mathbb{R}^d$ are parameter vectors. In this formulation, $\alpha_i:=\langle \bm\psi^\star, \bm x_i \rangle$ represents the customer's intrinsic preference of the item $i$, while $\beta_i:=\langle \bm \phi^\star, \bm x_i\rangle$ captures the customer's price sensitivity.

For analytical convenience, the utility function can be written as a single linear interaction between two augmented vectors. Define $\bm\vartheta^\star:=(\bm\psi^\star,\bm\phi^\star)\in \mathbb{R}^{2d}$, $\widetilde{\bm x}_i(p_i):=(\bm x_i,-p_i\bm x_i)\in \mathbb{R}^{2d}$, and let the corresponding design matrix be $\widetilde{\bm{X}}(\bm p):= (\widetilde{\bm x}_1(p_1),\dots,\widetilde{\bm x}_N(p_N))\in \mathbb{R}^{N\times 2d}$. Then, equation~\eqref{preliminary-eq: utlity-with-price} can be equivalently written as $v_i(p_i)=\exp(\widetilde{\bm x}_i(p_i)^\top \bm\vartheta^\star)$. Consequently, the choice probabilities and the expected revenue under each pair $(S, \bp)$ and $i \in S$ can be expressed as follows:
\begin{align}
q_i(S\lvert \exp(\widetilde{\bm{X}}(\bm p)^\top \bm\vartheta^\star))& = \frac{ \exp(\widetilde{\bm x}_i(p_i)^\top \bm\vartheta^\star)}{1+\sum_{j\in S} \exp(\widetilde{\bm x}_j(p_j)^\top \bm\vartheta^\star)},
\label{preliminary-eq: choice-with-price}\\
R(S\lvert \exp(\widetilde{\bm{X}}(\bm p)^\top \bm\vartheta^\star)) & = \frac{\sum_{i\in S} p_i\exp(\widetilde{\bm x}_i(p_i)^\top \bm\vartheta^\star)}{1+\sum_{j\in S} \exp(\widetilde{\bm x}_j(p_j)^\top \bm\vartheta^\star)}.\label{preliminary-eq: revenue-with-price}
\end{align}
Similar to the assortment optimization problem, when $\bm\vartheta^\star$ is known,  the joint assortment-pricing problem of finding optimal $(S, \bp)$ pair to maximize~\eqref{preliminary-eq: revenue-with-price} can be solved in polynomial time using existing algorithms \citep{wang2012capacitated,ke2025modeling}. As discussed in~Section~\ref{preliminary-subsection: assortment-optimization}, our interest lies in the data-driven setting without the knowledge of $\bm\vartheta^\star$, as detailed below.

Similar to Section~\ref{preliminary-subsection: assortment-optimization}, we assume the bounded parameter assumptions in this setting.
\begin{assumption}[Boundedness]\label{assump-bounded-pricing} For some constant $W>0$, we assume $\|\bm\vartheta^\star\|_2 \le {W}$, and
$\|\bm x_{i}\|_2 \le 1$ for all $i \in [N]$.
\end{assumption}

\paragraph{Offline Assortment and Pricing Optimization.} Similar as in Section~\ref{preliminary-subsection: assortment-optimization}, in the offline setting, the seller does not know the underlying parameter $\bm\vartheta^\star$ but can access a pre-collected dataset $\{(i_m,S_m, \bp_m,\bX_m)\}_{m = 1}^n$, the only difference here is the additional observation of per-time price $\bp_m$, and the choice model generating $i_m$ now follows~\eqref{preliminary-eq: choice-with-price} with some underlying parameter parameter $\bm\vartheta^\star$ and feature $\widetilde{\bX}_m(\bp_m)$. The seller's goal is to approximate the optimal assortment based on this dataset under a specific feature $\Xoff$. The learning objective is the sub-optimality gap, defined as 
    \begin{align*}
    \SubOpt(S,\bm p, \Xoff):= R(S_{\text{off}}^\star\lvert \exp(\widetilde{\bX}_{\text{off}}(\bm p_{\text{off}}^\star)^\top \bm\vartheta^\star))- R(S\lvert \exp(\widetilde{\bX}_{\text{off}}(\bm p)^\top \bm\vartheta^\star)),
    \end{align*}
where the optimal offline decision is given by $(S_{\text{off}}^\star, \bm{p}_{\text{off}}^\star) \in \text{argmax}_{S \in \mathcal{S}_K,\, \bm{p} \in \mathbb{R}^N} R(S\lvert \exp(\widetilde{\bX}_{\text{off}}(\bm p)^\top \bm\vartheta^\star))$.

\paragraph{Online Assortment and Pricing Optimization.} This setting is analogous to the online assortment optimization problem under the linear MNL model, with the key difference that the seller must simultaneously choose both the assortment and the price vector, $(S_t,\bm p_t)$, at each round $t$. The goal is to design a policy $\pi$ that adaptively generates $\{S_t,\bm p_t\}_{t=1}^T$ to minimize the cumulative regret:
    \begin{align}
        \text{Reg}(T)
        %:=\mathbb{E}\left[ \sum_{t=1}^T \left(R_t(S_t^\star, \bm{p}_t^\star\lvert \bm\vartheta^\star) - R_t(S_t, \bm{p}_t\lvert \bm\vartheta^\star)\right)\right],\\
        := \mathbb{E}\left[ \sum_{t=1}^T \left(R(S_t^\star\lvert \exp(\widetilde{\bm{X}}_t(\bm p^\star_t)^\top \bm\vartheta^\star)) - R(S_t\lvert \exp(\widetilde{\bm{X}}_t(\bm p_t)^\top \bm\vartheta^\star)) \right)\right].
    \end{align}
where the optimal decision in round $t$ is given by $(S_t^\star, \bm{p}_t^\star) \in \text{argmax}_{S \in \mathcal{S}_K,\, \bm{p} \in \mathbb{R}^N} R(S\lvert \exp(\widetilde{\bm X}_t(\bm p)^\top \bm\vartheta^\star))$.

%-----------------------Assumptions----------------------%

\begin{assumption}[Minimum Price Sensitivity]\label{assump-mim-price}
    There exists some constant $L_0>0$ such that $\langle \bm{\phi}^\star, \bm{x}_{ti} \rangle \ge L_0 $ for all $t\in[T]$ and $i\in[N]$.
\end{assumption} 
Assumption~\ref{assump-mim-price} has also been adopted in prior work ~\citep{miao2021dynamic, erginbas2025online}. It ensures that the utility function $v_{i}(p)$ is strictly decreasing, and the positive lower bound $L_0>0$ is shown to be necessary for achieving sublinear regret as in \cite{erginbas2025online}. An additional important implication of Assumption~\ref{assump-mim-price} is that, for every $t \in [T]$ and $i \in [N]$, the optimal price component $p_{ti}^\star$ is guaranteed to lie within a finite interval $[0, P]$.

\begin{lemma}\label{p-bound-refined}
    Under Assumption~\ref{assump-bounded-pricing} and~\ref{assump-mim-price}, we have $\bm{p}^\star_t$ satisfies $0\leq p_{ti}^\star  \leq P$ with $P=\frac{3+W+\log K}{L_0}$ for all  $i\in S^\star_t$ and $t\in [T]$.
\end{lemma}

Combining Assumption~\ref{assump-bounded-pricing} and Lemma~\ref{p-bound-refined} yields $\|\widetilde{\bx}_{ti}(p_{ti}^\star)\|_2 \le \bar{P} := \sqrt{1+P^2}$. Therefore, in the subsequent analysis, the setting where $\|\bx\|_2 \le 1$ and $p \in [0,P]$ implies that $\|\widetilde{\bx}(p)\|_2 \le \bar{P}$. For notational convenience, we define the feasible price domain as $\mathcal{P} := [0,P]$.

Thus, it is sufficient to consider the optimization problem over $\mathcal{S}_K \times \mathcal{P}^N$ instead of $\mathcal{S}_K \times \mathbb{R}_+^N$. The problem parameter in this setting is defined as 
\begin{align}
\label{eq-tilde-kappa-off}
\tilde\kappa_{\text{off}}:&= \min_{\substack{
S\in \mathcal{S}_K, \bm p\in  \mathcal{P}^N, \\
 i\in S , \lVert \bm\vartheta \rVert_2 \leq W
}} q_i(S\lvert \exp(\widetilde{\bX}_{\text{off}}(\bm p)^\top \bm\vartheta^\star)) q_0(S\lvert \exp(\widetilde{\bX}_{\text{off}}(\bm p)^\top \bm\vartheta^\star)), \\
\label{eq-tilde-kappa-online}
\tilde\kappa_{\text{on}} :&=  \min_{\substack{
t\in [T], S\in \mathcal{S}_K, \bm p\in \mathcal{P}^N,  \\
 i\in S, \lVert \bm\vartheta \rVert_2 \leq W
}}  q_i(S\lvert \exp(\widetilde{\bm X}_t(\bm p)^\top \bm\vartheta^\star)) q_0(S\lvert \exp(\widetilde{\bm X}_t(\bm p)^\top \bm\vartheta^\star)).   
\end{align}

%% ---- END 2-Preliminary.tex ----

%--------------Improved Confidence Region---------------------%
%% ---- BEGIN 3-Confidence Bound.tex ----
%--------------Improved Confidence Region---------------------%
\section{Improved Confidence Region}

Given the extensive body of research on the linear contextual MNL model, the most widely adopted approach for utilizing the parametric forms in~\eqref{preliminary-eq: linear-choice} and~\eqref{preliminary-eq: choice-with-price} relies on MLE-based estimators or their variants \citep{erginbas2025online,perivier2022dynamic,oh2021multinomial,agrawal2023tractable,miao2021dynamic}. In this section, we develop an improved MLE confidence region for a \emph{price-based demand model}, which strictly generalizes the standard linear contextual MNL model by allowing price-dependent features.
As a consequence, our result also yields an improved confidence region for the linear MNL MLE as a special case

%---------------------MNL Pricing-----------------------------%
\subsection{General Guarantees under Price-based Demand}\label{confidenceb-bound-subsec: linear pricing}

We first establish the improved confidence region for the price-based demand model.  
This setting generalizes the linear MNL by incorporating price-dependent features.
Let $\widetilde{\bm{X}}(\bm{p}) := (\widetilde{\bm{x}}_{1}(p_{1}), \dots, \widetilde{\bm{x}}_{N}(p_{N}))$ represent the price-parameterized feature matrix. Given any dataset $\widetilde{\cD} := \{i_m, \widetilde{\bm{X}}_m(\bm{p}_m), S_m\}_{m=1}^t$ and $\lambda\ge 0$, the price-based $\lambda$-regularized log-likelihood function is defined as
\begin{align}\label{eq-price-mle}
    \ell^\lambda_{\widetilde{\cD}}(\bm\vartheta):= -\sum_{m=1}^t\sum_{j\in (S_m)_+}y_{mj}\log q_{j}(S_m\lvert \exp(\widetilde{\bm{X}}_m(\bm p_m)^\top \bm\vartheta)) - \frac{\lambda}{2}\lVert \bm \vartheta \rVert_2^2
\end{align}
where $y_{mj} = \bm{1}\{i_m = j\}.$ We also define the Hessian matrix $\bH_{\widetilde\cD}^\lambda(\bm\vartheta):= \nabla^2_{\bm\vartheta }\ell^\lambda_{\widetilde\cD}(\bm\vartheta) $ and $\lambda$-regularized MLE $\hat{\bm \vartheta}_{\widetilde{\cD}}^\lambda:= \operatorname*{argmin}_{\bm\vartheta\in \mathbb{R}^{2d}} \ell^\lambda_{\widetilde{\cD}}(\bm\vartheta)$ associated to~\eqref{eq-price-mle}. In addition, we simply set $\bH_{\widetilde\cD}(\bm\vartheta), \hat{\bm\vartheta}_{\widetilde\cD}$ as un-regularized version with $\lambda = 0.$

% Our first result establishes a confidence interval for $\widetilde{\bm x}(p)^\top \hat{\bm\vartheta}_{\widetilde\cD}^\lambda$ with any $\lVert \bm x \rVert_2 \leq 1$ and $p\in[0,P]$ given the following \textit{conditional independence assumption:}

 Then, we present an improved confidence region result for the MLE estimator given the following \textit{conditional independence assumption:}
\begin{assumption}\label{assumption-conditional-independence-pricing}
Condition on $\{\widetilde{\bm X}_m(\bm p_m),S_m\}_{m = 1}^t,$ the observed choices $\{i_m\}_{m = 1}^t$ are mutually independent.
\end{assumption}

\begin{theorem}\label{prop-sup-ucb-confidence-bound}
Given $\widetilde{\cD}= \{i_m,\widetilde{\bm X}_m(\bm p_m),S_m\}_{m=1}^t$, and under Assumption~\ref{assumption-conditional-independence-pricing} and \ref{assump-bounded-pricing}, for any $\bm x \in \mathbb{R}^d$ with $\|\bm x\|_2 \leq 1$ and $p \in [0,P]$, we have condition on 
\begin{equation}\label{eq-burn-in}
    \max_{m\le t,j\in S_m} \lVert \widetilde{\bm x}_{mj}( p_{mj}) \rVert_{\bH_{\widetilde\cD}^{\lambda}(\bm\vartheta^\star)^{-1}} \leq  \frac{1}{144\sqrt{2d\log(N/\delta)}}\wedge \frac{1}{24\sqrt\lambda W},
\end{equation} 
it holds that with probability at least $1-\delta$ that 

i)  $$ \frac{1}{3}\bH_{\widetilde\cD}^\lambda(\bm\vartheta^\star) \preceq \bH_{\widetilde\cD}^\lambda(\widehat{\bm \vartheta}_{\widetilde\cD}^\lambda) \preceq 3\bH_{\widetilde\cD}^\lambda(\bm\vartheta^\star).$$

ii) 
$$|\widetilde{\bm x}(p)^{\top}(\widehat{\bm \vartheta}_{\widetilde\cD}^\lambda-\bm \vartheta^\star)| \leq 16\|\widetilde{\bm x}( p)\|_{\bH_{\widetilde\cD}^\lambda\left(\bm\vartheta^\star\right)^{-1}} \left( \sqrt{\log({NT}/{\delta})} + \sqrt{\lambda}W \right).
$$
\end{theorem}

Compared to the confidence region results in \cite{erginbas2025online}, Theorem~\ref{prop-sup-ucb-confidence-bound} removes the $\sqrt{d}$ term from the confidence radius with an extra $\log N$ factor, this is a critical step toward attaining the $\widetilde{\mathcal O}(\sqrt{dT})$ regret bound in online learning.

%---------------------------------------------------------%
\subsection{Implications for the Linear MNL Model}\label{sec3.2:linear-confidence bound}
The price-based model contains the linear contextual MNL model as a special case when the feature map is price-independent.
Concretely, suppose the price-dependent features satisfy
\[
\widetilde{\bm x}_{mj}(p)\equiv \begin{pmatrix}\bm x_{mj}\\ \bm 0\end{pmatrix}\in\mathbb R^{2d}\quad \text{for all }p\in[0,P],
\]
and the true parameter has the form $\bm\vartheta^\star=(\bm\theta^\star,\bm 0)\in\mathbb R^{2d}$, so that the utilities reduce to
$\widetilde{\bm x}_{mj}(p_{mj})^\top\bm\vartheta^\star=\bm x_{mj}^\top\bm\theta^\star$.
Then the induced choice probabilities coincide with those of the linear MNL model with parameter $\bm\theta^\star$.
In this case, Theorem~\ref{prop-sup-ucb-confidence-bound} directly implies the following confidence region for the linear MNL MLE, with the same sharp dependence on $d$.

\begin{assumption}\label{assumption-conditional-independence}
Condition on $\{(\bm{X}_m,S_m)\}_{m = 1}^t,$ the observed choices $\{i_m\}_{m = 1}^t$ are mutually independent.
\end{assumption}

\begin{corollary}\label{thm-sup-ucb-confidence-bound}
Given $\cD= \{(i_m, \bm{X}_m, S_m)\}_{m=1}^t$, and under  Assumption~\ref{assump-bounded} and~\ref{assumption-conditional-independence}, for any $\bm x \in \mathbb{R}^d$ with $\|\bm x\|_2 \leq 1$, 
we have condition on 
$$\max_{m\leq t, j\in S_m} \lVert   \bm x_{mj} \rVert_{\bH_{\cD}^\lambda(\bm\theta^\star)^{-1}} \leq  \frac{1}{144\sqrt{d\log(N/\delta)}}\wedge \frac{1}{24\sqrt\lambda W},$$ it holds that with probability at least $1-\delta$

i)  $$ \frac{1}{3}\bH_{\cD}^\lambda(\bm\theta^\star) \preceq \bH_{\cD}^\lambda(\hat{\bm\theta}_{\cD}^\lambda) \preceq 3\bH_{\cD}^\lambda(\bm\theta^\star).$$

ii) 
$$|\bm x^{\top}(\hat{\bm\theta}_{\cD}^\lambda-\bm\theta^\star)| \leq 16\|\bm x\|_{\bH_{\cD}^\lambda\left(\bm\theta^\star\right)^{-1}} \left(\sqrt{\log (N/ \delta)} + \sqrt{\lambda}W \right).
$$
\end{corollary}
%---------------------------------------------------------%
\paragraph{On the Role of Regularization.} Previous studies on generalized linear MLE and linear MNL MLE confidence regions that achieve $\sqrt{d}$-independent results \citep{li2017provably,oh2021multinomial,jun2021improved} all consider the unregularized MLE, i.e. $\lambda = 0$, which requires the Hessian matrix $\bH_{\cD}(\bm\theta^\star)$ to be invertible. In contrast, the regularized MLE confidence region results in \citet{perivier2022dynamic,lee2024nearly} remove this invertibility requirement of $\bH$ by setting $\lambda$ as an absolute constant. However, this approach introduces an extra dependency on $\sqrt{d}$. 

Corollary~\ref{thm-sup-ucb-confidence-bound} shows that there exists a range of $\lambda$, given by $\sqrt{\lambda} \lesssim 1/W$, under which a sharp dependency on $d$ can still be preserved, provided that $\max_{m\leq t, j\in S_m} \lVert   \bm x_{mj} \rVert_{\bH_{\cD}(\bm\theta^\star)^{\dagger}} \lesssim 1/\sqrt{d\log(N/\delta)}$, 
%\yf{If $\bm x\in \text{col}(\bH)$, $\lim_{\lambda\to 0}\bx^\top(\bH+\lambda\bI)^{-1}\bx=\bx^\top \bH^\dagger \bx $, i.e., $\lim_{\lambda\to 0}\|\bx\|_{\bH^\lambda_\cD(\bm\theta)^{-1}} =\|\bx\|_{\bH_\cD(\bm\theta)^{\dagger}}$. But it is hard to verify if $\bx_{mj}\in \text{col}(\bH^\lambda_\cD) = \text{span}\{\bx_{mj}-\sum_{i\in S_m}q_{mj}\bx_{mj}: m\le t, j\in S_m\}$}
without requiring $\bH_{\cD}(\bm\theta^\star)$ to be invertible. This extension is crucial to achieve optimal rates in the MNL online setting with a fixed design, particularly when the eigenvalue assumptions assumed in \citet{li2017provably,oh2021multinomial,jun2021improved} do not hold.

%---------------------------------------------------------%
%---------------------------------------------------------%
\paragraph{Comparison to previous works.} Prior to our result, the only confidence interval for the linear MNL MLE achieving a $\sqrt{d}$-free rate is stated in \citet{oh2021multinomial}. Their confidence region result scales as $\tilde{O}(\kappa^{-1} \lVert \bm x \rVert_{\bV_\cD^{-1}})$ under the burn-in condition 
\begin{equation*}
\lambda_{\min}(\bV_\cD) \geq \kappa^{-4} d^2,\quad \bV_\cD = \sum_{m\leq t} \sum_{j \in S_m} \bm x_{mj}\bm x_{mj}^\top.
\end{equation*}
By the relation $\bH_\cD(\bm\theta^\star) \succeq \kappa \bV_\cD$ and $\max_{m\leq t, j\in S_m} \lVert  \bm x_{mj} \rVert_{\bH_\cD(\bm\theta^\star)^{-1}}^{2} \leq \big(\kappa \lambda_{\min}(\bV_\cD)\big)^{-1}$, our result provides an improvement over \citet{oh2021multinomial} in both the confidence interval and the burn-in condition. 

Besides \citet{oh2021multinomial}, a line of recent works have focused on deriving confidence regions for the MLE that are independent of $\kappa$ \citep{perivier2022dynamic, agrawal2023tractable, lee2024nearly}. It is worth mentioning that results in \citet{perivier2022dynamic,agrawal2023tractable,lee2024nearly} no longer require the burn-in condition or the conditional independence assumption, with a price of introducing an additional factor of $\sqrt{d}$ in the resulting confidence interval. 
The requirement of a burn-in condition to achieve sharp dependency on $d$ first appears in the logistic linear bandit literature, where such a condition arises as \citet{jun2021improved, li2017provably} refine the results of \citet{faury2020improved, abeille2021instance, faury2022jointly} by improving the dependency on a $\sqrt{d}$ factor. This corresponds to a special case of our setting with $K = 1$.

%% ---- END 3-Confidence Bound.tex ----

%---------------------Linear MNL-----------------------------%
%--------------Offline---------------------%
%% ---- BEGIN 4-Linear MNL Offline.tex ----
%------------------- Linear MNL Offline ----------------------------%
\section{Offline Assortment and Pricing Optimization}\label{sec-offline}

In this section, we study the offline joint assortment and pricing problem under the price-based MNL model. Using the confidence region in Theorem~\ref{prop-sup-ucb-confidence-bound}, we propose a pessimistic plug-in algorithm and establish its sub-optimality guarantee. We then discuss several consequences of the main result, including the fixe design setting and the linear MNL model as a special case.

Throughout this section, we assume W.L.O.G. that $\bH_{\widetilde\cD}(\bm\vartheta^\star)$ is invertible and use the un-regularized MLE to maintain notational clarity; otherwise, we consider the $\lambda$-regularized MLE and its corresponding confidence region for $\lambda$ very close to $0$, which does not affect the theoretical results.

\subsection{The LCB-MNL-Pricing Algorithm}

We consider the offline joint assortment and pricing problem under the price-based demand model.
Given an offline dataset $\widetilde{\cD}:=\{(i_m,\widetilde{\bm{X}}_m(\bm p_m), S_m)\}_{m = 1}^n$ and a target feature matrix $\Xoff:=((\xoff)_1,\dots,(\xoff)_N)$, our goal is to output an assortment-price pair $(S,\bm p)$ that maximizes the expected MNL revenue under the target context. To achieve this, we propose the LCB-MNL-Pricing algorithm, as detailed in Algorithm~\ref{alg:pricing-mnl-lcb}.
Our algorithm is inspired by the lower-confidence-bound (LCB) technique, also known as the pessimistic principle, which is widely used in offline policy learning \citep{jin2021pessimism, rashidinejad2021bridging, zhong2022pessimistic}. The LCB-MNL-Pricing algorithm consists of two steps, the \textit{pessimistic estimation step} and the \textit{revenue maximization step}. 

\begin{algorithm}[h]
\caption{LCB-MNL-Pricing}
\label{alg:pricing-mnl-lcb}
\begin{algorithmic}[1]
\STATE \textbf{Input:} Dataset $\widetilde\cD:=\{(i_m,\widetilde{\bm{X}}_m(\bm p_m), S_m)\}_{m = 1}^n$, target feature matrix $\Xoff$.
\STATE Compute the MLE $\hat{\bm\vartheta}_{\widetilde \cD }= \text{argmax}_{\bm\vartheta} \ell_{\widetilde\cD}(\bm\vartheta)$. 
\STATE Define the pessimistic utility functions $u_i^{\mathrm{LCB}}(\cdot)$ for each item $i$ as in \eqref{eq-price-u-LCB}.
\STATE Select $(S^\text{LCB},\bm p^\text{LCB})$ as an arbitrary maximizer of pessimistic revenue,
$$(S^{\text{LCB}},\bm p^\text{LCB}) \in \operatorname*{argmax}_{S \in \mathcal{S}_K,\bm p\in \cP^N} R^\text{LCB}(S,\bm p):= \frac{\sum_{i\in S}p_i\exp(u_i^\text{LCB}(p_i))}{1+\sum_{i\in S}\exp(u_i^\text{LCB}(p_i))}.$$
\STATE \textbf{Return:} $(S^{\text{LCB}},\bm p^\text{LCB})$
\end{algorithmic}
\end{algorithm}

\paragraph{Pessimistic Estimation.} 
In the pessimistic estimation step, with the MLE $\hat{\bm\vartheta}_{\widetilde \cD}$, 
the algorithm first defines the pessimistic utility function with respect to $p$ as:
\begin{equation}
\begin{aligned}\label{eq-price-u-LCB}
  u_i^{\mathrm{LCB}}(p) :=   (\widetilde{\xoff})_i(p)^\top \hat{\bm\vartheta}_{\widetilde \cD}  - 16\sqrt3\lVert (\widetilde{\xoff})_i(p)\rVert_{\bH_{\widetilde \cD}(\hat{\bm\vartheta}_{\widetilde \cD})^{-1}}\sqrt{\log(N/\delta)},
\end{aligned}
\end{equation}
where $(\widetilde{\xoff})_i(p)=((\bx_{\text{off}})_i,-p(\bx_{\text{off}})_i)$ is the target price-dependent feature of item $i$ at price $p$.

According to Theorem~\ref{prop-sup-ucb-confidence-bound}, with probability at least $1-\delta$, the Hessian comparison
$\bH_{\widetilde \cD}(\hat{\bm\vartheta}_{\widetilde \cD})^{-1} \preceq 3\,\bH_{\widetilde \cD}(\bm\vartheta^\star)^{-1}$
holds. Moreover, this guarantees that the following lower bound holds uniformly over all items $i\in[N]$ and prices $p_i\in\cP$:
\[
u_i^{\mathrm{LCB}}(p_i)\ \le\ (\widetilde{\xoff})_i(p_i)^\top \bm\vartheta^\star.
\]

\paragraph{Revenue Maximization.} 
Given the pessimistic utilities $\{u_i^{\mathrm{LCB}}(\cdot)\}_{i\in[N]}$, the algorithm selects the assortment-price pair $(S,\bm p)$ that maximizes the expected MNL revenue over $\cS_K\times\cP^N$. In general, this optimization step may not be computationally efficient; developing tractable algorithms with the same regret guarantee is left for future work.

\subsection{Sub-optimality Gap Guarantee}\label{sec4.2}
Since both $\Xoff$ and $\bm{\vartheta}^\star$ are fixed throughout this section, we simplify the notation by defining
\begin{align*}
    q_i(S,\bm p|\bm\vartheta):=q_i(S\lvert \exp(\widetilde{\bX}_{\text{off}}(\bm p)^\top \bm\vartheta)) \quad \text{and} \quad (S^\star,\bm p^\star):= \operatorname*{argmax}_{S\in \cS_K,\bm p\in\cP^N} R(S\lvert \exp(\widetilde{\bX}_{\text{off}}(\bm p)^\top \bm\vartheta^\star)).
\end{align*}
We now establish the sub-optimality gap guarantee for Algorithm~\ref{alg:pricing-mnl-lcb}.

\begin{theorem}\label{thm-price-sub-optimality}
   Under Assumption~\ref{assump-bounded-pricing} and~\ref{assump-mim-price}, and suppose the burn-in condition 
   \begin{align}\label{eq:burn-in-thm4.1}
       144\max_{m\le n, j\in S_m} \lVert \widetilde{\bm x}_{mj}(p_{mj})\rVert_{{\bH}_{\widetilde\cD}(\bm\vartheta^\star)^{-1}} \leq \big(2d\log (N/\delta)\big)^{-1/2}
   \end{align}
holds, then with probability at least $1-\delta,$ we have 
\begin{align*}
       & \SubOpt(S^\mathrm{LCB},\bm p^\mathrm{LCB}, \Xoff) 
        \\=&\ \widetilde{\mathcal{O}}\left(\frac{W}{L_0}\sqrt{\sum_{j\in S^\star}q_j(S^\star, \bm p^\star|\bm\vartheta^\star)q_0(S^\star, \bm p^\star|\bm\vartheta^\star)
        \| (\widetilde{\xoff})_j( p^\star_j)\|^2_{\bH_{\widetilde\cD}(\bm\vartheta^\star)^{-1}}} + \frac{W}{L_0}  \cdot \max_{j\in S^\star} \| (\widetilde{\xoff})_j( p^\star_j)\|^2_{\bH_{\widetilde\cD}(\bm\vartheta^\star)^{-1}}\right).
\end{align*}

\end{theorem}

Theorem~\ref{thm-price-sub-optimality} shows that the offline suboptimality gap is governed by a local information-weighted version of item-level coverage quantity at the optimal assortment-price pair. 
This contrasts with prior offline results in two respects. \citet{dong2023pasta} requires assortment-level coverage, captured by the number of times the exact assortment $S^\star$ appears in the data, which is a strictly stronger requirement than item-level coverage. \citet{han2025optimal} shows that item-level coverage is both sufficient and necessary in the assortment-only setting. Our result extends this item-level perspective to joint assortment--pricing: the relevant coverage notion is no longer a raw count but an information-geometric quantity that depends on the optimal prices through $(\widetilde{\xoff})_j( p^\star_j)$ and the data distribution through $\bH_{\widetilde\cD}(\bm\vartheta^\star)^{-1}$.

%-----------------------------------No Burn-in---------------------------%

\paragraph{Eliminating the Burn-in Condition.} The burn-in condition \eqref{eq:burn-in-thm4.1} in Theorem~\ref{thm-price-sub-optimality} arises from the condition required by our confidence region guarantee in Theorem~\ref{prop-sup-ucb-confidence-bound}. Meanwhile, the item-wise pessimistic estimation scheme in Algorithm~\ref{alg:pricing-mnl-lcb} is flexible and can be adapted to other plug-in confidence region results by replacing ~\eqref{eq-price-u-LCB} with alternative confidence region-based bounds on $(\widetilde{\xoff})_i(p)^\top\bm\vartheta^\star$. In particular, integrating confidence region results for adaptively collected data—such as those in \citet{oh2021multinomial,perivier2022dynamic}—enables a burn-in-free offline learning guarantee at the cost of an additional $\sqrt{dK\bar P^4W^3}$ multiplicative dependency. We summarize the main result below and defer the detailed plug-in confidence region analysis and proof to Appendix~\ref{sec-appendix-offline-burn-in-free}.
\begin{proposition}\label{prop-plug-in-confidence-region}
    There exists a plug-in confidence region result such that, after replacing the pessimistic estimation step in~\eqref{eq-price-u-LCB} by this result, the output of Algorithm~\ref{alg:pricing-mnl-lcb} satisfies, with probability at least $1-\delta$,
    \begin{alignat*}{2}
        &\SubOpt(S^\mathrm{LCB},\bm p^\mathrm{LCB},\Xoff) \\
        &= \widetilde{\mathcal{O}}\!\left(
        \sqrt{\frac{dKW^{8}}{L^5_0}\sum_{j\in S^\star}q_j^\star q_0^\star
        \| (\widetilde{\xoff})_j( p^\star_j)\|^2_{\bH^\lambda_{\widetilde\cD}(\bm\vartheta^\star)^{-1}}}
        + \frac{dKW^7}{L^4_0} \max_{j\in S^\star} \| (\widetilde{\xoff})_j( p^\star_j)\|^2_{\bH^\lambda_{\widetilde\cD}(\bm\vartheta^\star)^{-1} }\right),
    \end{alignat*}
where $\lambda=2d\bar P/W$ and $q_j^\star := q_j(S^\star, \bm p^\star|\bm\vartheta^\star)$, $q_0^\star := q_0(S^\star, \bm p^\star|\bm\vartheta^\star)$.
\end{proposition}

%-----------------------------------Fixed Data Collecting---------------------------%
\paragraph{Results under Fixed Data Collecting Policy.}
In the fixed design setting, i.e., $\bX_m \equiv \Xoff:=\bX$, suppose the offline data are collected i.i.d. from a fixed exploration policy
$\pi_{\mathrm{off}}$ over assortment-price pairs, i.e., $\left(S_m, \boldsymbol{p}_m\right) \stackrel{\text { i.i.d. }}{\sim} \pi_{\text {off }}$. Define the pricing Hessian limit
\[
{\bH}_{\mathrm{off}}
:=\E_{(S,\bm p)\sim \pi_{\mathrm{off}}}\!\left[\sum_{i\in S} q_i(S,\bm p|\bm\vartheta^\star)q_0(S,\bm p|\bm\vartheta^\star)
\widetilde{\bx}_i(p_i)\widetilde{\bx}_i(p_i)^\top\right].
\]
Then $\frac{1}{n}\bH_{\widetilde\cD}(\bm\vartheta^\star)\to {\bH}_{\mathrm{off}}$, and the burn-in condition in Theorem~\ref{thm-price-sub-optimality} holds for all sufficiently large $n$. Consequently, Theorem~\ref{thm-price-sub-optimality} implies the large-sample asymptotic rate
\begin{align}\label{eq-fixed data collecting}
\SubOpt(S^\mathrm{LCB},\bm p^\mathrm{LCB},\bX)
=\widetilde{\mathcal{O}}\left(
\frac{W}{L_0}\sqrt{\frac{\sum_{j\in S^\star} q_j(S^\star,\bm p^\star|\bm\vartheta^\star)q_0(S^\star,\bm p^\star|\bm\vartheta^\star)
\big\|\widetilde{\bx}_j(p_j^\star)\big\|^2_{{\bH}_{\mathrm{off}}^{-1}}}{n}}
\right),
\text{as }n\to\infty.
\end{align}

%--------------------------------------------------------------%
%--------------------------------------------------------------%
\subsection{Implications for the Linear MNL Model.}
In this subsection, we show how our price-based results naturally specialize to the classic linear MNL model. We also derive more explicit guarantees for the fixed-design setting within the linear MNL framework.

As discussed in Section~\ref{sec3.2:linear-confidence bound}, the price-based model reduces to the linear contextual MNL model when the feature map is price-independent, or equivalently, when prices are fixed at $\bm p\equiv\bm r$. Under this specialization, Algorithm~\ref{alg:pricing-mnl-lcb} reduces to an item-wise LCB algorithm for linear MNL assortment selection. Specifically, in the pessimistic estimation step, the utility $$u_i^{\mathrm{LCB}}=(\xoff)_i^\top \hat{\bm\theta}_\cD  - 16\sqrt3\lVert \bm (\xoff)_i\rVert_{\bH_\cD(\hat{\bm\theta}_\cD)^{-1}}\sqrt{\log(N/\delta)}$$
is a fixed number for each item rather than a function of $p$, and in the revenue maximization step the algorithm simply chooses the assortment $S\in\cS_K$ that maximizes the MNL revenue under these pessimistic utilities. This step can be efficiently solved using several well-studied polynomial-time algorithms \citep{rusmevichientong2010dynamic, davis2013assortment, avadhanula2016tightness}. 

We then obtain the following suboptimality guarantee as a specialization of Theorem~\ref{thm-price-sub-optimality}, by replacing $(\widetilde{\xoff})_j(p_j^\star)$ with $(\xoff)_j$ and replacing the price upper bound $W/L_0$ with an upper bound $M$ on the revenue.

\begin{corollary}\label{thm-linear-mnl-sub-optimality}
    Under Assumption~\ref{assump-bounded}, and suppose the burn-in condition $$144\max_{m\le n, j\in S_m} \lVert \bm x_{mj}\rVert_{{\bH}_\cD(\bm\theta^\star)^{-1}} \leq \big(d\log (N/\delta)\big)^{-1/2}$$ 
holds, then with probability at least $1-\delta,$ we have 
\begin{align*}
        \SubOpt(S^\mathrm{LCB}, \Xoff) 
        &= \widetilde{\mathcal{O}} \left( M\sqrt{\sum_{j\in S^\star}q_j(S^\star|\bm\theta^\star)q_0(S^\star|\bm\theta^\star)\lVert (\xoff)_j \rVert^2_{\bH_\cD(\bm\theta^\star)^{-1}} }+ M\cdot \max_{j\in S^\star}\lVert (\xoff)_j \rVert_{\bH_\cD(\bm\theta^\star)^{-1}}^2 \right).
\end{align*}
\end{corollary}

%---------------------------------------fixed design-----------------------%
\subsubsection{Guarantees for the Fixed Design Setting.} 
In the {fixed design} setting, namely when $\bX_m \equiv \Xoff:=\bX$, the leading term in Corollary~\ref{thm-linear-mnl-sub-optimality} admits a more interpretable upper bound in terms of item-wise coverage counts. Specifically, define the coverage number $n_i:=\sum_{m\le n}\bm 1\{i\in S_m\}$ and $n^\star:=\min_{i\in S^\star} n_i$. Let $v_i:=\exp\!\left({\bx}_i^\top \bm\vartheta^\star\right)$ denote the true MNL value of item $i$. The following proposition upper bounds the effective uncertainty on the optimal assortment by $n^\star$ and a mild density-ratio term.

\begin{proposition}\label{prop-leading-lcb-upper-bound}
With the notation above, it holds that 
\begin{align*}
        \sqrt{\sum_{j\in S^\star} q_j(S^\star|\bm\theta^\star)q_0(S^\star|\bm\theta^\star) \lVert \bx_j \rVert_{\bH_\cD(\bm\theta^\star)^{-1}}^2}\leq \sqrt{\frac{K}{n^\star}} \cdot \max_{ m\le n}\frac{1+\sum_{k\in S_m}v_k}{1+\sum_{k\in S^\star}v_k} 
    \end{align*}
\end{proposition}

Combining Proposition~\ref{prop-leading-lcb-upper-bound} with Corollary~\ref{thm-linear-mnl-sub-optimality} yields an item-coverage based sub-optimality guarantee for Algorithm~\ref{alg:pricing-mnl-lcb}.

\begin{corollary}\label{corollary-subopt-via-itemcover}
Under the same conditions as Theorem~\ref{thm-price-sub-optimality}, with probability at least $1-\delta$, 
\begin{align*}
    \SubOpt(S^\mathrm{LCB},\bX) = \widetilde{\mathcal{O}} \left(M\left( \sqrt{\frac{K}{n^\star}} \cdot \max_{ m\le n}\frac{1+\sum_{k\in S_m}v_k}{1+\sum_{k\in S^\star}v_k} + \frac{1}{n^\star\kappa_{\mathrm{off}}} \right)\right)
\end{align*}
\end{corollary}

The bounds in Proposition~\ref{prop-leading-lcb-upper-bound} and Corollary~\ref{corollary-subopt-via-itemcover} are strictly looser than those in Corollary~\ref{thm-linear-mnl-sub-optimality} in the general scenario, as they only consider assortments containing the $j$-th item when bounding $\|\bx_j\|_{\bH_{\cD}(\bm\theta^\star)^{-1}}$. By contrast, Corollary~\ref{thm-linear-mnl-sub-optimality} incorporates all assortments with non-orthogonal feature vectors.

%-----------------------------------Canonical---------------------------%
\paragraph{Results under Canonical $N$-item Setting.} 
In the \textit{canonical $N$-item MNL setting}, a special case of the linear MNL setting with $d = N$ and $\bx_{mj} =(\xoff) _j = \be_j$ for all $m$, the fixed-design specialization (Corollary~\ref{corollary-subopt-via-itemcover}) recovers an item-coverage type guarantee closely related to the recent result of \citet{han2025optimal}. Their algorithm design shares the same spirit as ours in employing an item-wise pessimistic principle. However, their algorithm and analysis rely heavily on the rank-breaking technique \citep{saha2024stop,saha2019active,khetan2016data}, which is challenging to extend to the general linear MNL setting. More specifically, they propose a $\tilde{\Theta}(K/\sqrt{n^\star})$ complexity result and a $\tilde{\Theta}(\sqrt{K/n^\star})$ complexity result in the uniform reward setting ($r_i \equiv 1$). 
In particular, for the general non-uniform reward setting, Proposition~\ref{prop-leading-lcb-upper-bound} implies an upper bound of $\tilde{O}(K^{3/2}/\sqrt{n^\star})$. In the uniform reward setting, since $S^\star$ consists of the top-$K$ items by value, the inequality in Proposition~\ref{prop-leading-lcb-upper-bound} can be further simplified to
$
\sqrt{\sum_{j \in S^\star} q_j(S^\star|\bm\theta^\star)q_0(S^\star|\bm\theta^\star) \lVert (\xoff)_j \rVert_{\bH_\cD(\bm\theta^\star)^{-1}}^2} \lesssim \sqrt{K/n^\star}.
$
These results, together with the lower bound established in \citet{han2025optimal}, suggest a corresponding lower bound of
$
\Omega\left(M\sqrt{\sum_{j \in S^\star} q_j(S^\star|\bm\theta^\star)q_0(S^\star|\bm\theta^\star) \lVert (\xoff)_j\rVert_{\bH_\cD(\bm\theta^\star)^{-1}}^2}\right)
$
for the offline linear MNL setting for both the uniform and non-uniform reward setting.

\begin{remark}[Lower bound for the offline joint assortment--pricing problem]
A lower bound for the joint problem follows by reduction to the fixed-price subclass $p_{ti}\equiv p^\star_i$ for all $t$, under which the problem reduces to the offline linear MNL assortment problem with augmented features $(\xoff)_j = ((\xoff)_j,-p^\star_j (\xoff)_j)$ and $M = \max_{i\in [N]}\sqrt{1+(p^\star_j)^2}$. Since Lemma~\ref{assump-mim-price} gives $p^\star_i \le P := (3+W+\log K)/L_0$, the lower bound becomes $\Omega\!\left(\frac{W}{L_0}\sqrt{\sum_{j\in S^\star} q_j(S^\star,\bm p^\star|\bm\vartheta^\star)q_0(S^\star,\bm p^\star|\bm\vartheta^\star)\,\|(\widetilde{\xoff})_j(p_j^\star)\|^2_{\bH_{\widetilde\cD}(\bm\vartheta^\star)^{-1}}}\right).$
\end{remark}

%-----------------------------------Fixed Data Collecting---------------------------%
\paragraph{Results under Fixed Data Collecting Policy.} In the fixed design setting, where $\{S_m\}_{m=1}^n$ is sampled i.i.d. from some fixed exploration policy $\pi_\text{off}$, the corresponding asymptotic rate becomes
\[
\SubOpt(S^\text{LCB},\Xoff)
=
\widetilde{\mathcal O}\left(M
\sqrt{
\frac{
\sum_{j\in S^\star}
q_j(S^\star|\bm\theta^\star)\,
q_0(S^\star|\bm\theta^\star)\,
\|\bx_j\|^2_{\bH_{\mathrm{off}}^{-1}}
}{n}}
\right).
\]
This should be compared with the result of \citet{dong2023pasta}, whose sample complexity scales as $\tilde{O}\left(\sqrt{\frac{d/K}{\kappa n \pi_{\text{off}}(S^\star)}}\right)$ in our notation, which scales with the frequency that optimal assortment $S^\star$ sampled from $\pi_{\text{off}}$. In comparison, our sample complexity result does not scale multiplicatively with $d,\kappa^{-1}$ in the leading-order term and, most importantly, does not rely on such a strict coverage condition. More precisely, it allows the items in $S^\star$ to be well-explored through assortments that only cover items with feature vectors that are non-orthogonal to them.

%% ---- END 4-Linear MNL Offline.tex ----

%--------------Online---------------------%
%% ---- BEGIN 5.1 MNL Pricing Online.tex ----
\section{Online Assortment and Pricing Optimization}\label{sec-online}

% \yf{For adversary setting, we have MLE-based UCB algorithm.}

In this section, we study the online joint assortment and pricing problem under a \emph{fixed design} setting, where the context is unchanged across time, i.e., $\bm X_t \equiv \bm X$ for all $t\in[T]$. We first present the SupCB-MNL-Pricing algorithm, which is built on the confidence interval in Theorem~\ref{prop-sup-ucb-confidence-bound}, and establish its regret guarantee. We then describe a Thompson Sampling alternative that is computationally simpler to implement, and finally show that the linear MNL model follows as a special case of the pricing setting.

Throughout this section, since the feature matrix $\bX$ is fixed, we simply write
\[
R(S,\bm p\lvert \bm\vartheta):=R(S\lvert \exp(\widetilde{\bm{X}}(\bm p)^\top \bm\vartheta)),
\qquad
q_{i}(S,\bm p\lvert \bm\vartheta):=q_i(S\lvert \exp(\widetilde{\bm{X}}(\bm p)^\top \bm\vartheta))\]
for the expected revenue and the choice probability of item $i$ under $(S,\bm p)$.

%--------------------------------subsec-----------------------------------------%
\subsection{The SupCB-MNL-Pricing Algorithm}

Although Theorem~\ref{prop-sup-ucb-confidence-bound} provides a tighter confidence radius, it relies on an additional conditional independence assumption on the collected data. 
This assumption is typically violated in standard adaptive algorithms, such as UCB-type methods, as discussed in \cite{han2020sequential}.
To ensure conditional independence while still allowing adaptive decision-making, we propose the SupCB-MNL-Pricing algorithm (Algorithm~\ref{alg:sup-lin-mnl-pricing}).
Our design builds on the SupCB framework of \cite{auer2002using}, which combines action elimination with sample splitting to ensure that each confidence update is based on data that are independent of the current action choice, as in \citet{chu2011contextual,li2017provably,jun2021improved,blanchet2024delay,oh2021multinomial}. We further introduce several modifications tailored to the pricing setting to ensure the burn-in condition required by Theorem~\ref{prop-sup-ucb-confidence-bound}, and exploit the first-order structure of the MNL revenue function, as described below.

%------------------------------------------------------%
\paragraph{The SupCB Framework.} Similar to the standard framework in \citet{auer2002using}, Algorithm~\ref{alg:sup-lin-mnl-pricing} divides the collected samples into $J$ bins, denoted as $\Psi_1, \dots, \Psi_{J}$. To maintain independence while accounting for the first-order geometry, we add an additional bin, $\Psi_{J+1}$, as in \citet{jun2021improved}, which contains a rough estimator to approximate the first-order coefficients of the revenue function. 
After bypassing an initial $\tau$-length pure-exploration period designed to ensure that each bin collects a sufficient number of samples, as explained in the next paragraph, the algorithm enters an adaptive elimination phase conducted through a multi-layer procedure that loops over the first $J$ bins. 

In the following context, we denote $\widetilde\cD_{t,\ell}$ as the samples collected in bin $\ell$ up to time $t$, moreover, we use the shorthand notations $\hat{\bm\vartheta}^\lambda_{t,\ell}:= \hat{\bm{\vartheta}}^\lambda_{\widetilde\cD_{t,\ell}}, \bH_{t,\ell}^\lambda(\bm\vartheta):= \bH_{\widetilde\cD_{t,\ell}}^\lambda(\bm\vartheta), \kappa:= \tilde\kappa_{\mathrm{on}}$. We always use Theorem~\ref{prop-sup-ucb-confidence-bound} with $\delta$ selected as $ 1/T$ without additional clarification.

%-----------------------Algorithm------------------------------%
\begin{algorithm}[h]
\caption{SupCB-MNL-Pricing}
\label{alg:sup-lin-mnl-pricing}
\begin{algorithmic}[1]
\STATE \textbf{Input:} Time horizon $T,$ fixed item feature matrix $\bX$, regularized parameter $\lambda\ge \bar P$.

\STATE \textbf{Initialize} $\tau = 1, J =\lceil \frac{1}{2}\log_2 T \rceil, \Psi_0= \dots = \Psi_{J+1} = \emptyset.$ 
\WHILE{\eqref{eq-explore-criteria} is not satisfied for some $\ell \in [J+1]$}
\STATE Select the assortment-price pair $(S_\tau, \bm{p}_\tau)=\left\{\operatorname{argmax}_{j\in[N],p\in\mathcal{P}}\lVert \widetilde{\bx}_{j}(p) \rVert_{\left(\kappa\bV_{\tau,\ell}+\lambda\bI\right)^{-1}}\right\}.$
\STATE $\Psi_{\ell} \leftarrow \Psi_{\ell} \cup \{\tau\}$, $\tau \leftarrow \tau +1$.
\ENDWHILE

\STATE Compute~$\widehat{\bm \vartheta}_0$ based on samples in $\widetilde{\cD}_{\tau,J+1}$ as in \eqref{eq-vartheta0}.
\FOR{$t = \tau + 1,\dots, T$}
    \STATE Set  $S_t = \emptyset, \ell = 1$, and 
    $\mathcal{A}_1 = \mathcal{S}_K \times \mathcal{P}^N.$
    \WHILE{$S_t = \emptyset$}
    \STATE Compute $W_{t,\ell}(S,\bm p),R^\mathrm{UCB}_{t,\ell}(S,\bm p),\forall (S,\bm p)\in \mathcal{A}_\ell$ as in \eqref{eq-assortment-uncertain}, ~\eqref{eq-ucb-in-supCB}.

    \IF{$W_{t,\ell}(S,\bm p) > 2^{-\ell}$ for some $(S, \bm p) \in \mathcal{A}_\ell$}
    \STATE Select an assortment-price pair $(S_t,\bm p_t)\in \{(S,\bm p)\in \mathcal{A}_\ell : W_{t,\ell}(S,\bm p) > 2^{-\ell}\}$.     
    \STATE $\Psi_{\ell} \leftarrow \Psi_\ell \cup \{t\}$.
    
    \ELSIF{$W_{t,\ell}(S,\bm p)\leq 1/\sqrt{T}$ for all $(S,\bm p)\in \mathcal{A}_\ell$ }
    \STATE Select an assortment-price pair $(S_t, \bm p_t) \in \text{argmax}_{(S, \bm p)\in \mathcal{A}_{\ell}}R^\mathrm{UCB}_{t,\ell}(S, \bm p)$.
    \STATE $\Psi_0 \leftarrow \Psi_0 \cup \{t\}$.
    
    \ELSE
    \STATE $\widehat{R} \leftarrow \max_{(S,\bm p)\in \mathcal{A}_\ell}  R^\mathrm{UCB}_{t,\ell}(S, \bm p)$.
    \STATE $\mathcal{A}_{\ell+1} \leftarrow \left\{ (S, \bm p)\in \mathcal{A}_\ell, R^\mathrm{UCB}_{t,\ell}(S, \bm p)\geq \widehat{R} - 2^{-\ell }\right\}$.
    \STATE $\ell \leftarrow \ell + 1$.
    \ENDIF
  \ENDWHILE
  \ENDFOR
\end{algorithmic}
\end{algorithm}

%------------------------------------------------------%
\paragraph{Initial Exploration Phase.}
The initial exploration phase (line~2--6) is used to ensure the burn-in condition in Proposition~\ref{prop-sup-ucb-confidence-bound}. Concretely,  this phase ensures that $\max_{j\in[N],\,p\in[0,P]}\big\|\widetilde{\bx}_j(p)\big\|_{\bH^\lambda_{\tau,\ell}(\bm\vartheta^\star)^{-1}}$ is well controlled for every $\ell\in[J+1]$ with high probability.
Because the Hessian satisfies the relation
\[
\bH_{\tau,\ell}^\lambda(\bm\vartheta^\star)\ \succeq\ \kappa\,\bV_{\tau,\ell}+\lambda\bI,
\quad \text{with }
\bV_{\tau,\ell}:=\sum_{s\le \tau, s\in\Psi_\ell}\ \sum_{j\in S_s}\widetilde{\bx}_j(p_{sj})\widetilde{\bx}_j(p_{sj})^\top,
\]
it suffices to control the same quantity with $(\kappa\bV_{\tau,\ell}+\lambda\bI)^{-1}$ in place of
$\bH_{\tau,\ell}^\lambda(\bm\vartheta^\star)^{-1}$. Specifically, we require
\begin{equation}\label{eq-explore-criteria}
\max_{j\in[N],\,p \in [0,P]}
\big\|\widetilde{\bx}_j(p)\big\|_{(\kappa \bV_{\tau,\ell} + \lambda \bI)^{-1}}
\leq \frac{1}{144\sqrt{2d\log (NT)}} \ \wedge\  \frac{1}{24\sqrt \lambda W},
\end{equation}
which implies that the burn-in condition~\eqref{eq-burn-in} in
Proposition~\ref{prop-sup-ucb-confidence-bound} holds.

Based on criterion~\eqref{eq-explore-criteria}, Algorithm~\ref{alg:sup-lin-mnl-pricing} performs exploration within each bin: whenever the criterion is violated for some $\ell$, the algorithm keeps sampling actions that include
uncertain assortment-price pairs (i.e., those with large norm under $(\kappa\bV_{\tau,\ell}+\lambda\bI)^{-1}$) to increase
$\bV_{\tau,\ell}$. Once~\eqref{eq-explore-criteria} is satisfied for all $\ell\in[J+1]$, the algorithm proceeds to the exploitation
step.

Finally, we bound the maximum length of this initial exploration phase in the following lemma.

\begin{lemma}\label{lem-initial-length-bound}
For $\lambda\ge \bar P$, there exists an absolute constant $C>0$ such that after at most
\[
\tau:=C\,\kappa^{-1}d\Big(d\log(NT)\vee\lambda W^2\Big)\log\!\left(1+\frac{\kappa T\bar P^2}{2d\lambda}\right)
\]
rounds, the criterion~\eqref{eq-explore-criteria} holds for all $\ell\in[J+1]$.
\end{lemma}

%------------------------------------------------------%
\paragraph{Adaptive Elimination Phase.} 
Once the initial exploration phase ends, all bins $\Psi_\ell=[\tau]$ are identical, and the algorithm stops allocating samples to $\Psi_{J+1}$. We use the data in $\Psi_{J+1}$ up to time $\tau$ to compute the MLE 
\begin{equation}\label{eq-vartheta0}
\hat{\bm\vartheta}_{0}:= \operatorname*{argmax}_{\bm\vartheta\in\mathbb{R}^{2d}} \ell_{\widetilde{\cD}_{\tau,J+1}}(\bm\vartheta).
\end{equation}

Next, the algorithm enters an adaptive elimination phase that iterates over the bins $\Psi_1, \dots, \Psi_J$, based on the \textit{uncertainty level} $W_{t,\ell}(S,\bm p)$ computed for each $(S, \bm{p})$ pair. To define the uncertainty level, we first introduce an item-wise uncertainty $w_{ti}^\ell(p_i)$ for each item $i$ at price $p_i$ in bin $\ell$:
\begin{align}\label{eq-item-uncertain}
    w^\ell_{ti}(p_i) := 16\sqrt{3} \lVert \widetilde{\bx}_i(p_i) \rVert_{H_{t,\ell}^\lambda(\hat{\bm\vartheta}_{0})^{-1}} \left(\sqrt{\log(NT)} +\sqrt{\lambda} W \right).
\end{align}
Based on~\eqref{eq-item-uncertain}, the assortment-level uncertainty for a given pair $(S, \bm{p})$ in bin $\ell$ is defined as:
\begin{equation}\label{eq-assortment-uncertain}
    W_{t,\ell}(S,\bm p) := 4e^2P \sqrt{2\sum_{i \in S} q_i(S, \bm p\lvert \hat{\bm \vartheta}_0)q_0(S, \bm p\lvert \hat{\bm \vartheta}_0) \left(w^\ell_{ti}(p_i)\right)^2} + 20P \max_{i \in S} \left(w^\ell_{ti}(p_i)\right)^2.
\end{equation}
Once we have the uncertainty $W_{t,\ell}(S,\bm{p})$ for all candidate pairs in our set $\mathcal{A}_\ell$, the algorithm proceeds as follows:

%------------------------------------------------------%
\paragraph{Step~(a): Explore Uncertain Items.} 
If there exists some uncertain assortment-price pair (i.e., $W_{t,\ell}(S,\bm{p}) > 2^{-\ell}$), the algorithm selects such $(S,\bm{p})$ to gather more data.

%------------------------------------------------------%
\paragraph{Step~(b): Output the UCB Assortment.} If all assortment-price pairs are sufficiently certain (i.e., $W_{t,\ell}(S,\bm p)\leq 1/\sqrt{T}$ for all $(S, \bm{p}) \in \mathcal{A}_\ell$), the algorithm outputs the pair with the highest UCB revenue, computed using the optimistic utility function:
\begin{align}\label{eq-ucb-in-supCB}
   R^{\mathrm{UCB}}_{t,\ell}(S,\bm p):= \frac{\sum_{i\in S} p_i\exp(u_{ti}^\ell(p_i))}{1+\sum_{i\in S}\exp(u_{ti}^\ell(p_i))}, \quad \text{ with }  u^\ell_{ti}(p_i) := \widetilde{\bx}_i(p_i)^\top \hat{\bm\vartheta}^\lambda_{t,\ell}  +w_{ti}^\ell(p_i).
\end{align}

%------------------------------------------------------%
\paragraph{Step~(c): Elimination.} 
Otherwise, the algorithm eliminates suboptimal assortment-price pairs in $\mathcal{A}_\ell$ based on their optimistic revenue estimates. First, it finds the highest possible revenue estimate $\widehat{R}$ among the current candidates. Then, it only keeps the pairs that can reach at least $\widehat{R} - 2^{-\ell}$, forming the next smaller candidate set $\mathcal{A}_{\ell+1}$. The algorithm then proceeds to the next iteration and repeats steps~(a)–(c) until either step~(a) or (b) is triggered in some round. Note that, by our choice of $J = \lceil \frac{1}{2}\log_2 T\rceil$, the algorithm is guaranteed to terminate on or before the $J$-th loop. Most importantly, the elimination step never discards the optimal assortment--price pair.

\begin{proposition}\label{prop-optimal-elimination}
With probability at least $1-1/T$, for every $\tau\leq t\leq T$ it holds that $(S_t^\star, \bm p_t^\star)\in \mathcal{A}_\ell$ if the $t$-th step terminates at the $\ell$-th loop.
\end{proposition}

%------------------------------------------------------%
\paragraph{Conditional Independence Guarantee.}
The key reason for using the SupCB framework is that it preserves conditional independence within each bin. The decision to keep or remove an action in bin $\ell$ depends only on the uncertainty terms $w_{ti}^{\ell}(p_i)$ and $W_{t,\ell}(S,\bm p)$ together with the fixed estimator $\hat{\bm\vartheta}_0$, which is computed solely from the samples in $\Psi_{J+1}$. Hence, the choice outcomes collected in bin $\ell$ remain conditionally independent,
exactly as required by Theorem~\ref{prop-sup-ucb-confidence-bound}.

%-----------------------subsec: Analysis-------------------------------%
%----------------------------------------------------------------------%
\subsection{Regret Guarantee of SupCB-MNL-Pricing}

We now state the regret guarantee of Algorithm~\ref{alg:sup-lin-mnl-pricing}. The proof combines the confidence region in Theorem~\ref{prop-sup-ucb-confidence-bound}, the burn-in guarantee from Lemma~\ref{lem-initial-length-bound}, and a perturbation analysis of the MNL revenue function under joint assortment and pricing decisions.

\begin{theorem}\label{thm-regret-supCB}
    Under Assumption~\ref{assump-bounded-pricing} and~\ref{assump-mim-price}, with probability at least $1-1/T,$ Algorithm~\ref{alg:sup-lin-mnl-pricing} achieves
\begin{align*}
    \mathrm{Reg}(T) = \widetilde{\mathcal{O}}\bigg( W\sqrt{dT\log (NT)}/L_0 + \kappa^{-1}d^2W\log(NT)/L_0\bigg).
\end{align*}
\end{theorem}

Compared to the results in \cite{erginbas2025online}, Theorem~\ref{thm-regret-supCB} improves the leading-order regret bound by removing a $\sqrt{d}$ factor, at the cost of introducing an additional $\sqrt{\log N}$ term. As a result, it outperforms the previous bounds whenever $\log N = o(d)$, i.e., when $N$ is subexponential in $d$.

A useful benchmark within our framework is the \textit{canonical $N$-item setting} studied by \citet{miao2021dynamic}, where $d=N$ and $\bx_i=\be_i$ for all $i\in[N]$. In this setting, \citet{miao2021dynamic} proposed a posterior-sampling algorithm with regret $\widetilde{\mathcal O}(\kappa^{-1}e^W\sqrt{NT\log N}/L_0^2)$. By contrast, Theorem~\ref{thm-regret-supCB} yields the sharper guarantee $\widetilde{\mathcal O}(W\sqrt{NT\log N}/L_0)$, removing the exponential dependence on $\kappa^{-1}$ and $e^W$, as well as one factor of $1/L_0$.

\begin{remark}[Lower bound for the online joint assortment--pricing problem]
A natural lower-bound benchmark for the online joint assortment--pricing problem follows from the same canonical specialization. Consider the fixed-design setting with $d=N$ and $\bx_i=\be_i$ for all $i\in[N]$, and suppose prices are fixed across rounds, i.e., $p_{ti}\equiv p\in\cP$ for all $t$ and $i$. Then, after reparameterizing $\bm\theta^\star=\bm\psi^\star-p\bm\phi^\star$, the price-based contextual MNL model reduces to the canonical $N$-item MNL assortment problem with uniform revenue $p$. Since the lower bound in \citet{chen2018note} for the unit-revenue case scales linearly with the revenue level, this specialization yields the lower bound $\Omega(p\sqrt{NT})$, or equivalently $\Omega(p\sqrt{dT})$ when $N=d$. In particular, taking $p=\Theta(W/L_0)$ suggests that $\Omega(W\sqrt{dT}/L_0)$ is the natural lower-bound benchmark for the online joint assortment--pricing problem.
\end{remark}

Theorem~\ref{thm-regret-supCB} should primarily be viewed as a statistical result. In particular, the elimination step in line~20 and the confidence-level computation in line~11 both require enumeration over $\mathcal A_\ell$, as in \citet{chen2020dynamic,oh2021multinomial}, and are therefore computationally expensive. Designing computationally efficient algorithms that retain the same regret guarantee remains an interesting direction for future work.
%----------------------------------------------------linear-----------------------------------%
\subsubsection{Implications for the Linear MNL Model}\label{subsec:online-linear-special}

As discussed in Section~\ref{sec3.2:linear-confidence bound}, the price-based model reduces to the linear contextual MNL model when the feature map is price-independent. In this case, SupCB-MNL-Pricing reduces to an assortment-only SupCB algorithm for linear MNL. The following corollary is immediate from Theorem~\ref{thm-regret-supCB} by following the same analysis and removing the price decision.

\begin{corollary}\label{cor-online-linear-mnl}
Under Assumption~\ref{assump-bounded}, with probability at least $1-1/T$, the linear MNL specialization satisfies
\begin{align*}
\mathrm{Reg}(T) =
\widetilde{\mathcal{O}}\left(
M\sqrt{dT\log(NT)}
+
\kappa^{-1}d^2M\log(NT) \right),
\end{align*}
where $M$ is an upper bound on the revenue, typically normalized to 1.
\end{corollary}

Compared to the $\widetilde{\mathcal{O}}(\kappa^{-1}\sqrt{dT})$ bound in \citet{oh2021multinomial} and the $\widetilde{\mathcal{O}}(d\sqrt{T} + \kappa^{-1}d^2)$ bound in \citet{lee2024nearly}, our analysis yields the first $\widetilde{\mathcal{O}}(\sqrt{dT\log N})$ leading-order regret for the linear MNL setting, successfully relegating the $\kappa^{-1}$ dependence to a second-order term.

%-------- Fixed-Assortment Case --------------------------------------------------------%
\subsubsection{Implications for the Fixed-Assortment Pricing Problem}\label{subsec:online-fixed-assortment}

A natural specialization of the joint assortment-and-pricing problem arises when the offered assortment $S$ is held fixed throughout the horizon. In this case, the action space reduces from assortment-price pairs to price vectors only, yielding a multi-item pricing problem that has been studied extensively in the literature; see e.g., \cite{li2011pricing,wang2012capacitated,broder2012dynamic,javanmard2020multi,xu2021logarithmic}. Let $S\in\mathcal{S}_K$ be fixed.
With the fixed assortment, the regret of a policy $\pi$ is
\[
\mathrm{Reg}(T) := \mathbb{E}\left[\sum_{t=1}^T \big(R(S,\bm p^\star|\bm\vartheta^\star) - R(S,\bm p_t|\bm\vartheta^\star)\big)\right],
\]
where the optimal price vector $\bm p^\star$ is defined as $\bm p^\star \in \operatorname*{argmax}_{\bm p \in \mathcal{P}^{|S|}} R(S,\bm p|\bm\vartheta^\star)$.

The fixed-assortment specialization of Algorithm~\ref{alg:sup-lin-mnl-pricing} simply replaces the candidate sets
$A_\ell \subseteq \mathcal S_K \times P^N$ by $A_\ell \subseteq P^N$, while retaining the same initial
exploration, sample-splitting, confidence construction, and elimination steps. In particular, the
price-uniform confidence region in Theorem~\ref{prop-sup-ucb-confidence-bound} continues to control the estimation error, and the proof of Theorem~\ref{prop-sup-ucb-confidence-bound} specializes directly to this
reduced action space.

\begin{corollary}\label{cor-online-fixed-assortment}
Under Assumption~\ref{assump-bounded} and \ref{assump-mim-price}, with probability at least $1-1/T$, the fixed-assortment specialization of Algorithm~\ref{alg:sup-lin-mnl-pricing} satisfies
\begin{align*}
\mathrm{Reg}(T) = \widetilde{\mathcal{O}}\left(
\frac{W}{L_0}\sqrt{dT\log(KT)}
+
\frac{W}{L_0}\kappa^{-1}d^2\log(KT)
\right).
\end{align*}
\end{corollary}

Compared with the $\widetilde O(d\sqrt{T})$ regret for multi-product contextual MNL pricing under adversarial arrivals in \cite{goyal2022dynamic}, and the $\widetilde O(d\sqrt{T})$ regret for single-product contextual pricing with arbitrary covariates in \cite{wang2025dynamic}, Corollary~\ref{cor-online-fixed-assortment} yields a $\widetilde O(\sqrt{dT})$ leading-order regret bound for multi-product contextual MNL pricing. The $\sqrt d$ improvement is enabled by our SupCB-style algorithm, which leverages phased elimination and sample splitting to avoid the extra $\sqrt d$ loss present in prior one-shot UCB/ONS analyses. This leading-order dependence extends the $\widetilde \Theta(\sqrt{dT})$ benchmark established for single-product contextual pricing with stochastic contexts in \cite{zhao2024contextual} to the multi-product MNL setting.

Indeed, in Appendix~\ref{sec-appendix-time-varying-pricing}, we show that Algorithm~\ref{alg:sup-lin-mnl-pricing} extends to a broader partially adversarial setting: only the contexts used during exploration epochs are required to be i.i.d.\ across rounds, while the remaining contexts may be chosen adversarially. Under a nondegeneracy condition on the exploration covariance, specifically $\lambda_{\min}\!\Big(\E\big[\widetilde{\bx}(p) \widetilde{\bx}(p)^\top\big]\Big)\ge \sigma_0>0,$ the same sample-splitting and elimination argument yields the same leading-order $\widetilde O(\sqrt{dT})$ regret bound.

\subsection{A Computational Alternative: Thompson Sampling}\label{sec:TS-pricing}

The SupCB-MNL-Pricing algorithm gives a frequentist regret guarantee by enforcing conditional independence via sample splitting. In practice, one may also prefer a simpler computational approach that directly handles the continuous price decision. In this subsection, we present
a Thompson Sampling (TS) algorithm for the online joint assortment and pricing problem and state its Bayesian regret guarantee. This approach can be viewed as a pricing extension of TS methods for linear MNL contextual bandits in \cite{oh2019tsmnl}, and it shares structural similarities with the random-sampling methods for dynamic pricing in \cite{miao2021dynamic}. The full procedure is detailed in Algorithm~\ref{alg:TS-MNL}.

\paragraph{Bayesian Regret and Posterior Update.}
% Because Thompson Sampling inherently assumes that the unknown parameter is drawn from a prior distribution, its performance is most naturally evaluated using a Bayesian framework. 
Let the unknown parameter be $\bm\vartheta\in\mathbb R^{2d}$ and assume a prior distribution $Q$ supported on $B_2(W)=\{\bm\vartheta:\|\bm\vartheta\|_2\le W\}$. We define the \textit{Bayesian regret} as
\begin{align}
    BR_{Q}(T)
    =\E_{\bm\vartheta\sim Q}\!\left[\sum_{t=1}^T \Big(R(S^\star,\bm p^\star|\bm\vartheta)
    -R(S_t,\bm p_t|\bm\vartheta)\Big)\right],
\end{align}
where $(S^\star,\bm p^\star)\in\argmax_{S\in\cS_K,\bm p\in\cP^N}R(S,\bm p|\bm\vartheta)$ denotes an optimal assortment--price pair under $\bm\vartheta$. 

Operating under a \textit{fixed design} setting, as the algorithm interacts with the environment, it collects observations $\{(S_\tau,\bm p_\tau,i_\tau)\}_{\tau=1}^{t-1}$. Using these observations, the posterior distribution at time $t$ is updated via Bayes' rule:
\begin{align}\label{eq:post-distri}
    Q_t(\bm\vartheta)
    \propto Q(\bm\vartheta)\prod_{\tau=1}^{t-1} q_{i_\tau}(S_\tau,\bm p_\tau|\bm\vartheta).
\end{align}

\begin{algorithm}[h]
\caption{TS-MNL-Pricing}
\label{alg:TS-MNL}
\begin{algorithmic}[1]
\STATE \textbf{Input:} Time horizon $T$, fixed item feature matrix $X$, prior distribution $Q$, price upper bound $P$.
\STATE \textbf{Initialize:} $\mathcal{P}:=[0,P]$.
\FOR{$t=1,\dots,T$}
    \STATE Sample $\bm\vartheta_t\sim Q_t$ using \eqref{eq:post-distri}.
    \STATE Choose $(S_t,\bm p_t)\in\argmax_{S\in\cS_K,\bm p\in\mathcal{P}^N} R(S,\bm p|\bm\vartheta_t)$ and observe $i_t$.
\ENDFOR
\end{algorithmic}
\end{algorithm}

\paragraph{Computation.}
Algorithm~\ref{alg:TS-MNL} requires two computational steps in each round: (i) sampling $\bm\vartheta_t$ from the posterior $Q_t$, and (ii) solving the deterministic joint assortment and pricing problem $\max_{S,\bm p} R(S,\bm p|\bm\vartheta_t)$. For posterior sampling, we can use standard MCMC routines (e.g., Metropolis--Hastings) to obtain approximate samples from $Q_t$. For the optimization step, given $\bm\vartheta_t$, the MNL joint assortment and pricing problem can be solved efficiently using known algorithms based on fixed-point characterizations and one-dimensional search; see, e.g., \citet{WANG2012492} for a polynomial-time procedure under a capacity constraint.

\begin{theorem}\label{thm-TS-regret}
Under Assumption~\ref{assump-bounded-pricing} and~\ref{assump-mim-price}, Algorithm~\ref{alg:TS-MNL} achieves the Bayesian
regret bound
\begin{align*}
    BR_Q(T) =\widetilde{\mathcal{O}}\left( \frac{W}{L_0}\left( d\sqrt{T} +{\kappa^{-1}}d^2 \right)
     \right).
\end{align*}
\end{theorem}

\paragraph{Comparison with Existing Literature.}
After treating $W/L_0$ as the effective price-range factor, the leading term in Theorem~\ref{thm-TS-regret} has order $\widetilde{\mathcal O}(d\sqrt{T})$. Thus, the Bayesian regret has the same leading dependence on $d$ and $T$ as the TS result of \citet{oh2019tsmnl} for linear MNL contextual bandits without pricing, showing that adding a continuous pricing decision does not fundamentally change the statistical rate. More broadly, this result is consistent with recent online learning guarantees for contextual MNL pricing and assortment, such as the $\widetilde{\mathcal O}(d\sqrt{T})$ regret bound in \citet{perivier2022dynamic}, and with the Bayesian regret analysis of random-sampling based methods in \citet{miao2021dynamic}.

%% ---- END 5.1 MNL Pricing Online.tex ----

% \input{5-Linear MNL Online}

%--------------Technical Proof---------------------%
% \input{6-Technical Proof}

\bibliographystyle{ims}
\bibliography{sample}

%%%%%%%%%%%%%%%%%%%%%%%%%%%%%%%%%%%%%%%%%%%%%%%%%%%%%%%%%%%%%%%%%%%%%%%%%%%%%%%
%%%%%%%%%%%%%%%%%%%%%%%%%%%%%%%%%%%%%%%%%%%%%%%%%%%%%%%%%%%%%%%%%%%%%%%%%%%%%%%
% APPENDIX
%%%%%%%%%%%%%%%%%%%%%%%%%%%%%%%%%%%%%%%%%%%%%%%%%%%%%%%%%%%%%%%%%%%%%%%%%%%%%%%
%%%%%%%%%%%%%%%%%%%%%%%%%%%%%%%%%%%%%%%%%%%%%%%%%%%%%%%%%%%%%%%%%%%%%%%%%%%%%%%
\newpage
\appendix

%% ---- BEGIN appendix_section2.tex ----
%----------proof of lemma in Section2-----------%

\section{Proof of Lemma~\ref{p-bound-refined}}
\begin{proof}[\textbf{Proof of Lemma~\ref{p-bound-refined}}] 
If we denote $\alpha_{ti}:= \bm{x}_{ti}^\top \bm\psi^\star, \beta_{ti}:= \bm{x}_{ti}^\top\bm{\phi}^\star,$ then
by Corollary~3 in \cite{wang2012capacitated}, we have there exists some constant $p_0>0$ so that $p_{tj}^\star = p_0 + 1/\beta_{tj}$ for all $j\in S^\star.$ As a result, we have then \begin{align*}
    R_t(S_t^\star,\bm p^\star_t|\bm\vartheta^\star) &= \frac{\sum_{j\in S_t^\star} (p_{0} + 1/\beta_{tj}) \exp(\alpha_{tj}- \beta_{tj}p_0 - 1)}{1+ \sum_{j\in S_t^\star}  \exp(\alpha_{tj}- \beta_{tj}p_0 - 1)}\\
    &\leq  \frac{(p_0 + 1/{L_0})\sum_{j\in S_t^\star} \exp(\alpha_{tj}- \beta_{tj}p_0 - 1)}{1+ \sum_{j\in S_t^\star}  \exp(\alpha_{tj}- \beta_{tj}p_0 - 1)}.
\end{align*}
On the other hand, by $\bm{p}_t^\star$ is the optimal price, we have \begin{align*}
    &\frac{1}{L_0}\cdot\frac{\sum_{j\in S_t^\star} \exp(\alpha_{tj}- \beta_{tj}p_0 - 1)}{1+ \sum_{j\in S_t^\star}  \exp(\alpha_{tj}- \beta_{tj}p_0 - 1)} \\
    &\geq p_0 \bigg(\frac{\sum_{j\in S_t^\star} \exp(\alpha_{tj}- \beta_{tj}p_0 )}{1+ \sum_{j\in S_t^\star} \exp(\alpha_{tj}- \beta_{tj}p_0 )} - \frac{\sum_{j\in S_t^\star} \exp(\alpha_{tj}- \beta_{tj}p_0 - 1)}{1+ \sum_{j\in S_t^\star}  \exp(\alpha_{tj}- \beta_{tj}p_0 -1)} \bigg).
\end{align*}
Let $A:=\sum_{j\in S_t^\star} \exp(\alpha_{tj}- \beta_{tj}p_0 )$ and re-arrange the inequality, we can get \begin{align*}
    &\frac{1}{L_0} \cdot \frac{A}{e+A} \geq p_0 \big(\frac{A}{1+A} - \frac{A}{e+A} \big)
    = p_0\frac{(e-1)A}{(1+A)(e+A)}\\
    &\implies p_0\leq \frac{1+A}{L_0(e-1)}  \leq \frac{1+K \exp(W-L_0p_0)}{L_0(e-1)}.
\end{align*}
Finally, we have either $p_0\leq \frac{W+\log K}{L_0}$ or \begin{align*}
    p_0\geq \frac{W+\log K}{L_0}\implies \exp(W-L_0p_0) \leq 1/K \implies p_0\leq \frac{1+K\exp(W-L_0p_0)}{L_0} \leq \frac{2}{L_0}.
\end{align*}
This leads to \begin{align*}
    \max_{j\in S^\star_t} p_{tj}^\star \leq p_0 + \frac{1}{L_0}\leq \frac{3+W+\log K}{L_0},
\end{align*}
as desired.
\end{proof}
%% ---- END appendix_section2.tex ----

%% ---- BEGIN appendix_section3.tex ----

\section{Proof of Theorem~\ref{prop-sup-ucb-confidence-bound} and Corollary~\ref{thm-sup-ucb-confidence-bound}}
\label{app:proof-confidence}

We prove the confidence region in the linear MNL model and then obtain the price-based result as an immediate extension.
The two proofs are identical: the price-based model only changes the feature dimension, since we can treat
$\widetilde{\bx}(p)=(\bx,-p\bx)\in\mathbb R^{2d}$ as a generic feature vector throughout the argument. We do not rely on any special structure induced by the price variable, beyond boundedness assumptions. 

For notational simplicity, we therefore present the proof for the linear MNL model. The proof of Theorem~\ref{prop-sup-ucb-confidence-bound} for the price-based model then follows directly by replacing $\bx\in\mathbb R^{d}$ with $\widetilde{\bx}(p)\in\mathbb R^{2d}$ and replacing $\bm\theta\in\mathbb R^{d}$ with $\bm\vartheta\in\mathbb R^{2d}$.

%--------------------------------------------------------------------%
%---------------------Proof of Corollary 3.4-------------------------%
%--------------------------------------------------------------------%

\paragraph{Setup and notation.}
We define the linear MNL log-likelihood function as $$\ell_{\cD}(\bm\theta):= -\sum_{m = 1}^t \sum_{j \in (S_m)_+} y_{mj} \log q_j(S_m\lvert \exp(\bm{X}_m^\top \bm\theta)),$$ and denote its Hessian matrix by $\bH^\lambda_{\cD}(\bm\theta)$. Then, we have $\ell_{\cD}(\bm\theta)=\ell_{\cD}(\bm\theta)-\frac{\lambda}{2}\|\bm\theta\|_2^2$.

Throughout the proof, we omit the subindex $\cD$ for simplicity. More precisely, we denote $\ell_{\cD}^\lambda$ by $\ell^\lambda,$ $\hat{\bm\theta}_{\cD}^\lambda$ by $\hat{\bm\theta}^\lambda$, $\bH_{\cD}^\lambda(\bm\theta^\star)$ by $\bH^\lambda$, $\bH_{\cD}(\bm\theta^\star)$ by $\bH$ for simplicity. In addition, we denote $q_j(S_m\lvert \exp(\bm{X}_m^\top \bm\theta))$ by $q_{mj}(\bm\theta)$. 

Now, we begin with the following Lemma:
\begin{lemma}\label{lem-confidence-step1}
Suppose the same independence structure as in Theorem~\ref{thm-sup-ucb-confidence-bound}, and define 
\begin{align*}
 \zeta:= 3\sqrt{2}\max_{m\leq t, j\in S_m} \lvert\bm x_{mj}^\top (\hat{\bm\theta}^\lambda- \bm\theta^\star) \rvert,\quad \xi: =   \max_{m\leq t,j\in S_m} \lVert\bm x_{mj}\rVert_{(\bH^\lambda)^{-1}},\quad \varphi(\zeta) := \left(\frac{e^\zeta - 1}{\zeta} - 1\right)(1+\zeta),
\end{align*}  
then, with probability at least $1-\delta$, we have
\begin{align*}
  \frac{\lvert\bm x^\top (\hat{\bm\theta}^\lambda - \bm\theta^\star)\rvert}{8\lVert\bm x \rVert_{(\bH^\lambda)^{-1}}}  &\leq \left(\sqrt{\log(1/\delta)} + \xi  \log(1/\delta)\right)+ \varphi(\zeta)  \left( \sqrt{d\log (1/\delta)} + d\xi\log(1/\delta)  + \sqrt{\lambda}W \right) + \sqrt{\lambda(1+\zeta)}W.
\end{align*}
\end{lemma}

\begin{proof}[\textbf{Proof of Lemma~\ref{lem-confidence-step1}}]
Noticing that $\hat{\bm\theta}^\lambda$ minimizes $\ell^\lambda(\bm\theta)$, we have 
\begin{align*}
    & \nabla \ell^\lambda(\hat{\bm\theta}^\lambda) = 0  \implies \sum_{m = 1}^t \sum_{j\in S_m}\bm x_{mj} \left(q_{mj}(\hat{\bm\theta}^\lambda) - y_{mj} \right) + \lambda \hat{\bm\theta}^\lambda= 0  \implies\\
    & \sum_{m = 1}^t \sum_{j\in S_m}\bm x_{mj} \left(q_{mj}(\hat{\bm\theta}^\lambda) - q_{mj}(\bm\theta^\star) \right) +  \lambda(\hat{\bm\theta}^\lambda - \bm\theta^\star) = \sum_{m = 1}^t \sum_{j\in S_m} \bm x_{mj} \underbrace{\left(y_{mj} - q_{mj}(\bm\theta^\star) \right)}_{: = \eta_{mj}} - \lambda\bm\theta^\star
\end{align*}
Now, we define $\bL_m(\bm\theta): = \sum_{j\in S_m}\bm x_{mj}q_{mj}(\bm\theta)$ and its derivative as 
\begin{align*}
      \bDL_m(\bm\theta):= \sum_{j\in S_m} q_{mj}(\bm\theta)\bm x_{mj}\bm x_{mj}^\top - \sum_{i\in S_mj \in S_m} q_{mj}(\bm\theta)q_{mi}(\bm\theta)\bm x_{mi}\bm x_{mj}^\top.
\end{align*}
For any $\bm\theta$, we have the following 
\begin{align*}
&  \sum_{j\in S_m}\bm x_{mj} \left(q_{mj}(\bm\theta) - q_{mj}(\bm\theta^\star) \right)= \bL_m(\bm\theta) - \bL_m(\bm\theta^\star) \\
&= \int_{0}^1 \bDL_{m}(\bm\theta^\star + s(\bm\theta-\bm\theta^\star))(\bm\theta - \bm\theta^\star)  ds\\
    &= \bDL_m(\bm\theta^\star)(\bm\theta - \bm\theta^\star) + \left(\int_{0}^1 \left(\bDL_{m}(\bm\theta^\star + s(\bm\theta-\bm\theta^\star)) - \bDL_m(\bm\theta^\star) \right) ds\right) (\bm\theta - \bm\theta^\star),
\end{align*}
Note that  $\bH^\lambda= \sum_{m= 1}^t \bDL_m(\bm\theta^\star) + \lambda \bI$ and define
\begin{align*}
     \bE:= \sum_{m= 1}^t \int_{0}^1 \left(\bDL_{m}(\bm\theta^\star + s(\hat{\bm\theta}^\lambda-\bm\theta^\star)) - \bDL_m(\bm\theta^\star) \right) ds.
\end{align*}
It follows that
\begin{align*}
    (\bH^\lambda+\bE)(\bm\theta - \bm\theta^\star) = \sum_{m=1}^t \sum_{j \in S_m}\bm x_{mj} \eta_{mj} -\lambda \bm\theta^\star .
\end{align*}
Thus, for any $\bm x\in \mathbb R^d$,
\begin{align*}
    &\bm x^\top (\hat{\bm\theta}^\lambda-\bm\theta^\star) =\bm x^\top (\bH^\lambda + \bE)^{-1}\left( \sum_{m=1}^t \sum_{j \in S_m}\bm x_{mj} \eta_{mj} -\lambda \bm\theta^\star \right) \\
    &=\bm x^\top \left( (\bH^\lambda)^{-1} -(\bH^\lambda)^{-1}\bE(\bH^\lambda+\bE)^{-1}\right)\left( \sum_{m=1}^t \sum_{j \in S_m}\bm x_{mj} \eta_{mj} -\lambda \bm\theta^\star \right) \\
    &= \underbrace{\bm x^\top  (\bH^\lambda)^{-1} \sum_{m=1}^t \sum_{j \in S_m}\bm x_{mj} \eta_{mj}}_{:= J_1} \underbrace{ - \bm x^\top (\bH^\lambda)^{-1}\bE(\bH^\lambda+\bE)^{-1} \sum_{m=1}^t \sum_{j \in S_m}\bm x_{mj} \eta_{mj}}_{:= J_2}  \underbrace{ - \lambda\bm x^\top (\bH^\lambda+\bE)^{-1}\bm\theta^\star}_{:= J_3}.
\end{align*}
Now we provide upper bounds for $J_1,J_2,J_3$ separately.

%-----------------Bound term J1---------------------%
%---------------------------------------------------%
\subsubsection*{The first term $J_1$} 
The first term can be bounded by variance-aware concentration results using conditional independence (Assumption~\ref{assumption-conditional-independence}) as in previous works, obtaining sharp bounds for the logistic setting \cite{jun2021improved}, the main difference between the logistic setting is the dependency across $\eta_{mj}$ for $j\in S_m:$ 

\noindent For every $m$, we have 
\begin{align*}
 Z_m:=   \sum_{j\in S_m}\bm x^\top (\bH^\lambda)^{-1}\bm x_{mj} \eta_{mj} 
\end{align*}
is centered and bounded as 
$$\lvert Z_{m}\rvert\leq \max_{j\in S_m}\lvert\bm x^\top (\bH^\lambda)^{-1}\bm x_{mj}\rvert\cdot \sum_{j\in S_m} \lvert \eta_{mj}\rvert \leq 2\max_{j\in S_m}\lvert\bm x^\top (\bH^\lambda)^{-1}\bm x_{mj}\rvert,$$
where the last inequality follows from
    $\sum_{j \in S_m} \lvert \eta_{mj}\rvert\leq \sum_{j\in S_m} y_{mj}+ \sum_{j\in S_m} q_{mj}(\bm \theta^\star)\leq 2.$

\noindent On the other hand, the variance of $Z_m$ is given as
\begin{align*}
    \mathbb{E}[Z_m^2] &= \sum_{i,j\in S_m} (\bm x^\top (\bH^\lambda)^{-1}\bm x_{mj})(\bm x^\top (\bH^\lambda)^{-1}\bm x_{mi})  \mathbb{E}[\eta_{mj}\eta_{mi}].
\end{align*}
Note that 
\begin{align*}
    \E[\eta_{mi}^2] = q_{mi}(\bm\theta^\star)(1-q_{mi}(\bm\theta^\star)),\quad \E[\eta_{mi}\eta_{mj}] = -q_{mi}(\bm\theta^\star) q_{mj}(\bm\theta^\star),
\end{align*}
implies that 
\begin{align*}
        \mathbb{E}[Z_m^2] &=\bm x^\top (\bH^\lambda)^{-1}\underbrace{\left( \sum_{i\in S_m}\bm x_{mi}\bm x_{mi}^\top q_{mi}(\bm\theta^\star) (1-q_{mi}(\bm\theta^\star)) - \sum_{i\neq j}\bm x_{mi}\bm x_{mj}^\top q_{mi}(\bm\theta^\star) q_{mj}(\bm\theta^\star) \right)}_{: = \bU_m} (\bH^\lambda)^{-1}\bm x.
\end{align*}
By $\sum_{m = 1}^t \bU_m = \bH,$ we have $\sum_{m = 1}^t \E[Z_m^2] =\bm x^\top (\bH^\lambda)^{-1} \bH(\bH^\lambda)^{-1}\bm x\leq \bm x^\top (\bH^\lambda)^{-1}\bm x$. Now, we can apply the following Bernstein inequality:
\begin{lemma}[Bernstein's Inequality]\label{lem-bernstein} Let $  X_1, \ldots,  X_n$ be independent zero-mean random variables. Suppose that $\left|  X_i\right| \leq M$ almost surely for all $i$. Then, for all positive $t$,
$$
\mathbb{P}\left(\sum_{i=1}^n   X_i \geq u\right) \leq \exp \left(-\frac{\frac{1}{2} u^2}{\sum_{i=1}^n \mathbb{E}\left[  X_i^2\right]+\frac{1}{3} M u}\right).
$$
\end{lemma}
\noindent to obtain that \begin{align*}
    \mathbb{P}\left( \lvert J_1\rvert  \geq  u\right) \leq \exp\left( -\frac{u^2}{2 \lVert\bm x \rVert_{(\bH^\lambda)^{-1}}^2 + 2\max_{m\leq t}\max_{j\in S_m}\lvert\bm x^\top (\bH^\lambda)^{-1}\bm x_{mj} \rvert  u }  \right).
\end{align*}
Introducing the notation  $\xi =   \max_{m\leq t,j\in S_m} \lVert\bm x_{mj}\rVert_{(\bH^\lambda)^{-1}} $ in Lemma~\ref{lem-confidence-step1}, equivalently, we can equivalently state that, with probability at least $1-\delta,$
\begin{align}\label{appendix-eq: J1-bound}
    \lvert J_1\rvert \leq 2\lVert\bm x \rVert_{(\bH^\lambda)^{-1}}\left(\sqrt{\log(1/\delta)} + \xi \log(1/\delta)\right).
\end{align}

%-----------------Bound term J2---------------------%
%---------------------------------------------------%
\subsubsection*{The second term $J_2$} For the second term, we denote $\by_t:= \sum_{m=1}^t \sum_{j \in S_m}\bm x_{mj} \eta_{mj}$, then 
\begin{align*}
    \lvert J_2 \rvert & = \lvert\bm x^\top (\bH^\lambda)^{-1}\bE(\bH^\lambda+\bE)^{-1} \by_t\rvert\\
    &\leq \lVert\bm x\rVert_{(\bH^\lambda)^{-1}} \lVert (\bH^\lambda)^{-1/2} \bE (\bH^\lambda)^{-1/2}\rVert \lVert (\bH^\lambda + \bE)^{-1}\by_t \rVert_{\bH^\lambda} .
\end{align*}

\noindent $\bullet$ \textbf{ Bounding $\lVert (\bH^\lambda)^{-1/2}\bE(\bH^\lambda)^{-1/2}\rVert$.}
Given any vector $\bm w$, if we denote $$a_{mj}:=\bm x_{mj}^\top \bm w,\ b_{mj} =\bm x_{mj}^\top (\hat{\bm\theta}^\lambda - \bm\theta^\star),\ q_{mj}(s):= q_{mj}(\bm\theta^\star + s(\hat{\bm\theta}^\lambda  - \bm\theta^\star)),$$
and 
$$ u_{mj}(s):=\bm x_{mj}^\top (\bm\theta^\star + s(\hat{\bm\theta}^\lambda  - \bm\theta^\star)),\ v_{mj}(s):= e^{u_{mj}(s)},$$ 
for $j\in S_m$, and set $a_{m0}=b_{m0} = u_{m0} = 0, v_{m0} = 1$
for convenience. Then, we have
\begin{align*}
\lVert \bm w \rVert_{ \bDL_m(\bm\theta^\star + s(\hat{\bm\theta}^\lambda - \bm\theta^\star))}^2 =  \frac{\sum_{j\in S_m}\sum_{i\in S_m: 0\leq i<j} (a_{mi} - a_{mj})^2 v_{mi}(s) v_{mj}(s) }{(1+\sum_{j \in S_m} v_{mj}(s))^2 }.
\end{align*}
Analogous to the analysis in Proposition C.1 of \cite{lee2024nearly}, we have
\begin{align*}
&\left\lvert \frac{d}{ds} \lVert \bm w \rVert_{ \bDL_m(\bm\theta^\star + s(\hat{\bm\theta}^\lambda - \bm\theta^\star))}^2 \right\rvert \\ =& 
\left\lvert \frac{\sum_{j\in S_m} \sum_{i\in (S_m)_+: 0\leq i< j} (a_{mi} - a_{mj})^2 v_{mi}(s) v_{mj}(s)\left[\sum_{k\in (S_m)_+ } (b_{mi} + b_{mj} - 2 b_{mk})v_{mk}(s) \right]}{(1+\sum_{j\in S_m} v_{mj}(s))^3}\right\rvert\\
\leq & 3\sqrt{2}\max_{j\in S_m} \lvert b_{mj}\rvert \cdot \frac{\sum_{j\in S_m} \sum_{i\in (S_m)_+: 0\leq i< j} (a_{mi} - a_{mj})^2 v_{mi}(s) v_{mj}(s)}{(1+\sum_{j\in S_m} v_{mj}(s))^2}\\
\leq& 3\sqrt{2}\max_{j\in S_m}\lvert b_{mj}\rvert \lVert \bm w\rVert_{\bDL_m(\bm\theta^\star + s(\hat{\bm\theta}^\lambda  - \bm\theta^\star))}^2.
\end{align*}
As a result, $\psi(s):= \log  \lVert \bm w\rVert_{\bDL_m(\bm\theta^\star + s(\hat{\bm\theta}^\lambda  - \bm\theta^\star))}^2$
satisfies $\lvert \psi'(s) \rvert \leq 3\sqrt{2}\max_{j\in S_m} \lvert b_{mj}\rvert$, which then leads to \begin{align}\label{eq: proof-to-H-dom}
 e^{-3\sqrt{2}\max_{j\in S_m} \lvert b_{mj}\rvert \cdot s}\lVert \bm w\rVert^2_{\bDL_m(\bm\theta^\star)}   \leq \lVert \bm w\rVert^2_{\bDL_m(\bm\theta^\star + s(\hat{\bm\theta}^\lambda  - \bm\theta^\star))} \leq e^{3\sqrt{2}\max_{j\in S_m} \lvert b_{mj}\rvert \cdot s}\lVert \bm w\rVert^2_{\bDL_m(\bm\theta^\star)}.
\end{align}
Use the notation $\zeta= 3\sqrt{2}\max_{m\leq t, j\in S_m} \lvert\bm x_{mj}^\top (\hat{\bm\theta}^\lambda- \bm\theta^\star) \rvert$ introduced in Lemma~\ref{lem-confidence-step1}, and note that $\bH= \sum_{m= 1}^t \bDL_m(\bm\theta^\star)$. For any $\bm w\in \mathbb{R}^d$, by summing over $m$, we obtain
\begin{align*}
    \bm w^\top \bE \bm w & =   \int_0^1 \sum_{m=1}^t\left(  \bm w^\top \bDL_m(\bm\theta^\star + s(\hat{\bm\theta}^\lambda - \bm\theta^\star))\ \bm w - \bm w^\top \bDL_m(\bm\theta^\star)\ \bm w\right) ds\\
    & =   \int_0^1 \sum_{m=1}^t\left(  \lVert \bm w\rVert^2_{\bDL_m(\bm\theta^\star + s(\hat{\bm\theta}^\lambda  - \bm\theta^\star))} - \lVert \bm w\rVert^2_{\bDL_m(\bm\theta^\star )} \right) ds\\
    &\leq  \int_0^1 (e^{\zeta s} - 1 )  \lVert \bm w\rVert_{\bH}^2\ ds  = \left(\frac{e^\zeta - 1}{\zeta} - 1 \right) \  \lVert \bm w \rVert_{\bH}^2,
\end{align*}
which leads to 
\begin{align}\label{eq-appendix-tmp1}
    \bE \preceq  \left(\frac{e^\zeta - 1}{\zeta} - 1 \right) \bH \preceq  \left(\frac{e^\zeta - 1}{\zeta} - 1 \right) \bH^\lambda.
    \end{align}
On the other hand, we have 
\begin{align*}
    \bH+\bE &= \int_0^1 \sum_{m=1}^t \bDL_m(\bm\theta^\star + s(\hat{\bm\theta}^\lambda - \bm\theta^\star)) ds \\
    &\succeq \int_0^1 e^{-\zeta s}  ds \cdot \bH \succeq \frac{1-e^{-\zeta}}{\zeta} \bH \succeq \frac{1}{1+\zeta} \bH,
\end{align*}
where the last inequality is by the elementary inequality $\frac{1-e^{-x}}{x}\geq \frac{1}{1+x},\ \forall x>0.$ This leads to
\begin{equation}\label{eq-appendix-tmp2}
\begin{aligned}
    \bE  \succeq \left(\frac{1}{1+\zeta}-1\right) \bH = -\frac{\zeta}{1+\zeta} \bH .
\end{aligned}
\end{equation}
Applying inequalities \eqref{eq-appendix-tmp1} and \eqref{eq-appendix-tmp2}, together with $\bm 0\preceq(\bH^\lambda)^{-1/2} \bH (\bH^\lambda)^{-1/2}\preceq \bI$, we obtain
\begin{align*}
   \lVert (\bH^\lambda)^{-1/2} \bE (\bH^\lambda)^{-1/2}\rVert   &= \max_{\lVert \bm w \rVert_2 = 1} \lvert \bm w^\top (\bH^\lambda)^{-1/2} \bE (\bH^\lambda)^{-1/2} \bm w \rvert\\
   &\leq  \max \bigg\{\frac{e^\zeta - 1}{\zeta} - 1,\ \frac{\zeta}{1+\zeta}  \bigg\} \leq 2 \big(\frac{e^\zeta - 1}{\zeta} - 1 \big),
\end{align*}
where the last inequality holds because $e^\zeta \ge 1+\zeta+\frac{\zeta^2}{2}$ implies $2 \big(\frac{e^\zeta - 1}{\zeta} - 1 \big)\ge \zeta\ge \frac{\zeta}{1+\zeta}$.
Thus,
\begin{align*}
    |J_2| \leq 2\lVert\bm x \rVert_{(\bH^\lambda)^{-1}}\left(\frac{e^\zeta - 1}{\zeta} - 1 \right) \lVert (\bH^\lambda + \bE)^{-1}\by_t \rVert_{\bH^\lambda}.
\end{align*}

\noindent $\bullet$ \textbf{Bounding $\lVert (\bH^\lambda + \bE)^{-1}\by_t \rVert_{\bH^\lambda}$.}
By~\eqref{eq-appendix-tmp2}, we have \begin{align*}
    \bH \preceq (1+\zeta) (\bH+\bE) &\implies (\bH^\lambda+\bE)^{-1} \bH^\lambda (\bH^\lambda+\bE)^{-1}\preceq (1+\zeta) (\bH^\lambda + \bE)^{-1}\preceq (1+\zeta)^2 \bH^\lambda.
\end{align*}
Thus, it holds that
\begin{align}\label{eq2}
    \lVert (\bH^\lambda + \bE)^{-1}\by_t \rVert_{\bH^\lambda}^2 &= \by_t^\top (\bH^\lambda+\bE)^{-1} \bH^\lambda (\bH^\lambda+\bE)^{-1}\by_t\nonumber\\
    &\leq (1+\zeta )^2  \lVert \by_t \rVert_{(\bH^\lambda)^{-1}}^2
\end{align}
To control $\lVert \by_t \rVert_{(\bH^\lambda)^{-1}}=\max_{\|\bm w\|_2=1}\lvert \bm w^\top (\bH^\lambda)^{-1/2}\by_t \rvert$, we have for any unit vector $\bm w$, it holds that \begin{align*}
    \bm w^\top (\bH^\lambda)^{-1/2}\by_t =\underbrace{((\bH^\lambda)^{1/2}\bm w)^\top}_{:= \bm{\tilde{x}}^\top} (\bH^\lambda)^{-1}\by_t.
\end{align*}
This expression has the same form as $J_1$ after replacing $\bm{\tilde{x}}$ by $\bm x$, thus by the same argument used in bounding $J_1$, inequality~\eqref{appendix-eq: J1-bound} still applies: with probability at least $1-\delta',$
\begin{align}\label{eq1}
   \lvert \bm w^\top (\bH^\lambda)^{-1/2}\by_t \rvert \leq 2\underbrace{\lVert \bm{\tilde{x}} \rVert_{(\bH^\lambda)^{-1}}}_{=\lVert \bm w\rVert_2 = 1}\left(\sqrt{\log(1/\delta')} + \max_{m\leq t,j\in S_m} \lVert\bm x_{mj}\rVert_{(\bH^\lambda)^{-1}}  \log(1/\delta')\right).
\end{align}
Now consider a $1/2$-net $\mathcal{N}\in\mathbb{S}^{d-1}$ on the unit ball, that is, for every unit $\bm w$, there exists $\bm v\in\mathcal{N}$ with $\|\bm w-\bm v\|_2\le\frac{1}{2}.$ Standard geometry gives a net with size $|\mathcal{N}|\le 4^d.$ Then, applying \eqref{eq1} to each $\bm v\in\mathcal{N}$ and taking the union bound with $\delta=4^d\delta'$ leads to 
\begin{align*}
  \mathbb{P}\left( \forall\bm v\in\mathcal
  N: \lvert \bm v^\top (\bH^\lambda)^{-1/2}\by_t \rvert \leq 2\left(\sqrt{\log(4^d/\delta)} + \xi \log(4^d/\delta)\right)\right)\ge 1-\delta,
\end{align*}
where we use the notation  $\xi =   \max_{m\leq t,j\in S_m} \lVert\bm x_{mj}\rVert_{(\bH^\lambda)^{-1}} $ as in Lemma~\ref{lem-confidence-step1}.
Note that $\lVert \by_t \rVert_{(\bH^\lambda)^{-1}} = \|(\bH^\lambda)^{-1/2}\by_t\|_2$. For any unit $\bm w$, pick $\bm v \in \mathcal{N}$ with $\|\bm w-\bm v\|_2\le\frac{1}{2}$. Then,
\begin{align*}
    \lvert \bm w^\top (\bH^\lambda)^{-1/2}\by_t \rvert \le \lvert \bm v^\top (\bH^\lambda)^{-1/2}\by_t \rvert + \|\bm w-\bm v\|_2 \|(\bH^\lambda)^{-1/2}\by_t\|_2 \le   \lvert \bm v^\top (\bH^\lambda)^{-1/2}\by_t \rvert + \frac{1}{2} \lVert \by_t \rVert_{(\bH^\lambda)^{-1}}.
\end{align*}
Taking the maximum over all $\|\bm w\|_2=1$ gives $$\lVert \by_t \rVert_{(\bH^\lambda)^{-1}} \le 2\max_{\bm v\in\mathcal{N}} \lvert \bm v^\top (\bH^\lambda)^{-1/2}\by_t \rvert.$$
Then, with probability at least $1-\delta,$ we have
\begin{align*}
    \lVert \by_t \rVert_{(\bH^\lambda)^{-1}} \leq 4\sqrt{\log(4^d/\delta)} + 4\xi\log(4^d/\delta)\leq 8\sqrt{d\log(1/\delta)} + 8d\xi\log(1/\delta),
\end{align*}
where the last inequality follows from the facts that when $\delta\le\frac{1}{4}$, we have $\log 4\le\log(1/\delta)$, and by applying the inequality $\sqrt{a+b}\le\sqrt{a}+\sqrt{b}$ for $a,b\ge 0$.
Finally, recall the inequality \eqref{eq2}, we obtain
\begin{equation}\label{eq-theta-hat-bound}
     \lVert (\bH^\lambda + \bE)^{-1}\by_t \rVert_{\bH^\lambda} \leq (1+\zeta)\left(8\sqrt{d\log(1/\delta)} + 8d\xi \log(1/\delta)\right). 
\end{equation}
Therefore, with probability at least $1-\delta,$
\begin{align*}
    |J_2| \leq 8\lVert\bm x\rVert_{(\bH^\lambda)^{-1}}\underbrace{\left(\frac{e^\zeta - 1}{\zeta} - 1\right)  \left(1+\zeta\right)}_{= \varphi(\zeta)}\left( \sqrt{d\log (1/\delta)} + d\xi\log(1/\delta)  \right).
\end{align*}

%-----------------Bound term J3---------------------%
%---------------------------------------------------%
\subsubsection*{The last term $J_3$}
We have already proved that
\begin{align*}
    \bH + \bE +\lambda \bI \succeq \frac{1}{1+\zeta} \bH +\lambda \bI \succeq \frac{1}{1+\zeta}\bH^\lambda .
\end{align*} 
Consequently,
\begin{align*}
|J_3| \leq    \lvert \lambda\bm x^\top (\bH^\lambda+\bE)^{-1}\bm\theta^\star\rvert &\leq \sqrt{\lambda}\lVert \bm\theta^\star \rVert \lVert\bm x\rVert_{(\bH^\lambda +\bE)^{-1}}\leq \sqrt{\lambda (1+\zeta)}W\lVert\bm x\rVert_{(\bH^\lambda)^{-1}}.
\end{align*}

%-----------------Combing all bounds---------------------%
%---------------------------------------------------%
\subsubsection*{Combining all bounds}
Now combining all above bounds, we get, with probability at least $1-\delta,$
\begin{align*}
    \frac{\lvert\bm x^\top (\hat{\bm\theta}^\lambda - \bm\theta^\star)\rvert}{8\lVert\bm x \rVert_{(\bH^\lambda)^{-1}}}  &\leq \left(\sqrt{\log(1/\delta)} + \xi  \log(1/\delta)\right)+ \varphi(\zeta)  \left( \sqrt{d\log (1/\delta)} + d\xi\log(1/\delta)  \right) + \sqrt{\lambda(1+\zeta)}W,
\end{align*}
as desired. 
\end{proof}

%---------------------------remark---------------------------%
%--------------------------------------------------------------%
\begin{remark}
  In the above proof of~\eqref{eq-appendix-tmp1}, we borrow the calculation in \cite{lee2024nearly}, which was used to verify the self-concordant-like property of the linear MNL-likelihood function. The key distinction in our approach is that we retain the $b_{mj}$ term throughout the calculation, whereas \cite{lee2024nearly} directly apply the bound $\lvert b_{mj}\rvert \lesssim \lVert \hat{\bm\theta}^\lambda - \bm\theta^\star\rVert_2.$ This difference allows us to derive a bound on $\bE$ that depends only on $\max_{m\le t,j\in S_m} \lvert\bm x_{mj}^\top (\hat{\bm\theta}^\lambda - \bm\theta^\star) \rvert$, rather than on $\lVert \hat{\bm\theta}^\lambda - \bm\theta^\star \rVert_2$. It is worth noting that a very recent work~\cite{lee2025improved} also uses a similar argument to establish a refined self-concordant-like property for the MNL likelihood function (specifically, the $\ell_\infty$ self-concordant-like property in their Proposition B.3) to obtain sharper confidence bounds, and the bound in~\eqref{eq-appendix-tmp1} can also be derived from their Proposition B.3.

  While the above proof primarily focuses on the relation between $\bH(\hat{\bm\theta}^\lambda)$ and $\bH(\bm\theta^\star)$, the inequality in~\eqref{eq: proof-to-H-dom} can also imply a dominance relation between $\bH(\bm\theta)$ and $\bH(\bm\theta^\star)$ for general $\bm\theta$. We summarize this general result in the following proposition for future applications.

\begin{proposition}\label{prop-H-dominance}
    Given any $\bm\theta\in \mathbb{R}^d,\lambda \geq 0,$ denoting $\zeta_{\bm\theta}:=3\sqrt{2}\max_{m\leq t, j\in S_m} \lvert\bm x_{mj}^\top (\bm\theta - \bm\theta^\star) \rvert,$ then it holds that \begin{align*}
  \frac{1}{1+2\varphi(\zeta_{\bm\theta})}  \bH(\bm\theta^\star)  \preceq  \bH(\bm\theta) \preceq  (1+2\varphi(\zeta_{\bm\theta}))\bH(\bm\theta^\star).
\end{align*}
\end{proposition}

\begin{proof}[\textbf{Proof of Proposition~\ref{prop-H-dominance}}]
    By taking $s = 1$ and taking summation over $m\le t$ in~\eqref{eq: proof-to-H-dom}, we can get \begin{align*}
    e^{-\zeta_{\bm\theta}}\lVert \bw \rVert_{\bH(\bm\theta^\star)}\leq    \lVert \bw \rVert_{\bH(\bm\theta)} \leq e^{\zeta_{\bm\theta}}\lVert \bw \rVert_{\bH(\bm\theta^\star)},
    \end{align*}
    for any unit vector $\bm w$, thus it holds that \begin{align*}
   e^{-\zeta_{\bm\theta}} \bH(\bm\theta^\star) \preceq     \bH(\bm\theta) \preceq e^{\zeta_{\bm\theta}} \bH(\bm\theta^\star).
    \end{align*}
    Now by \begin{align*}
       1+2\varphi(\zeta_{\bm\theta}) = 1+2(1+\zeta_{\bm\theta})\left(\frac{e^{\zeta_{\bm\theta}} - 1}{\zeta_{\bm\theta}} - 1 \right)= e^{\zeta_{\bm\theta}} + (e^{\zeta_{\bm\theta}} - 1 -\zeta_{\bm\theta}) + 2\left(\frac{e^{\zeta_{\bm\theta}} - 1}{\zeta_{\bm\theta}} - 1 - \frac{\zeta_{\bm\theta}}{2}\right) \geq e^{\zeta_{\bm\theta}},
    \end{align*}
    the claim holds.
\end{proof}
\end{remark}

Lemma~\ref{lem-confidence-step1} provides an upper bound on the ratio between $\bm x^\top(\hat{\bm\theta}^\lambda - \bm\theta^\star)$ and $\lVert\bm x \rVert_{(\bH^\lambda)^{-1}}$. However, the upper bound involves both $\xi$ and $\zeta$ simultaneously. Our next lemma shows that when $\xi$ is well-controlled, the $\zeta$ factor can be bounded by $\xi$.

%---------------------------Lemma B5---------------------------%
%--------------------------------------------------------------%
\begin{lemma}\label{lem-confidence-step2}
Under the condition 
\begin{align*}
  \xi \leq \min\left\{\frac{1}{144\sqrt{d\log(N/\delta)}}, \frac{1}{24\sqrt\lambda W}\right\},
\end{align*}
we have 
\begin{align}\label{eq-condition-in-proof}
 \varphi(\zeta)  \leq  \zeta \leq 32\sqrt{2}\xi \left(\sqrt{\log(N/\delta)}  + \xi \log(N/\delta) + \sqrt{2\lambda}W \right).
\end{align}
\end{lemma}

\begin{proof}[\textbf{Proof of Lemma~\ref{lem-confidence-step2}}]
Applying Lemma~\ref{lem-confidence-step1} with $\delta'$ to each $\bm x_{mj}$ and take a union bound with $\delta'=\delta/N$, we have then with probability at least $1-\delta,$ 
\begin{align*}
    \frac{\zeta}{3\sqrt{2}\cdot 8 \xi} \leq \left(\sqrt{\log(N/\delta)} + \xi  \log(N/\delta)\right)+ \varphi(\zeta)  \left( \sqrt{d\log (N/\delta)} + d\xi\log(N/\delta) \right) + \sqrt{\lambda(1+\zeta)}W.
\end{align*}
Now by the elementary inequality $$e^{x} \leq 1+x+\frac{5}{8}x^2,\quad \forall\bm x \in \left[0,\frac{3}{5}\right],$$ it holds that 
\begin{align*}
    \zeta\leq \frac{3}{5} \quad \implies \quad \varphi(\zeta) =(1+\zeta)(\frac{e^{\zeta} - 1}{\zeta} - 1)\leq \frac{8}{5}\cdot \frac{5\zeta}{8}\leq \zeta.
\end{align*}  
Then, when $\zeta\leq \frac{3}{5}$, the above inequality can be reduced to  \begin{align*}
      \zeta\left[1 - 24\sqrt{2}\xi\left( \sqrt{d\log (N/\delta)} + d \xi  \log(N/\delta) \right)\right] \leq  24\sqrt{2}\xi \left(\sqrt{\log(N/\delta)}  + \xi \log(N/\delta) + \sqrt{2\lambda}W \right).
\end{align*}
As a result, if the following holds simultaneously  
\begin{equation*}
\zeta \leq \frac{3}{5},\quad \xi\leq \frac{1}{144\sqrt{d\log(N/\delta)}},    
\end{equation*}
then $24\sqrt{2}\xi\left( \sqrt{d\log (N/\delta)} + d \xi  \log(N/\delta) \right)\le \frac{1}{4}$, thus
\begin{align*}
\varphi(\zeta)\le\zeta \leq 32\sqrt{2}\xi \left(\sqrt{\log(N/\delta)}  + \xi \log(N/\delta) + \sqrt{2\lambda}W \right) .
\end{align*}
On the other hand, by \eqref{eq-theta-hat-bound}, it holds that, with probability at least $1-\delta$, 
\begin{align*}
\zeta &\leq 3\sqrt{2}{\xi} \lVert \hat{\bm\theta}^\lambda - \bm\theta^\star \rVert_{\bH^\lambda_t}= 3\sqrt{2}{\xi} \lVert (\bH^\lambda + \bE)^{-1}(\by_t - \lambda \bm\theta^\star)\rVert_{\bH^\lambda}\\
&\leq 3\sqrt{2}\xi \left[(1+\zeta)\left( 8\sqrt{d\log(1/\delta)} + 8d\xi \log(1/\delta)\right) +\sqrt{\lambda}W \right]
\end{align*}
As a consequence, 
\begin{align*}
   & \xi \leq \min\left\{\frac{1}{144\sqrt{d\log(1/\delta)}}, \frac{1}{24\sqrt\lambda W}\right\}\\
    \implies& \zeta\leq \frac{24\sqrt{2}\xi\left(\sqrt{d\log(1/\delta)}+d\xi \log(1/\delta) \right)+ 3\sqrt{2}\xi\sqrt{\lambda}W}{1-24\sqrt{2}\xi\left(\sqrt{d\log(1/\delta)}+d\xi \log(1/\delta) \right) } \leq \frac{\frac{1}{4}+ \frac{1}{5} }{1-\frac{1}{4}} = \frac{3}{5}.
\end{align*}
This verifies that the condition $\xi \leq \min\left\{\frac{1}{144\sqrt{d\log(N/\delta)}}, \frac{1}{24\sqrt\lambda W}\right\}$ is sufficient to ensure that \eqref{eq-condition-in-proof} holds, as desired.
\end{proof}

\noindent Now, combining the results of Lemma~\ref{lem-confidence-step1} and Lemma~\ref{lem-confidence-step2}, we have $\xi \leq \min\left\{\frac{1}{144\sqrt{d\log(N/\delta)}}, \frac{1}{24\sqrt\lambda W}\right\}$ which implies that, with probability at least $1-\delta,$
\begin{align*}
    \frac{\lvert\bm x^\top (\hat{\bm\theta}^\lambda - \bm\theta^\star)\rvert}{8\lVert\bm x \rVert_{(\bH^\lambda)^{-1}} }
      &\leq   \varphi(\zeta)  \left( \sqrt{d\log (1/\delta)} + d\xi\log(1/\delta)  \right) + \sqrt{\lambda(1+\zeta)}W + \sqrt{\log(1/\delta)}+\xi \log(1/\delta) \\
      &\leq {32\sqrt{2}\xi} \left(\left(1+\frac{1}{144\sqrt{d}}\right)\sqrt{\log (N/\delta)}  + \sqrt{2\lambda} W \right) \cdot\big( \sqrt{d\log (1/\delta)} + d\xi\log(1/\delta)  \big) \\
      & + \sqrt{2\lambda}W + \sqrt{\log(1/\delta)} + \xi\log(1/\delta)\\
      &\leq  \left(\left(1+\frac{1}{144\sqrt{d}}\right)\sqrt{\log (N/\delta)} +\sqrt{2\lambda} W \right) \cdot \frac{1}{3} + \sqrt{2\lambda} W  + \left(1+\frac{1}{144\sqrt{d}}\right) \sqrt{\log(N/\delta)} \\
      &\leq 2\sqrt{\log(N/\delta)} + 2\sqrt{\lambda} W.
\end{align*}
Thus, with probability at least $1-\delta,$
$$|\bm x^{\top}(\hat{\bm\theta}^\lambda-\bm\theta^\star)| \leq 16 \|\bm x\|_{(\bH^\lambda)^{-1}} \left(\sqrt{\log (N/\delta)} + \sqrt{\lambda}W \right).
$$
This finishes the proof of the confidence bound result under the burn-in condition. Moreover, by $\varphi(\zeta)\leq \zeta \leq 3/5,$ Proposition~\ref{prop-H-dominance} implies
\begin{align*}
 \frac{1}{3}\bH^\lambda \preceq \frac{1}{1+2\zeta}\bH^\lambda  \preceq   \bH^\lambda + \bE \preceq (1+2\varphi(\zeta)) \bH^\lambda \preceq 3\bH^\lambda.
\end{align*}
This finishes the proof of Corollary~\ref{thm-sup-ucb-confidence-bound}.

%% ---- END appendix_section3.tex ----

%% ---- BEGIN appendix_perturbation.tex ----
\section{Perturbation Results for Revenue Functions}\label{sec:perturbation-pricing}

Before proving the results in Section~\ref{sec-offline} and Section~\ref{sec-online}, we first introduce several
perturbation bounds for the revenue function with respect to the utility vector $\bm u$.
Throughout this section, we treat the action $(S,\bm p)$ as fixed and study how the MNL revenue changes when $\bm u$ is perturbed.
To emphasize the dependence on $\bm u$, for any $S\subset[N]$ and price vector $\bm p\in[0,P]^N$, define
\begin{equation}\label{appendix-C: Q-u-def}
    Q(\bm u; S,\bm p)
    :=
    R(S,\bm p| e^{\bm u})
    =
    \frac{\sum_{j\in S} p_j e^{u_j}}{1+\sum_{j\in S} e^{u_j}}.
\end{equation}
We also write the induced MNL choice probabilities under utility $\bm u$ as
\[
q_i(S,\bm p|e^{\bm u})
:=
\frac{e^{u_i}}{1+\sum_{j\in S}e^{u_j}},
\ \forall i\in S,
\quad\text{and}\quad
q_0(S,\bm p|e^{\bm u})
:=
\frac{1}{1+\sum_{j\in S}e^{u_j}}.
\]
Note that $q_i(S,\bm p|e^{\bm u})$ depends on $(S,\bm p)$ only through the utility vector $\bm u$. Most importantly, in our model $\bm u$ is generated by the price-based features $\widetilde{\bX}(\bm p)$, it is itself price-dependent. Hence, all utility comparisons must be made under the same price vector $\bm p$.

First, we recall first-order and second-order derivative identities, originally proved in \citet{perivier2022dynamic} and refined in
\citet{lee2024nearly} (for the uniform case). The statements below allow non-identical prices $p_i$ and include a short
proof for completeness.

\begin{proposition}[Lemma~E.3 of \cite{lee2024nearly}]\label{appendix-prop-perturbation-of-Q}
Fix $\bm p\in[0,P]^N$ and $S\subset[N]$, for any $\bm u,\bm w\in\mathbb R^N$, it follows that
\begin{align}
\label{ch4-eq: Q-first-order}
&\langle\nabla Q(\bm u;S,\bm p),\bm w\rangle
=
\sum_{i\in S}
q_i(S,\bm p| e^{\bm u})\big(p_i-Q(\bm u;S,\bm p)\big)w_i,
\\
\label{ch4-eq: Q-second-order}
&\left|\bm w^\top\nabla^2Q(\bm u;S,\bm p)\bm w\right|
\le
3P\cdot \max_{i\in S} w_i^2.
\end{align}
\end{proposition}

\begin{proof}[\textbf{Proof of Proposition~\ref{appendix-prop-perturbation-of-Q}}]
Fix $(S,\bm p)$ and write $Q(\bm u;S,\bm p)$ as $Q(\bm u)$ for brevity.
To prove~\eqref{ch4-eq: Q-first-order}, note that $\partial_{u_i}Q(\bm u)=0$ for $i\notin S$ and for $i\in S$,
\begin{align*}
\partial_{u_i}Q(\bm u)
&=
\frac{p_ie^{u_i}}{1+\sum_{j\in S}e^{u_j}}
-
\frac{e^{u_i}\sum_{j\in S}p_je^{u_j}}{\big(1+\sum_{j\in S}e^{u_j}\big)^2}
\\
&=
q_i(S,\bm p| e^{\bm u})\big(p_i-Q(\bm u)\big).
\end{align*}
Therefore,
\[
\langle\nabla Q(\bm u),\bm w\rangle
=
\sum_{i\in S}\partial_{u_i}Q(\bm u)\,w_i
=
\sum_{i\in S}q_i(S,\bm p| e^{\bm u})\big(p_i-Q(\bm u)\big)w_i.
\]
To prove~\eqref{ch4-eq: Q-second-order}, note that $\partial_{u_i u_j}^2Q(\bm u)=0$ if $i\notin S$ or $j\notin S$.
For $i,j\in S$,
\begin{align*}
\partial_{u_j}\partial_{u_i}Q(\bm u)
&=
\partial_{u_j}\Big(q_i(S,\bm p| e^{\bm u})\big(p_i-Q(\bm u)\big)\Big)
\\
&=
\partial_{u_j}q_i(S,\bm p| e^{\bm u})\big(p_i-Q(\bm u)\big)
-
q_i(S,\bm p| e^{\bm u})q_j(S,\bm p| e^{\bm u})\big(p_j-Q(\bm u)\big).
\end{align*}
Also, for $i,j\in S$,
\[
\partial_{u_j} q_i(S,\bm p| e^{\bm u})=
\begin{cases}
-q_i(S,\bm p| e^{\bm u})q_j(S,\bm p| e^{\bm u}) & j\neq i,\\
(1-q_i(S,\bm p| e^{\bm u}))q_i(S,\bm p| e^{\bm u}) & j=i.
\end{cases}
\]
Hence, for $i\in S$,
\[
\partial^2_{u_i}Q(\bm u)
=
q_i(S,\bm p| e^{\bm u})\big(1-2q_i(S,\bm p| e^{\bm u})\big)\big(p_i-Q(\bm u)\big),
\]
and for $i\neq j$ with $i,j\in S$,
\[
\partial_{u_j}\partial_{u_i}Q(\bm u)
=
-q_i(S,\bm p| e^{\bm u})q_j(S,\bm p| e^{\bm u})\big(p_i+p_j-2Q(\bm u)\big).
\]
Therefore,
\begin{align*}
\left|\bm w^\top\nabla^2Q(\bm u)\bm w\right|
&\le
\sum_{i,j\in S}|w_iw_j|\cdot\left|\partial_{u_i u_j}^2Q(\bm u)\right|
\\
&\le
P\sum_{i\in S}w_i^2 q_i(S,\bm p| e^{\bm u})
+
2P\sum_{i\neq j,\ i,j\in S}|w_iw_j|\,q_i(S,\bm p| e^{\bm u})q_j(S,\bm p| e^{\bm u})
\\
&\le
P\sum_{i\in S}w_i^2 q_i(S,\bm p| e^{\bm u})
+
P\sum_{i\neq j,\ i,j\in S}(w_i^2+w_j^2)q_i(S,\bm p| e^{\bm u})q_j(S,\bm p| e^{\bm u})
\\
&\le
3P\sum_{i\in S}w_i^2 q_i(S,\bm p| e^{\bm u})
\le
3P\max_{i\in S}w_i^2.
\end{align*}
\end{proof}

 Based on Proposition~\ref{appendix-prop-perturbation-of-Q}, we now present perturbation bounds for $Q$.
%--------------------------------------------%
\begin{proposition}\label{prop: assortment-perturbation}
The following perturbation bounds hold:
\begin{enumerate}
\item[(i)] For any $(S,\bm p)\in\cS_K\times[0,P]^N$ and $\bm u,\bm u'\in\mathbb R^N$, let $\bm w:=\bm u'-\bm u$. Then, 
\begin{align*}
\left|Q(\bm u';S,\bm p)-Q(\bm u;S,\bm p)\right|
&\le
\sqrt{\sum_{j\in S}e^{u_j}\left|p_j-Q(\bm u;S,\bm p)\right|^2}
\cdot
\sqrt{\sum_{j\in S}q_j(S,\bm p| e^{\bm u})q_0(S,\bm p| e^{\bm u})\,w_j^2}
+\frac{3}{2}P\cdot \max_{j\in S} w_j^2.
\end{align*}

\item[(ii)]
Let $(\tilde S,\tilde{\bm p})\in\argmax_{(S,\bm p)\in\cS_K\times[0,P]^N}R(S,\bm p| e^{\tilde{\bm u}})$ for some $\tilde{\bm u}\ge \bm u$
(elementwise), and define $\tilde{\bm w}:=\tilde{\bm u}-\bm u$. Then,
\[
\left|Q(\tilde{\bm u};\tilde S,\tilde{\bm p})-Q(\bm u;\tilde S,\tilde{\bm p})\right|
\le
P\sqrt{\sum_{j\in \tilde S}q_j(\tilde S,\tilde{\bm p}| e^{\tilde{\bm u}})q_0(\tilde S,\tilde{\bm p}| e^{\tilde{\bm u}})\tilde w_j^2}
+\frac{3}{2}P\cdot\max_{j\in \tilde S}\tilde w_j^2.
\]
\end{enumerate}
\end{proposition}

\begin{proof}[\textbf{Proof of Proposition~\ref{prop: assortment-perturbation}}]
By Proposition~\ref{appendix-prop-perturbation-of-Q}, there exists $\xi\in(0,1)$ such that
\begin{align*}
\left|Q(\bm u';S,\bm p)-Q(\bm u;S,\bm p)-\langle\nabla Q(\bm u;S,\bm p),\bm w\rangle\right|
=
\left|\frac{1}{2}\bm w^\top\nabla^2Q(\bm u+\xi\bm w;S,\bm p)\bm w\right|
\le
\frac{3}{2}P\cdot\max_{j\in S}w_j^2.
\end{align*}
On the other hand,
\begin{align*}
\left|\langle\nabla Q(\bm u;S,\bm p),\bm w\rangle\right|
&\le
\sum_{j\in S}q_j(S,\bm p| e^{\bm u})\left|p_j-Q(\bm u;S,\bm p)\right|\cdot |w_j|
\\
&\le
\sqrt{\sum_{j\in S}e^{u_j}\left|p_j-Q(\bm u;S,\bm p)\right|^2}\cdot
\sqrt{\sum_{j\in S}\frac{e^{u_j}w_j^2}{\left(1+\sum_{j\in S}e^{u_j}\right)^2}}
\\
&=
\sqrt{\sum_{j\in S}e^{u_j}\left|p_j-Q(\bm u;S,\bm p)\right|^2}\cdot
\sqrt{\sum_{j\in S}q_j(S,\bm p| e^{\bm u})q_0(S,\bm p| e^{\bm u})w_j^2},
\end{align*}
which proves statement (i).

% For (ii), if $p_j\equiv P$ then
% \[
% p_j-Q(\bm u;S,\bm p)
% =
% P\left(1-\frac{\sum_{j\in S}e^{u_j}}{1+\sum_{j\in S}e^{u_j}}\right)
% =
% P\cdot q_0(S,\bm p| e^{\bm u}),
% \]
% and hence
% \[
% \left|\langle\nabla Q(\bm u;S,\bm p),\bm w\rangle\right|
% \le
% \sum_{j\in S}q_j(S,\bm p| e^{\bm u})\left|p_j-Q(\bm u;S,\bm p)\right||w_j|
% =
% P\sum_{j\in S}q_j(S,\bm p| e^{\bm u})q_0(S,\bm p| e^{\bm u})|w_j|.
% \]

For statement (ii), apply (i) to $(\tilde S,\tilde{\bm p})$ and upper bound
$\sum_{j\in \tilde S}e^{\tilde u_j}\left|\tilde{p}_j-Q(\tilde{\bm u};\tilde S,\tilde{\bm p})\right|^2$ by applying Proposition~\ref{prop-non-uniform-reward-optimal-assortment}. The first statement of Proposition~\ref{prop-non-uniform-reward-optimal-assortment}
implies $\tilde p_j\ge Q(\tilde{\bm u};\tilde S,\tilde{\bm p})$ for all $j\in\tilde S$.
Moreover, let $(\bar S,\bar{\bm p})\in\argmax_{(S,\bm p)\in\cS_K\times[0,P]^N}Q(\bm u;S,\bm p)$, then by the second statement of Proposition~\ref{prop-non-uniform-reward-optimal-assortment}, we have 
$$\max_{(S,\bm p)\in\cS_K\times[0,P]^N}Q(\bu; S,\bm p)=Q(\bu;\bar{S},\bar{\bm p}) \leq Q(\tilde{\bu};\bar{S},\bar{\bm p})\leq  \max_{(S,\bm p)\in\cS_K\times[0,P]^N}Q(\tilde \bu; S,\bm p),$$
which implies that
\[
Q(\bm u;\tilde S,\tilde{\bm p})
\le
\max_{(S,\bm p)\in\cS_K\times[0,P]^N}Q(\bm u;S,{\bm p})
\le
\max_{(S,\bm p)\in\cS_K\times[0,P]^N}Q(\tilde{\bm u};S,{\bm p})
=
Q(\tilde{\bm u};\tilde S,\tilde{\bm p}).
\]
Finally, since
\begin{equation}\label{eq: characterization-of-Q}
Q(\tilde{\bm u};\tilde S,\tilde{\bm p})
=
\frac{\sum_{j\in \tilde S}p_je^{\tilde u_j}}{1+\sum_{j\in \tilde S}e^{\tilde u_j}}
\quad\Longrightarrow\quad
\sum_{j\in \tilde S}e^{\tilde u_j}\big(p_j-Q(\tilde{\bm u};\tilde S,\tilde{\bm p})\big)=Q(\tilde{\bm u};\tilde S,\tilde{\bm p}),
\end{equation}
we obtain
\begin{align*}
\sum_{j\in \tilde S}e^{\tilde u_j}\left|p_j-Q(\tilde{\bm u};\tilde S,\tilde{\bm p})\right|^2
&\le
P\sum_{j\in \tilde S}e^{\tilde u_j}\big(p_j-Q(\tilde{\bm u};\tilde S,\tilde{\bm p})\big)
\\
&=
P\cdot Q(\tilde{\bm u};\tilde S,\tilde{\bm p})
\le
P^2.
\end{align*}
This completes the proof.
\end{proof}

%--------------------------------------------%
\begin{proposition}[Lemma~A.3 in \cite{agrawal2019mnl}]\label{prop-non-uniform-reward-optimal-assortment}
For any utility vector $\bm u$, define the optimal assortment-price pair given $\bm u$ as
$(\bar S,\bar{\bm p})\in\argmax_{(S,\bm p)\in\cS_K\times[0,P]^N}Q(\bm u;S,\bm p)$. Then the following properties hold:
\begin{enumerate}
\item[(i)]  $\bar p_i\ge Q(\bm u;\bar S,\bar{\bm p})$ for all $i\in \bar S$.
\item[(ii)]  For any $\bm u'\ge \bm u$ (elementwise), $Q(\bm u;\bar S,\bar{\bm p})\le Q(\bm u';\bar S,\bar{\bm p})$.
\end{enumerate}
\end{proposition}

\begin{proof}[\textbf{Proof of Proposition~\ref{prop-non-uniform-reward-optimal-assortment}}]
Let $v_i := e^{u_i}$ and $v'_i := e^{u'_i}$, and denote the optimal expected revenue as $\bar Q := Q(\bm u;\bar S,\bar{\bm p})$. For convenience, let $\bar V := \sum_{j\in \bar S} v_j$ and $\bar A := \sum_{j\in \bar S} \bar p_j v_j$, which implies $\bar Q = \frac{\bar A}{1+\bar V}$.

To prove statement (i), suppose for contradiction that there exists some $i\in\bar S$ where $\bar p_i < \bar Q$. Consider the restricted assortment $S' := \bar S \setminus \{i\}$ while keeping the prices unchanged. The expected revenue becomes:
$$Q(\bm u;S',\bar{\bm p}) = \frac{\bar A-\bar p_i v_i}{1+\bar V-v_i}.$$
Because $\bar p_i < \bar Q = \frac{\bar A}{1+\bar V}$, by simply calculation, we have $Q(\bm u;S',\bar{\bm p}) > \bar Q$, which directly contradicts the optimality of $(\bar S,\bar{\bm p})$. Hence, we must have $\bar p_i \ge \bar Q$ for all $i\in\bar S$.

To prove statement (ii), consider $\bm u' \ge \bm u$ elementwise, which implies $v'_i \ge v_i$. We first write $\bar Q$ as:
$$\sum_{i\in\bar S}(\bar p_i-\bar Q)v_i = \bar A - \bar Q\bar V = \bar Q.$$
From the first statement, we know that $\bar p_i-\bar Q\ge 0$ for all $i\in\bar S$. It follows that
$$\sum_{i\in\bar S}(\bar p_i-\bar Q)v'_i \ge \sum_{i\in\bar S}(\bar p_i-\bar Q)v_i = \bar Q.$$
Rearranging this inequality yields
$$\begin{aligned}
\sum_{i\in\bar S}\bar p_i v'_i - \bar Q\Big(1+\sum_{i\in\bar S} v'_i\Big) &= \sum_{i\in\bar S}(\bar p_i-\bar Q)v'_i - \bar Q \ge 0,
\end{aligned}$$
which immediately implies
$$\frac{\sum_{i\in\bar S}\bar p_i v'_i}{1+\sum_{i\in\bar S} v'_i} \ge \bar Q.$$
Thus, $Q(\bm u';\bar S,\bar{\bm p})\ge Q(\bm u;\bar S,\bar{\bm p})$, as we desired.
\end{proof}

\section{Proof of Results in Section~\ref{sec-offline}}
Throughout this section, we simplify the notation for the fixed target feature by letting $\bx_i := (\xoff)_i$ and $(S^\star, \bm{p}^\star) := (S^\star_{\text{off}}, \bm{p}_{\text{off}}^\star)$. For brevity, we also suppress the explicit feature dependence in the revenue and choice models, denoting $R(S, \bm{p} | \bm\vartheta) := R(S | \exp(\widetilde{\bX}_{\text{off}}(\bm{p})^\top \bm\vartheta))$ and $q_i(S, \bm{p} | \bm\vartheta) := q_i(S | \exp(\widetilde{\bX}_{\text{off}}(\bm{p})^\top \bm\vartheta))$.
%--------------------------------------------%
\subsection{Proof of Theorem \ref{thm-price-sub-optimality}}

\begin{proof}[\textbf{Proof of Theorem \ref{thm-price-sub-optimality}}]
    Recalling the notation introduced in Algorithm~\ref{alg:pricing-mnl-lcb}, we define the pessimistic utility and the corresponding expected revenue as:
    \begin{align*}
    &u_i^{\mathrm{LCB}}(p) :=   \widetilde{\bx}_i(p)^\top \hat{\bm\vartheta}_{\widetilde \cD}  - 16\sqrt3\lVert \widetilde{\bx}_i(p)\rVert_{\bH_{\widetilde \cD}(\hat{\bm\vartheta}_{\widetilde \cD})^{-1}}\sqrt{\log(N/\delta)},\\
     & R^\text{LCB}(S,\bm p):= \frac{\sum_{i\in S}p_i\exp(u_i^\text{LCB}(p_i))}{1+\sum_{i\in S}\exp(u_i^\text{LCB}(p_i))}.
    \end{align*}
    By Theorem~\ref{prop-sup-ucb-confidence-bound}, with probability at least $1-\delta$, it holds that $u_i^{\mathrm{LCB}}(p)\ \le\ \widetilde{\bx}_i(p)^\top \bm\vartheta^\star $ for any $p\in[0,P]$. Let $w_j=\widetilde{\bx}_j(p^\star_j)^\top \bm\vartheta^\star - u_j^\mathrm{LCB}(p^\star_j) .$ Conditioned on this high-probability event, we can bound the suboptimality as follows:
    \begin{align*}
    \SubOpt(S^\mathrm{LCB},\bm p^\mathrm{LCB}, \Xoff) &= R(S^\star,\bm p^\star\lvert \bm \vartheta^\star) - R(S^\mathrm{LCB},\bm p^\mathrm{LCB}\lvert \bm \vartheta^\star)\\
    & \leq  R(S^\star,\bm p^\star\lvert \bm \vartheta^\star) - R^\mathrm{LCB}(S^\mathrm{LCB},\bm p^\mathrm{LCB}) \\&\leq  R(S^\star,\bm p^\star\lvert \bm \vartheta^\star) - R^\mathrm{LCB}(S^\star,\bm p^\star) \\
    &\leq  P\sqrt{\sum_{j\in S^\star}q_j(S^\star, \bm p^\star|\bm\vartheta^\star)q_0(S^\star,\bm p^\star|\bm\vartheta^\star) w_j^2} + \frac{3}{2}P\cdot \max_{j\in S^\star} w_j^2 .
\end{align*}
The first inequality follows directly from Proposition~\ref{prop-non-uniform-reward-optimal-assortment}. The second inequality holds because $(S^\mathrm{LCB},\bm p^\mathrm{LCB})$ maximizes the pessimistic revenue, implying $R^\mathrm{LCB}(S^\star,\bm p^\star) \leq R^\mathrm{LCB}(S^\mathrm{LCB},\bm p^\mathrm{LCB})$. The final inequality uses statement (ii) of Proposition~\ref{prop: assortment-perturbation}, alongside the fact that $p^\star_j \geq R(S^\star,\bm p^\star | \bm \vartheta^\star)$ for all $j\in S^\star$. Finally, applying Theorem~\ref{prop-sup-ucb-confidence-bound} to bound $w_j \le 64 \lVert \widetilde{\bx}_j( p^\star_j) \rVert_{\bH_{\widetilde\cD}({\bm\vartheta}^\star)^{-1}}\sqrt{\log(N/\delta)}$ completes the proof.

\end{proof}

%-----------------------proof of  Prop 4.4-----------------------------------------%
\subsection{Details of Proposition~\ref{prop-plug-in-confidence-region}}\label{sec-appendix-offline-burn-in-free}

In this section, we provide the detail on how to plug other confidence region results in Algorithm~\ref{alg:pricing-mnl-lcb} to achieve a burn-in-free offline learning guarantee. In \citet{perivier2022dynamic}, confidence region results with the radius $\tilde{O}(\sqrt{d}\lVert \bx_j \rVert_{\bH_{\cD}(\bm\theta^\star)^{-1}})$ are available\footnote{It should be noted that the original result in \citet{perivier2022dynamic} includes an additional $K$-dependency due to the self-concordant coefficient they established for the MNL likelihood function. This coefficient was later refined by \citet{lee2024nearly}, and incorporating their result eliminates the $K$-dependency in \citet{perivier2022dynamic}.}, and we take the result in \citet{perivier2022dynamic} to establish the confidence region for the price-based model. Here we first restate the confidence region bound of \citet{perivier2022dynamic} under a fixed price:

\begin{lemma}[Proposition~3.3 and Lemma~C.4 in \citet{perivier2022dynamic}]\label{lem-non-convex-confidence-region} Let $ \lambda = 2d\bar P/W$. For any fixed $p\in[0,P]$, we denote 
\begin{align}\label{eq-non-convex-confidence-region}
    \hat{\bm\vartheta}_{\cD,j}^\lambda(p):= \operatorname*{argmin}_{\|\bm\vartheta\|_2\le W}\{\bx_j(p)^\top \bm\vartheta: \lVert \nabla \ell_\cD^\lambda(\bm\vartheta) - \nabla \ell_\cD^\lambda(\hat{\bm\vartheta}_\cD^\lambda)\rVert_{\bH^\lambda_{\cD}(\bm\vartheta)^{-1}} \leq 4\sqrt{2d(1+\bar PW)} \log(K\bar PWn/d^2\delta) \}.
\end{align} 
Then, with probability at least $1-\delta,$ it holds that
$$\lVert \hat{\bm\vartheta}_{\cD,j}^\lambda(p) -\bm\vartheta^\star \rVert_{\bH^\lambda_{\cD}(\bm\vartheta^\star)} \le 4(1+\sqrt{6K}\bar PW)\sqrt{2d(1+\bar PW)} \log(K\bar PWn/d^2\delta).$$
\end{lemma}

\begin{proof}[\textbf{Proof of Proposition~\ref{prop-plug-in-confidence-region}}]
Let $\xi_i:=4(1+\sqrt{6K}\bar PW)\sqrt{2d(1+\bar PW)} \log(K\bar PWn/d^2\delta)$ and assign the LCB utility function for each $j\in [N]$ as\begin{align*}
    \tilde{u}^\text{LCB}_j(p):= \widetilde\bx_j(p)^\top \hat{\bm\vartheta}_{\cD,j}^\lambda(p) - \xi_i.
\end{align*} 
It can be seen from Lemma~\ref{lem-non-convex-confidence-region} that with probability at least $1-\delta,$
$$0\leq 
\widetilde\bx_j(p)^\top\bm\vartheta^\star - \tilde{u}_j^\text{LCB} (p)\lesssim\sqrt{dK\bar P^3W^3}\log(K\bar PWn/d^2\delta)\lVert \widetilde\bx_j(p)\rVert_{\bH_{\cD}^\lambda(\bm\vartheta^\star)^{-1} }.$$
Following the proof of Theorem~\ref{thm-price-sub-optimality} and letting $ w_j= \widetilde\bx_j(p^\star_j)^\top\bm\vartheta^\star - \tilde{u}_j^\text{LCB} (p^\star_j)$, we obtain  
\begin{align*}
    \SubOpt(S^\mathrm{LCB}, \bm p^\mathrm{LCB}, \Xoff)  &\leq  P\sqrt{\sum_{j\in S^\star}q_j(S^\star, \bm p^\star|\bm\vartheta^\star)q_0(S^\star,\bm p^\star|\bm\vartheta^\star) w_j^2} + \frac{3}{2}P\cdot \max_{j\in S^\star} w_j^2 \\
    &\lesssim \log(K\bar PWn/d^2\delta)\sqrt{dK\bar P^5W^3\sum_{j\in S^\star}q_j(S^\star, \bm p^\star|\bm\vartheta^\star)q_0(S^\star,\bm p^\star|\bm\vartheta^\star)  \lVert \widetilde\bx_j(p_j^\star)\rVert^2_{\bH_{\cD}^\lambda(\bm\vartheta^\star)^{-1} }} \\
    &\quad + dK\bar P^4W^3 \log^2(K\bar PWn/d^2\delta)\cdot \max_{j\in S^\star} \lVert \widetilde\bx_j(p_j^\star)\rVert^2_{\bH_{\cD}^\lambda(\bm\vartheta^\star)^{-1} }.
\end{align*}
\end{proof}

%--------------------------------------------%
\subsection{Proof of Proposition~\ref{prop-leading-lcb-upper-bound} and Corollary~\ref{corollary-subopt-via-itemcover}}

Following the notation defined in Section~\ref{sec4.2}, and since $\Xoff$ is fixed, we denote $(\xoff)_i$ by $\bx_i$ for simplicity.

\begin{proof}[\textbf{Proof of Proposition~\ref{prop-leading-lcb-upper-bound}}]
For each $j\in S^\star$ and any $\lambda > 0,$ assume W.L.O.G. that $j$  appears exactly in the sets $S_1,\dots,S_{n_j}$. Since
\begin{align*}
     \bH_{\cD}(\bm\theta^\star)+\lambda \bI &\succeq \sum_{m = 1}^n\sum_{k \in S_m} q_k(S_m|\bm\theta^\star)q_0(S_m|\bm\theta^\star) \bx_k\bx_k^\top + \lambda \bI\\
     &\succeq  \underbrace{\sum_{m > n_j}^n\sum_{k \in S_m} q_k(S_m|\bm\theta^\star)q_0(S_m|\bm\theta^\star) \bx_k\bx_k^\top + \lambda \bI}_{: = \bZ_j} + \underbrace{\sum_{m \leq n_j} q_j(S_m|\bm\theta^\star) q_0(S_m|\bm\theta^\star)}_{: = \gamma_j} \bx_j\bx_j^\top,
 \end{align*}
and by the Sherman–Morrison formula, we have
 \begin{align*}
        \lVert \bx_j \rVert_{(\bH_{\cD}(\bm\theta^\star)+\lambda \bI)^{-1}}^2 &= \bx_j^\top\left(\bZ_j + \gamma_j \bx_j\bx_j^\top \right)^{-1}\bx_j = \bx_j^\top \left(\bZ_j^{-1} + \frac{\gamma_j  \bZ_j^{-1}\bx_j\bx_j^\top \bZ_j^{-1}}{1+ \gamma_j \bx_j^\top \bZ_j^{-1}\bx_j } \right) \bx_j \\
        &=  \bx_j^\top \bZ_j^{-1} \bx_j \left( 1 - \frac{\gamma_j \bx_j^\top \bZ_j^{-1} \bx_j}{1+\gamma_j \bx_j^\top \bZ_j^{-1} \bx_j} \right) = \frac{\bx_j^\top \bZ_j^{-1} \bx_j}{1 + \gamma_j \bx_j^\top \bZ_j^{-1} \bx_j}  \leq \frac{1}{\gamma_j}.
\end{align*}
Taking the limit as $\lambda_j\to 0$ and using the continuity of  $\lVert \bx_j \rVert_{(\bH_{\cD}(\bm\theta^\star)+\lambda \bI)^{-1}}^2$, we obtain $\lVert \bx_j \rVert_{\bH_{\cD}(\bm\theta^\star)^{-1}}^2 \leq \gamma_j^{-1}.$\\
Therefore, since $S^\star\in\cS_k$, we conclude that
\begin{align*}
    \sum_{j\in S^\star} q_j(S^\star|\bm\theta^\star) q_0(S^\star|\bm\theta^\star) \lVert \bx_j \rVert^2_{\bH_{\cD}(\bm\theta^\star)^{-1}}&\leq  \sum_{j\in S^\star} q_j(S^\star|\bm\theta^\star) q_0(S^\star|\bm\theta^\star) \gamma_j^{-1}\\
    &= \sum_{j\in S^\star} \frac{v_j}{(1+\sum_{k\in S^\star}v_k)^2} \cdot \left[\sum_{m \leq n_j} \frac{v_j}{(1+\sum_{k\in S_m} v_k)^2} \right]^{-1} \\
    &\leq \sum_{j\in S^\star} \frac{v_j}{(1+\sum_{k\in S^\star}v_k)^2} \cdot \frac{(1+\sum_{k\in S_m} v_k)^2}{n_j v_j}\\
    &\leq  \frac{K}{n^\star} \cdot \frac{   \max_{1\le m \le n}(1+\sum_{k\in S_m} v_k)^2}{(1+\sum_{k\in S^\star}v_k)^2}.
\end{align*}
This finishes the proof of Proposition~\ref{prop-leading-lcb-upper-bound}.
\end{proof}

%-----------------------proof of  Cor 4.3-----------------------------------------%
\begin{proof}[\textbf{Proof of Corollary~\ref{corollary-subopt-via-itemcover}}]
This corollary follows directly from Corollary~\ref{thm-linear-mnl-sub-optimality} and Proposition~\ref{prop-leading-lcb-upper-bound}, together with
\begin{align*}
    \lVert \bx_j \rVert_{\bH_{\cD}(\bm\theta^\star)^{-1}}^2 \leq \gamma_j^{-1} = \left[\sum_{m \leq n_j} q_j(S_m|\bm\theta^\star) q_0(S_m|\bm\theta^\star)\right]^{-1} \le \Koff^{-1} n_j^{-1}.
\end{align*}

\end{proof}

%% ---- END appendix_section4.tex ----

%% ---- BEGIN appendix_section5.tex ----
\newcommand{\UCB}{\mathrm{UCB}}
\newcommand{\x}{\bm{x}}
\newcommand{\mH}{\bm{H}}
\newcommand{\mE}{\bm{E}}
\newcommand{\bmtheta}{\bm{\vartheta}}

\section{Proof of Results in Section~\ref{sec-online}}

%-------------------------------------------------%
\subsection{Proof of Lemma~\ref{lem-initial-length-bound}}

\begin{proof}[\textbf{Proof of Lemma~\ref{lem-initial-length-bound}}]
Use $\cT_\ell$ to denote all the rounds that enter the initial exploration phase in Algorithm~\ref{alg:sup-lin-mnl-pricing}. In the fixed design setting, $\Psi_\ell$ are the same for all $\ell\in[J]$ after each round $t\in[\tau]$. Specifically, $\cT_\ell=[\tau]$.
By the definition of $\bV_{t,\ell}$, we have
\begin{align*}
   \kappa \bV_{t,\ell} + \lambda \bI = \kappa \bV_{t-1,\ell} + \lambda \bI + \kappa \sum_{t\in \Psi_{\ell},k\in S_t}\widetilde{\bx}_{k}(p_{tk})\widetilde{\bx}_{k}(p_{tk})^\top.
\end{align*}
where $\sum_{t\in \Psi_{\ell},k\in S_t}\widetilde{\bx}_{k}(p_{tk})\widetilde{\bx}_{k}(p_{tk})^\top=\bm{1}\{t\in \Psi_{\ell}\}\sum_{k\in S_t}\widetilde{\bx}_{k}(p_{tk})\widetilde{\bx}_{k}(p_{tk})^\top$.
This implies that
\begin{align*}
    \text{det}\left(\kappa \bV_{t,\ell} + \lambda \bI\right) &\ge \text{det}\left(\kappa \bV_{t-1,\ell} + \lambda \bI\right) \left(1+\kappa \sum_{t\in \Psi_{\ell},k\in S_t} \|\widetilde{\bx}_{k}(p_{tk})\|^2_{(\kappa \bV_{t,\ell} + \lambda \bI)^{-1}} \right)\\
    & \ge \text{det}\left(\kappa \bV_{t-1,\ell} + \lambda \bI\right) \left(1+\kappa \sum_{t\in \Psi_{\ell}} \max_{j\in[N]}\|\widetilde{\bx}_{j}(p_{tj})\|^2_{(\kappa \bV_{t,\ell} + \lambda \bI)^{-1}} \right),
\end{align*}
where the second inequality holds because $S_t=\left\{\operatorname{argmax}_{j\in[N]}\lVert\widetilde{\bx}_{j}(p_{tj}) \rVert_{\left(\kappa\bV_{t,\ell}+\lambda\bI\right)^{-1}}\right\}$ for every $t\in\Psi_{\ell}$.

Since $\lambda\ge \bar P$, we have $\kappa \max_{j\in[N]}\|\widetilde{\bx}_{j}(p_{tj})\|^2_{(\kappa \bV_{t,\ell} + \lambda \bI)^{-1}} \le \frac{\kappa}{\lambda}\|\widetilde{\bx}_{j}(p_{tj})\|^2 \le \kappa\le 1$ for all $t\in[T]$. Then, using the fact that $z\le 2\log(1+z)$ for any $z\in[0,1]$, we get
\begin{align*}
     \sum_{t\le\tau} \max_{j\in [N]} \lVert\widetilde{\bx}_{j}(p_{tj})\rVert^2_{(\kappa \bV_{t,\ell} + \lambda \bI)^{-1}} &\le \frac{2}{\kappa}\sum_{t\le\tau}\log\left(1+\kappa\max_{j\in [N]} \lVert\widetilde{\bx}_{j}(p_{tj})\rVert^2_{(\kappa \bV_{t,\ell} + \lambda \bI)^{-1}}\right)\\
    &\le \frac{2}{\kappa} \sum_{t\le\tau}\log\left(\frac{\text{det}\left(\kappa \bV_{t,\ell} + \lambda \bI\right)}{\text{det}\left(\kappa \bV_{t-1,\ell} + \lambda \bI\right)}\right)\\
    &\le \frac{2}{\kappa} \log\left(\frac{\text{det}\left(\kappa \bV_{T,\ell} + \lambda \bI\right)}{\text{det}\left(\kappa \bV_{0,\ell} + \lambda \bI\right)}\right)\\
    &\le \frac{2}{\kappa}2d \log\left(\frac{\text{tr}\left(\kappa \bV_{T,\ell} + \lambda \bI\right)}{2d\lambda} \right) \\
    &\le \frac{4}{\kappa}d \log\left(1+\frac{\kappa T\bar P^2}{2d\lambda} \right).
\end{align*}
On the other hand, we have
\begin{align*}
    \sum_{t\le\tau} \max_{j\in [N]} \lVert\widetilde{\bx}_{j}(p_{tj})\rVert^2_{(\kappa \bV_{t,\ell} + \lambda \bI)^{-1}} \ge \left( \frac{1}{144\sqrt{2d\log(NT)}}\wedge\frac{1}{24\sqrt{\lambda}W} \right)^2 \cdot \tau .
\end{align*}
Combining them, we can derive that
\begin{align*}
    \tau\le \frac{4}{\kappa}d \log\left(1+\frac{\kappa T\bar P^2}{2d\lambda} \right) \left(144^22d\log(NT)\vee 24^2\lambda W^2\right)
\end{align*}
\end{proof}

%----------------------proof of prop5.2--------------------------%
\subsection{Proof of Proposition~\ref{prop-optimal-elimination}}

\begin{proof}[\textbf{Proof of Proposition~\ref{prop-optimal-elimination}}]
For every $\tau\le t\le T$, we can prove this result by induction. First, the claim holds for $\ell = 1$ since $\mathcal{A}_1 = \mathcal{S}_K \times\mathcal{P}^N.$ Now suppose by induction that the result holds for general $\ell \geq 1$ and the algorithm enters the step~(c) of $\ell$-th loop. Let $(\hat S_{t,\ell},\hat{\bm p}_{t,\ell}):= \text{argmax}_{\mathcal{A}_\ell}R^\UCB_{t,\ell}(S, \bm p)$. By induction we have $(S_t^\star, \bm p_t^\star) \in \mathcal{A}_{\ell}$, it follows that
\begin{align}\label{eq:rev-UCB-relation}
    R^\mathrm{UCB}_{t,\ell}(\hat S_{t,\ell},\hat{\bm p}_{t,\ell}) \geq R^\mathrm{UCB}_{t,\ell}(S_t^\star, \bm p_t^\star) \geq R(S_t^\star, \bm p_t^\star \lvert \bm \vartheta^\star), 
\end{align}
where the last inequality is by the monotone property at optimal assortment (Lemma~\ref{prop-non-uniform-reward-optimal-assortment}). 

On the other hand, we cannot directly apply Proposition~\ref{prop: assortment-perturbation} to obtain a perturbation bound for $(\hat S_{t,\ell},\hat{\bm p}_{t,\ell})$, since it maximizes the optimistic revenue over an \textit{unstructured set} $\mathcal{A}_\ell$ rather than the structured set $\mathcal{S}_K\times\cP^N$.
For $(\hat S_{t,\ell},\hat{\bm p}_{t,\ell}),$ let $\xi_j:=\widetilde\bx_j((\hat p_{t,\ell})_j)^\top\hat{\bm\vartheta}_{t,\ell}^\lambda+w_{tj}^\ell((\hat p_{t,\ell})_j)-\widetilde\bx_j((\hat p_{t,\ell})_j)^\top\bm\vartheta^\star$. Under~\eqref{eq-Htl-bound} and~\eqref{eq-xti-bound}, we have $\xi_j  \le 2w_{tj}^\ell((\hat p_{t,\ell})_j)$. Let $v_j:=\exp(\widetilde\bx_j((\hat p_{t,\ell})_j)^\top\bm\vartheta^\star)$. By statement~(i) of Proposition~\ref{prop: assortment-perturbation}, it holds that
\begin{align*}
   & R^\text{UCB}_{t,\ell}(\hat S_{t,\ell},\hat{\bm p}_{t,\ell}) -  R(\hat S_{t,\ell},\hat{\bm p}_{t,\ell} \lvert \bm\vartheta^\star ) \\&
    \leq \sqrt{\sum_{j\in \hat S_{t,\ell}} v_j \lvert (\hat p_{t,\ell})_j - R(\hat S_{t,\ell},\hat{\bm p}_{t,\ell} \lvert \bm\vartheta^\star)\rvert^2} \cdot   \sqrt{\sum_{j\in \hat S_{t,\ell}}q_j(\hat S_{t,\ell},\hat{\bm p}_{t,\ell}\lvert \bm\vartheta^\star)q_0(\hat S_{t,\ell},\hat{\bm p}_{t,\ell}\lvert \bm\vartheta^\star) \xi_j^2} + \frac{3}{2}P\cdot \max_{j\in \hat S_{t,\ell}} \xi_j^2.
\end{align*}
Now if we denote \begin{align*}
    \hat S_{t,\ell}^+:=\{j\in \hat S_{t,\ell},  (\hat p_{t,\ell})_j \geq R(\hat S_{t,\ell},\hat{\bm p}_{t,\ell}\lvert \bm\vartheta^\star)\},\quad \hat S_{t,\ell}^-:= \hat S_{t,\ell} \setminus \hat S_{t,\ell}^+,
\end{align*}
then it holds that \begin{align*}
    &\sum_{j\in \hat S_{t,\ell}} v_j \lvert  (\hat p_{t,\ell})_j - R(\hat S_{t,\ell},\hat{\bm p}_{t,\ell} \lvert \bm\vartheta^\star)\rvert^2  \leq P\sum_{j\in \hat S_{t,\ell}} v_j \lvert  (\hat p_{t,\ell})_j -R(\hat S_{t,\ell},\hat{\bm p}_{t,\ell} \lvert \bm\vartheta^\star)\rvert\\
    &= P\left(\sum_{j\in \hat S_{t,\ell}^+} v_j ( (\hat p_{t,\ell})_j -R(\hat S_{t,\ell},\hat{\bm p}_{t,\ell} \lvert \bm\vartheta^\star))  - \sum_{j\in \hat S_{t,\ell}^-} v_j ( (\hat p_{t,\ell})_j -R(\hat S_{t,\ell} ,\hat{\bm p}_{t,\ell}\lvert \bm\vartheta^\star))\right)\\
    &= P\left(2\sum_{j\in \hat S_{t,\ell}^+} v_j ( (\hat p_{t,\ell})_j -R(\hat S_{t,\ell},\hat{\bm p}_{t,\ell} \lvert \bm\vartheta^\star))  - R(\hat S_{t,\ell} ,\hat{\bm p}_{t,\ell}\lvert \bm\vartheta^\star)\right) \\
    &\le  2P\sum_{j\in \hat S_{t,\ell}^+} v_j ( (\hat p_{t,\ell})_j -R(\hat S_{t,\ell},\hat{\bm p}_{t,\ell} \lvert \bm\vartheta^\star))\\
    &= 2P\sum_{j\in \hat S_{t,\ell}^+} v_j ( (\hat p_{t,\ell})_j -R_{t,\ell}^\text{UCB}(\hat S_{t,\ell},\hat{\bm p}_{t,\ell} )) + 2P\sum_{j\in \hat S_{t,\ell}^+} v_j( R_{t,\ell}^\text{UCB}(\hat S_{t,\ell},\hat{\bm p}_{t,\ell} ) -  R(\hat S_{t,\ell},\hat{\bm p}_{t,\ell} \lvert \bm\vartheta^\star))\\
    &\leq 2P\sum_{j\in \hat S_{t,\ell}^+} v_j ( (\hat p_{t,\ell})_j -R(S_t^\star,\bm p_t^\star\lvert \bm\vartheta^\star )) + 2Pq_0^{-1}(\hat S_{t,\ell},\hat{\bm p}_{t,\ell} \lvert \bm\vartheta^\star)( R_{t,\ell}^\text{UCB}(\hat S_{t,\ell},\hat{\bm p}_{t,\ell} ) -  R(\hat S_{t,\ell},\hat{\bm p}_{t,\ell} \lvert \bm\vartheta^\star))\\
    &\leq 2P R(S_t^\star,\bm p_t^\star \lvert \bm\vartheta^\star) + 2Pq_0^{-1}(\hat S_{t,\ell},\hat{\bm p}_{t,\ell} \lvert \bm\vartheta^\star)( R_{t,\ell}^\text{UCB}(\hat S_{t,\ell} ,\hat{\bm p}_{t,\ell}) -  R(\hat S_{t,\ell},\hat{\bm p}_{t,\ell} \lvert \bm\vartheta^\star)),
\end{align*}
 where the last second inequality follows from inequalities~\eqref{eq:rev-UCB-relation} and $$\sum_{j\in \hat S_{t,\ell}^+} v_j\le 1+\sum_{j\in \hat S_{t,\ell}} v_j=q_0^{-1}(\hat S_{t,\ell},\hat{\bm p}_{t,\ell} \lvert \bm\vartheta^\star),$$
and in the last inequality, we have
\begin{align*}
    &\sum_{j\in \hat S_{t,\ell}^+} v_j ( (\hat p_{t,\ell})_j -R(S_t^\star,\bm p_t^\star\lvert \bm\vartheta^\star )) \le \sum_{j\in \hat S_{t,\ell}^+} v_j ( (\hat p_{t,\ell})_j -R(\hat S_{t,\ell},\hat{\bm p}_{t,\ell}\lvert \bm\vartheta^\star ))\\
    & \le \sum_{j\in \hat S_{t,\ell}} v_j ( (\hat p_{t,\ell})_j -R(\hat S_{t,\ell},\hat{\bm p}_{t,\ell}\lvert \bm\vartheta^\star )) =R(\hat S_{t,\ell},\hat{\bm p}_{t,\ell} \lvert \bm\vartheta^\star ) \le R( S_t^\star,\bm p_t^\star\lvert \bm\vartheta^\star ).
\end{align*}
As a consequence, for $\Delta:= R^\text{UCB}_{t,\ell}(\hat S_{t,\ell},\hat{\bm p}_{t,\ell}) -  R(\hat S_{t,\ell},\hat{\bm p}_{t,\ell} \lvert \bm\vartheta^\star) $, we get 
\begin{align*}
    \Delta \leq \sqrt{2P R(S_t^\star,\bm p_t^\star \lvert \bm\vartheta^\star) + 2Pq_0^{-1}(\hat S_{t,\ell},\hat{\bm p}_{t,\ell} \lvert \bm\vartheta^\star) \Delta} \cdot \sqrt{\sum_{j\in \hat S_{t,\ell}}q_j(\hat S_{t,\ell},\hat{\bm p}_{t,\ell} \lvert \bm\vartheta^\star)q_0(\hat S_{t,\ell},\hat{\bm p}_{t,\ell} \lvert \bm\vartheta^\star) \xi_j^2} + \frac{3}{2}P\cdot\max_{j\in \hat S_{t,\ell}} \xi_j^2.
\end{align*}
Using the elementary inequalities \begin{align*}
    z^2 \leq Az + B \le \frac{A^2+z^2}{2}+ B \implies  z^2 \leq A^2+2B,
\end{align*}
we get then 
\begin{align*}
    \Delta &\leq 2P \sum_{j\in \hat S_{t,\ell}} q_j(\hat S_{t,\ell},\hat{\bm p}_{t,\ell} \lvert \bm\vartheta^\star) \xi_j^2 + 2\sqrt{2P^2\sum_{j\in \hat S_{t,\ell},\hat{\bm p}_{t,\ell}}q_j(\hat S_{t,\ell},\hat{\bm p}_{t,\ell}\lvert \bm\vartheta^\star)q_0(\hat S_{t,\ell},\hat{\bm p}_{t,\ell}\lvert \bm\vartheta^\star) \xi_j^2} + 3P\cdot \max_{j\in \hat S_{t,\ell}} \xi_j^2\\
    &\leq  2\sqrt{2P^2\sum_{j\in \hat S_{t,\ell}}q_j(\hat S_{t,\ell},\hat{\bm p}_{t,\ell}\lvert \bm\vartheta^\star)q_0(\hat S_{t,\ell},\hat{\bm p}_{t,\ell}\lvert \bm\vartheta^\star) \xi_j^2} + 5P\cdot \max_{j\in \hat S_{t,\ell}} \xi_j^2\\
    &\le 2\sqrt{2e^4P^2\sum_{j\in \hat S_{t,\ell}}q_j(\hat S_{t,\ell},\hat{\bm p}_{t,\ell}\lvert \hat{\bm\vartheta}_0)q_0(\hat S_{t,\ell},\hat{\bm p}_{t,\ell}\lvert \hat{\bm\vartheta}_0) (2w_{tj}^\ell((\hat p_{t,\ell})_j))^2} + 5P\cdot \max_{j\in \hat S_{t,\ell}} (2w_{tj}^\ell((\hat p_{t,\ell})_j))^2\\
    &\leq W_{t,\ell}(\hat S_{t,\ell},\hat{\bm p}_{t,\ell}) \leq 2^{-\ell },
\end{align*}
where the second inequality holds because 
$$\sum_{j\in \hat S_{t,\ell}} q_j(\hat S_{t,\ell},\hat{\bm p}_{t,\ell} \lvert \bm\vartheta^\star) \xi_j^2\le \sum_{j\in \hat S_{t,\ell}} q_j(\hat S_{t,\ell},\hat{\bm p}_{t,\ell} \lvert \bm\vartheta^\star) \cdot \max_{j\in \hat S_{t,\ell}} \xi_j^2 \le  \max_{j\in \hat S_{t,\ell}} \xi_j^2,$$
the third inequality follows from~\eqref{eq-WtS-bound} and $\xi_j\le 2w_{t,j}^\ell((\hat p_{t,\ell})_j)$, and the last inequality holds because Algorithm~\ref{alg:sup-lin-mnl-pricing} does not enter step (a) in $\ell$-th loop.
Now we get 
\begin{align*}
      R^\text{UCB}_{t,\ell}(S_t^\star,\bm p_t^\star )\geq R(\hat S_{t,\ell},\hat{\bm p}_{t,\ell} \lvert \bm\vartheta^\star) \geq R_{t,\ell}^\text{UCB}(\hat S_{t,\ell},\hat{\bm p}_{t,\ell}) - 2^{-\ell}
\end{align*}
thus $(S_t^\star,\bm p_t^\star) \in \mathcal{A}_{\ell+1}$ as desired.
\end{proof}

%----------------------analysis--------------------------%
\subsection{Proof of Theorem~\ref{thm-regret-supCB}}
 \paragraph{Step~1: Initialization Phase.}First, by Lemma~\ref{lem-initial-length-bound}, the regret incurred by the initial exploration phase is bounded by \begin{align}\label{eq-regret-initial-phase}
    O(\bar P\tau) = \widetilde{\mathcal O}\bigg( \frac{dW}{\kappa L_0} \left(d\log(NT)\vee\lambda W^2\right)\bigg).
\end{align}
And it remains to bound the regret incurred by the adaptive exploration phase. Since~\eqref{eq-explore-criteria} is ensured by the initial exploration phase, we can proceed with our analysis under the inequalities\begin{align}\label{eq-Htl-bound}
    &\frac{1}{3}\bH^\lambda_{t,\ell}(\bm\vartheta^\star) \preceq \bH^\lambda_{t,\ell}(\hat{\bm\vartheta}_0)\preceq 3\bH^\lambda_{t,\ell}(\bm\vartheta^\star),\\
\label{eq-xti-bound}    &\lvert \widetilde{\bx}_j(p)^\top (\hat{\bm\vartheta}_{t,\ell}^\lambda - \bm\vartheta^\star)\rvert \leq 16 \lVert \widetilde{\bx}_j(p) \rVert_{\bH^\lambda_{t,\ell}(\bm\vartheta^\star)^{-1}} \left(\sqrt{\log(NT)} +\sqrt{\lambda} W \right) ,\quad \forall j\in [N],  p \in [0,P],
\end{align}
which holds uniformly for all $t\ge \tau,\ell\in[J]$ with probability at least $1-1/T$ by Corollary~\ref{prop-sup-ucb-confidence-bound}.

In particular, under~\eqref{eq-Htl-bound} and~\eqref{eq-xti-bound}, together with criteria condition~\ref{eq-explore-criteria}, we have \begin{align*}
    \lvert \widetilde{\bx}_j(p_j)^\top (\hat{\bm\vartheta}_{0} - \bm\vartheta^\star)\rvert &\leq 16  \lVert \widetilde{\bx}_j(p_j) \rVert_{\bH^\lambda_{\tau,0}(\bm\vartheta^\star)^{-1}} \left(\sqrt{\log(NT)} + \sqrt{\lambda}W \right) \\
  &\leq \frac{1}{9}\left(1+\frac{\sqrt{\lambda}W}{\sqrt{2d\log(NT)}}\right) \wedge \frac{2}{3}\left(1+\frac{\sqrt{\log(NT)}}{\sqrt{\lambda}W}\right) \\
  &\leq 1,\quad \forall j\in[N].
\end{align*}
Thus, for any $(S,\bm p)\in \cS_K\times \cP^N$ and $i\in S,$ let $\hat{u}_i := \widetilde{\bx}_j(p_j)^\top \hat{\bm\vartheta}_0$,
\begin{equation}\label{eq-WtS-bound}
    \begin{aligned}
   & e^{-2} q_i(S,\bm p\lvert \hat{\bm \vartheta}_0)  \leq \frac{e^{\hat{u}_i-1}}{1+\sum_{j\in S} e^{\hat{u}_j+1}} \leq  q_i(S,\bm p\lvert \bm\vartheta^\star)\le \frac{e^{\hat{u}_i+1}}{1+ \sum_{j\in S} e^{\hat{u}_j - 1}} \leq e^2 q_i(S,\bm p\lvert \hat{\bm\vartheta}_0) .
\end{aligned}
\end{equation}
Now, based on the bounds in~\eqref{eq-Htl-bound}–\eqref{eq-WtS-bound}, we can derive a regret bound for Algorithm~\ref{alg:sup-lin-mnl-pricing} during the adaptive elimination phase. 

%---------------------------------------%
\paragraph{Step~2: Per-time Regret in Elimination Phase.}

Now we can divide the analysis of the regret incurred at any time $t \geq \tau$ into two cases, depending on which condition is triggered in the $\ell$-th iteration of the algorithm:

\paragraph{Case~1: Algorithm enter step~(a).} In this case, we have by $(S_t, \bm p_t)\in \mathcal{A}_{\ell}$, Proposition~\ref{prop-optimal-elimination} , and the construction in Step~(c), \begin{align}\label{eq-stepa-analysis}
    R^\text{UCB}_{t,\ell-1}(S_t,\bm p_t) \geq \text{argmax}_{(S,\bm p)\in \mathcal{A}_{\ell-1}}{R}_{t,\ell-1}^\text{UCB}(S, \bm p) - 2^{-\ell+1} \geq R(S_t^\star,\bm p_t^\star \lvert \bm \vartheta^\star)-2^{-\ell+1}.
\end{align}
By following the same argument for proving Proposition~\ref{prop-optimal-elimination}, since $(S_t, \bm p_t)\in \mathcal{A}_{\ell-1}$, we have
\begin{align*}
    R^\mathrm{UCB}_{t,{\ell-1}}(S_t, \bm p_t)-  R(S_t, \bm p_t \lvert \bm \vartheta^\star)\leq  W_{t,{\ell-1}}(S_t , \bm p_t ) .
\end{align*}
Combing with~\eqref{eq-stepa-analysis} and using the fact that $ W_{t,{\ell}}(S_t , \bm p_t ) >2^{-\ell}$ and $W_{t,{\ell-1}}(S_t , \bm p_t )\le 2^{-\ell+1}$, we arrive at \begin{align}\label{eq-stepa-gap}
    R(S_t^\star, \bm p_t^\star \lvert \bm \vartheta^\star) -R(S_t,\bm p_t\lvert \bm\vartheta^\star) \leq 2^{-\ell+1} +   2^{-\ell+1}  \leq 4 W_{t,{\ell}}(S_t , \bm p_t ).
\end{align}

\paragraph{Case~2: Algorithm enter step~(b).} By Proposition~\ref{prop-optimal-elimination}, the selected pair $(S_t, \bm p_t)\in \text{argmax}_{\mathcal{A}_\ell}R^\UCB_{t,\ell}(S, \bm p)$ satisfies  \begin{align}\label{eq-stepb-analysis}
    R^\mathrm{UCB}_{t,\ell}(S_t, \bm p_t) \geq R^\mathrm{UCB}_{t,\ell}(S_t^\star, \bm p_t^\star) \geq R(S_t^\star, \bm p_t^\star \lvert \bm \vartheta^\star), 
\end{align}
where the last inequality is by Lemma~\ref{prop-non-uniform-reward-optimal-assortment}. 
By following the same argument for proving Proposition~\ref{prop-optimal-elimination}, we have
\begin{align*}
    R^\mathrm{UCB}_{t,\ell}(S_t,\bm p_t)-  R(S_t, \bm p_t\lvert \bm \vartheta^\star)\leq  W_{t,\ell}(S_t,\bm p_t) .
\end{align*}
Combing with~\eqref{eq-stepb-analysis} and using the fact that $W_{t,\ell}(S_t,\bm p_t) \leq 1/\sqrt{T},$ we arrive at 
\begin{align}\label{eq-stepb-gap}
      R(S_t^\star, \bm p_t^\star \lvert \bm \vartheta^\star) -R(S_t,\bm p_t\lvert \bm\vartheta^\star) \leq W_{t,\ell}(S_t,\bm p_t)  \leq 1/\sqrt T.
\end{align}

\paragraph{Final Step: Putting All Together.} Now for each $\ell$, we have the cumulative regret incurred in its elimination phase at step~(b) is given by $O(\sqrt{T})$ by~\eqref{eq-stepb-gap}. To bound its cumulative regret incurred at step~(a) during elimination phase (denoted by $\mathcal{T}_a^\ell$), we have 
\begin{align}\label{eq-regret-in-a}
    &\sum_{t\in \mathcal{T}_a^\ell}R(S^\star, \bm p^\star \lvert \bm \vartheta^\star) - R(S_t,\bm p_t\lvert \bm \vartheta^\star ) \leq \sum_{t\in \mathcal{T}_a^\ell}4W_{t,{\ell}}(S_t , \bm p_t ) \nonumber\\
    &= 16e^2P\sum_{t\in \mathcal{T}_{a}^\ell}\bigg[\sqrt{2\sum_{i \in S_t} q_i(S_t, \bm p_t\lvert \hat{\bm \vartheta}_0)q_0(S_t, \bm p_t\lvert \hat{\bm \vartheta}_0) \left(w^\ell_{ti}(p_{ti})\right)^2} + 80P\max_{i \in S_t} \left(w^\ell_{ti}(p_{ti})\right)^2\bigg] \nonumber\\
    &\lesssim   \sum_{t\in \mathcal{T}_a^\ell} P \sqrt{\sum_{j\in S_t}q_j(S_t,\bm p_t|\hat{\bm\vartheta}_0)q_0(S_t,\bm p_t|\hat{\bm\vartheta}_0) \lVert \widetilde\bx_j(p_{tj}) \rVert_{\bH^\lambda_{t,\ell}(\hat{\bm\vartheta}_0)^{-1}}^2 } \left(\sqrt{\log (NT)}+\sqrt\lambda W\right)\nonumber\\
    &\quad +P \max_{j\in S_t}  \lVert \widetilde\bx_j(p_{tj}) \rVert_{\bH^\lambda_{t,\ell}(\hat{\bm\vartheta}_0)^{-1}}^2 \left(\sqrt{\log (NT)}+\sqrt\lambda W\right)^2 .
\end{align} 

To bound the above summation over $\mathcal{T}_a^\ell$, we apply the following linear MNL version elliptical potential lemma:
\begin{lemma}[Lemma~E.2 in \citet{lee2024nearly}]\label{lemma-elliptic}
For $\lambda \geq 1,$ it holds that \begin{enumerate}
    \item $\sum_{t\in \mathcal{T}_a^\ell} \sum_{j\in S_t}q_j(S_t,\bm p_t|\hat{\bm\vartheta}_0)q_0(S_t,\bm p_t|\hat{\bm\vartheta}_0) \lVert \widetilde\bx_j(p_{tj}) \rVert_{\bH^\lambda_{t,\ell}(\hat{\bm\vartheta}_0)^{-1}}^2\leq 2d\log(1+\frac{\lvert \mathcal{T}_a^\ell\rvert}{2d\lambda}).$
    \item $\sum_{t\in \mathcal{T}_a^\ell}\max_{i\in S_t} \lVert \widetilde\bx_j(p_{tj}) \rVert_{\bH^\lambda_{t,\ell}(\hat{\bm\vartheta}_0)^{-1}}^2 \leq 2d\kappa^{-1}\log(1+\frac{\lvert \mathcal{T}_a^\ell\rvert}{2d\lambda}).$
\end{enumerate}
\end{lemma}
Applying this Lemma and Cauchy-Schwartz inequality in \eqref{eq-regret-in-a} then leads to 
\begin{align*}
  &\sum_{\ell\in[J]}\sum_{t\in \mathcal{T}_a^\ell} R(S^\star, \bm p^\star \lvert \bm \vartheta^\star) - R(S_t,\bm p_t\lvert \bm \vartheta^\star ) \\
  &\lesssim   P\sqrt{dT  \log(1+T/2d\lambda)}\left(\sqrt{\log (NT)}+\sqrt\lambda W\right) + \frac{2dP}{\kappa}\log(1+T/d\lambda)\left(\sqrt{\log (NT)}+\sqrt\lambda W\right)^2\\
  &= \widetilde{\mathcal O}\left(\frac{W}{L_0}\left(\sqrt{dT\log (NT)} + \kappa^{-1} d \log (NT) \right)\right)
\end{align*}
Hence, combining all these regrets together, we derive that
\begin{align*}
    \mathrm{Reg}(T) &= \widetilde{\mathcal O}\bigg( \frac{dW}{\kappa L_0} \left(d\log(NT)\vee\lambda W^2\right)\bigg) + O(\sqrt{T})+  \widetilde{\mathcal O}\left(\frac{W}{L_0}\left(\sqrt{dT\log (NT)} + \kappa^{-1} d \log (NT) \right)\right)\\
    &=  \widetilde{\mathcal O}\left(\frac{W}{L_0}\left(\sqrt{dT\log (NT)} + \kappa^{-1} d^2 \log (NT) \right)\right)
\end{align*}

%-------------------------------------------------%
\subsection{Extension to Time-Varying Contexts for Fixed-Assortment Pricing}
\label{sec-appendix-time-varying-pricing}

In this subsection, we consider the fixed-assortment pricing setting of Corollary~\ref{cor-online-fixed-assortment} with time-varying contexts. The algorithm and analysis are identical to those in Section~\ref{sec-online}, except that the augmented features $\widetilde{\bx}_{ti}(p_i)$ now depend on the round $t$ through the context $\bx_{ti}$. To ensure the burn-in condition, we impose the following eigenvalue assumption on the exploration data, analogous to those in \citet{oh2021multinomial,chen2020dynamic,li2017provably}:
\begin{assumption}
\label{assump-contextual-fixed-assortment-pricing}
For some given $\tau,$ the contexts $\{\bx_{ti}\}_{i \in S,\, t\in [\tau]}$ are generated i.i.d.\ from some unknown distribution $Q$ supported on the $d$-dimensional unit ball, and the exploration prices $\{p_{ti}\}_{i\in S,\,t\in[\tau]}$ are generated i.i.d.\ from some distribution $\Pi$ on $\cP$, independently of the contexts. Moreover,
$\lambda_{\min}\!\Big(\E\big[\widetilde{\bx}_{ti}(p_{ti})\,\widetilde{\bx}_{ti}(p_{ti})^\top\big]\Big) \geq \sigma_0$ for some $\sigma_0 >0.$
\end{assumption}

We present the modified algorithm for time-varying contexts in Algorithm~\ref{alg:sup-mnl-pricing-appendix}, with modifications from Algorithm~\ref{alg:sup-lin-mnl-pricing} highlighted in blue. The algorithm is identical to Algorithm~\ref{alg:sup-lin-mnl-pricing}, with two modifications: (i) the assortment $S$ is fixed throughout, so the action space reduces from $\cS_K \times \cP^N$ to $\cP^{|S|}$; and (ii) the item-wise uncertainty levels are computed over $\widetilde{\bx}_{ti}(p_i)$ instead of $\widetilde{\bx}_i(p_i)$ for each $i\in S$, which then affects $W_{t,\ell}(S,\bm p)$ and the UCB revenues.
Thus the same analysis as in the proof of Theorem~\ref{thm-regret-supCB} can be conducted to derive the same regret bound once the burn-in condition can be verified.

\begin{algorithm}[h]
\caption{SupCB-MNL-Pricing with Time-Varying Contexts and Fixed Assortment}
\label{alg:sup-mnl-pricing-appendix}
\begin{algorithmic}[1]
\STATE \textbf{Input:} Time horizon $T$, regularized parameter $\lambda\ge \bar P$, fixed assortment $S$, \textcolor{blue}{exploration parameter $\tau$}.
\STATE \textbf{Initialize} $J =\lceil \frac{1}{2}\log_2 T \rceil, \Psi_0= \dots = \Psi_{J+1} = \emptyset.$
\FOR{\textcolor{blue}{$t = 1,\dots,(J+1)\tau$}}
    \STATE \textcolor{blue}{Sample exploration prices $p_{ti}\overset{\mathrm{i.i.d.}}{\sim}\Pi$ for each $i\in S$, add $t$ into $\Psi_{\lceil t/\tau \rceil}$.}
\ENDFOR
\STATE Compute $\widehat{\bm \vartheta}_0$ based on samples in $\widetilde{\cD}_{\tau,J+1}$ as in \eqref{eq-vartheta0}.
\FOR{$t = (J+1)\tau + 1,\dots, T$}
    \STATE Set $\bm p_t = \bm 0, \ell = 1$, and $\cA_1 = \cP^{|S|}.$
    \WHILE{$\bm p_t = \bm 0$}
    \STATE Compute $W_{t,\ell}(S,\bm p),R^\mathrm{UCB}_{t,\ell}(S,\bm p),\forall \bm p\in \cA_\ell$ as in \eqref{eq-assortment-uncertain},~\eqref{eq-ucb-in-supCB}, \textcolor{blue}{with}
    \[
    \textcolor{blue}{w^\ell_{ti}(p_i):= 16\sqrt{3}\, \lVert \widetilde{\bx}_{ti}(p_i) \rVert_{\bH^\lambda_{t,\ell}(\widehat{\bm\vartheta}_{0})^{-1}} \bigl(\sqrt{\log(KT)} +\sqrt{\lambda}\, W\bigr).}
    \]

    \IF{$W_{t,\ell}(S,\bm p) > 2^{-\ell}$ for some $\bm p \in \cA_\ell$}
    \STATE Select such $\bm p_t\in \cA_\ell$.
    \STATE $\Psi_{\ell} \leftarrow \Psi_\ell \cup \{t\}$.

    \ELSIF{$W_{t,\ell}(S,\bm p)\leq 1/\sqrt{T}$ for all $\bm p\in \cA_\ell$}
    \STATE Take the action $\bm p_t = \operatorname{argmax}_{\bm p\in \cA_{\ell}}R^\mathrm{UCB}_{t,\ell}(S,\bm p)$.
    \STATE $\Psi_0 \leftarrow \Psi_0 \cup \{t\}$.

    \ELSE
    \STATE $\widehat{R} \leftarrow \max_{\bm p\in \cA_\ell}  R^\mathrm{UCB}_{t,\ell}(S,\bm p)$.
    \STATE $\cA_{\ell+1} \leftarrow \left\{ \bm p\in \cA_\ell : R^\mathrm{UCB}_{t,\ell}(S,\bm p)\geq \widehat{R} - 2^{-\ell}\right\}$.
    \STATE $\ell \leftarrow \ell + 1$.
    \ENDIF
  \ENDWHILE
  \ENDFOR
\end{algorithmic}
\end{algorithm}

The main difference lies in the initial exploration phase, where the augmented features are collected i.i.d.\ according to Assumption~\ref{assump-contextual-fixed-assortment-pricing}. Now it suffices to prove the following burn-in condition guarantee:

\begin{lemma}
\label{lem-contextual-fixed-assortment-pricing-burnin}
With the selection $\tau = \Omega\!\Big(\sigma_0^{-1}\big[\bar P^4 d\log T/\sigma_0 + \bar P^2(d\log(KT) \vee \lambda W^2)\big]\Big),$ it holds that with probability at least $1-2/T$, the event $\widetilde{\cE}_1 \cap \widetilde{\cE}_2$, with \begin{align*}
    \widetilde{\cE}_1:&= \Big\{ \tfrac{1}{3}\bH^\lambda_{t,\ell}(\bm\vartheta^\star) \preceq \bH^\lambda_{t,\ell}(\widehat{\bm\vartheta}_0) \preceq 3\bH^\lambda_{t,\ell}(\bm\vartheta^\star),\quad \forall t>\tau,\, \ell \in [J] \Big\},\\
    \widetilde{\cE}_2:&= \Big\{ \lvert \widetilde{\bx}_{tj}(p_j)^\top (\widehat{\bm\vartheta}^\lambda_{t,\ell} - \bm\vartheta^\star)\rvert \leq 16\sqrt{3}\,\lVert \widetilde{\bx}_{tj}(p_j) \rVert_{\bH^\lambda_{t,\ell}(\widehat{\bm\vartheta}_0)^{-1}} \bigl(\sqrt{\log(KT)} + \sqrt{\lambda}W\bigr),\quad \forall t>\tau,\, j\in S,\, p_j\in\cP,\, \ell \in [J] \Big\},
    \end{align*}
holds.
\end{lemma}

\begin{proof}
We need only show that with probability at least $1-1/T,$ after the exploration phase, it holds for every $\ell$ that
\begin{align}\label{eq-appendix-pricing-eigen-lb}
    \lambda_{\min}(\bH^\lambda_{t,\ell}(\bm\vartheta^\star)) \geq \bar P^2\bigl(144^2 \cdot 2d\log(KT) \vee 24^2\lambda W^2\bigr).
\end{align}
From this, we obtain
\[
\lVert \widetilde{\bx}(p) \rVert_{\bH^\lambda_{t,\ell}(\bm\vartheta^\star)^{-1}} \leq \frac{1}{144\sqrt{2d\log(KT)}} \wedge \frac{1}{24\sqrt{\lambda}W}, \quad \forall \lVert \widetilde{\bx}(p) \rVert_2 \leq \bar P,
\]
which then allows the result to follow from Theorem~\ref{prop-sup-ucb-confidence-bound} and Proposition~\ref{prop-H-dominance}.

To prove~\eqref{eq-appendix-pricing-eigen-lb}, we adapt Proposition~1 in \citet{oh2021multinomial} and \citet{li2017provably} to the pricing setting where $\|\widetilde{\bx}(p)\|_2\le \bar P$. Since the sub-Gaussian parameter of $\bigl(\E[\widetilde{\bx}(p)\widetilde{\bx}(p)^\top]\bigr)^{-1/2}\widetilde{\bx}(p)$ scales as $\bar P/\sqrt{\sigma_0}$ instead of $1/\sqrt{\sigma_0}$, the first term in the burn-in condition picks up a factor of $\bar P^4$:
\begin{proposition}\label{prop-min-eigen-pricing}
For any constant $B>0,$ there exist absolute constants $c_1,c_2>0$ so that with the selection $$\kappa \tau \geq \frac{\bar P^4}{K}\left(\frac{c_1\sqrt{2d}+c_2\sqrt{2\log T}}{\sigma_0}\right)^2+ \frac{2B}{K\sigma_0},$$
it holds with probability at least $1-1/T^2$ that $\bH^\lambda_{t,\ell}(\bm\vartheta^\star) \succeq B \bI.$
\end{proposition}
Now selecting $\tau$ as in Proposition~\ref{prop-min-eigen-pricing} with $B = \bar P^2\bigl(144^2 \cdot 2d\log(KT) \vee 24^2\lambda W^2\bigr)$ then finishes our proof.
\end{proof}

With Lemma~\ref{lem-contextual-fixed-assortment-pricing-burnin} in place, the regret analysis follows the same argument as the proof of Theorem~\ref{thm-regret-supCB}, with the fixed features $\widetilde{\bx}_i(p_i)$ replaced by the
time-varying features $\widetilde{\bx}_{ti}(p_i)$. The only additional care needed is that the elliptic potential lemma (Lemma~\ref{lemma-elliptic}) must be applied with the time-varying Hessians $\bH^\lambda_{t,\ell}(\widehat{\bm\vartheta}_0)$, but since these are constructed from the same sample-splitting bins, the argument carries over without change. We thus obtain the following guarantee.

\begin{theorem}
\label{thm-contextual-fixed-assortment-pricing}
Under Assumptions~\ref{assump-bounded-pricing},~\ref{assump-mim-price}, and
\ref{assump-contextual-fixed-assortment-pricing}, with probability at least $1-1/T$, Algorithm~\ref{alg:sup-mnl-pricing-appendix}
satisfies
\[
\mathrm{Reg}(T)
=
\widetilde{\mathcal O}\!\left(
\frac{W}{L_0}\sqrt{dT\log(KT)}
+
\frac{W}{L_0}\kappa^{-1}d^2\log(KT)
\right).
\]
\end{theorem}

%% ---- END appendix_section5.tex ----

%% ---- BEGIN appendix_section6.tex ----

\section{Improved Regret Bounds under Adversarial Context}

The main goal of this section is to establish the Bayesian regret bound in Theorem~\ref{thm-TS-regret} under adversarially chosen contexts. The proof relies on an auxiliary online regret guarantee for a UCB policy built from sequential MLE confidence sets. We therefore proceed in two steps. First, we establish an online regret bound for a sequential-MLE UCB algorithm in the adversarial-context setting. This extends the stochastic-context result of \cite{erginbas2025online} and improves its leading regret term from $\widetilde{\mathcal O}(e^W d\sqrt{KT}/L_0)$ to $\widetilde{\mathcal O}(Wd\sqrt{T}/L_0)$. The improvement follows from recent sequential MLE results in \cite{lee2024unified,lee2025improved}, which remain valid under adversarial contexts and enable a refined argument that removes the extra dependence on \(K\). We then use the same optimistic revenue functional as an intermediate quantity in the Bayesian analysis.

For notational convenience, define
\begin{align*}
    q_{ti}(S_t,\bm p_t|\bm\vartheta):=q_i(S_t\lvert \exp(\widetilde{\bm{X}}_t(\bm p_t)^\top \bm\vartheta))
\end{align*}
%--------------------------------------------------------------%
\subsection{Online regret for a sequential-MLE UCB algorithm}\label{sec-online-adversial}

Given the collected historical data up to time $t-1$, the corresponding log-likelihood loss is defined as:
\begin{align}\label{loss-function-MLE}
    \mathcal L_t(\bm\vartheta):=-\sum_{\tau=1}^{t-1} \sum_{j\in S_\tau}y_{\tau j}\log q_{\tau j}(S_\tau, \bm p_\tau\lvert \bm\vartheta)
\end{align}
with $y_{\tau j}=\mathbf{1}\{i_\tau=j\}$. 
For the norm-constrained MLE estimator $\hat{\bm \vartheta}_t:= \text{argmin}_{\lVert \bm \vartheta \rVert \leq W}\mathcal L_t(\bm \vartheta)$, we can show the following result based on recent work \cite{lee2024unified,lee2025improved}:  

\begin{lemma}\label{lem-confidence-sequential-MLE} With probability at least $1-O(1/T),$ there exists some absolute constant $C>0$ so that $\bm \vartheta^\star \in \mathcal{C}_t$ with \begin{align*}
    \mathcal{C}_t:= \bigg\{\lVert \bm \vartheta \rVert_2 \leq W: \mathcal L_t(\bm \vartheta) - \mathcal L_t(\hat{\bm \vartheta}_t) \leq \log T+ d\log( C(1+W\bar{P}t/d ) ) \bigg\}
\end{align*}
holds for all $t\in [T].$
\end{lemma}

With Lemma~\ref{lem-confidence-sequential-MLE}, we can set the optimistic revenue
\begin{align}\label{eq-RtUCB}
    R^\text{UCB}_{t}(S, \bm p):= \frac{\sum_{j\in S}p_j\exp(u_{tj}^\text{UCB}(p_j))}{1+\sum_{j\in S}\exp(u_{tj}^\text{UCB}(p_j))} \text{ with }u_{tj}^\text{UCB}(p_j):= \max_{\bm \vartheta \in \mathcal{C}_{t-1}} \widetilde{\bm{x}}_{tj}( p_j)^\top \bm\vartheta,
\end{align}
and design the sequential MLE based UCB algorithm as in Algorithm~\ref{alg:ucb-joint-pricing}. 

\begin{algorithm}[h]
\caption{UCB-Based Joint Assortment and Pricing}
\label{alg:ucb-joint-pricing}
\begin{algorithmic}[1]
\STATE \textbf{Input:} Price upper bound $P$, parameter radius $W$ for computing $\mathcal{C}_t.$ 
\STATE \textbf{Initialize:} $\mathcal{P}:= [0,{P}]^N$
\FOR{$t = 1, \dots, T$}
    \STATE Compute the norm-constrained MLE $\hat{\bm{\vartheta}}_t$ based on the loss $\mathcal L_{t}(\bm{\vartheta})$ defined in~\eqref{loss-function-MLE}.
    \STATE Select and offer $(S_t, \bm{p}_t) := \operatorname{argmax
    }_{S\in\mathcal S_K, \bm p\in\mathcal P^N} R^\mathrm{UCB}_t(S, \bm{p})$ with $R^\mathrm{UCB}_t(S, \bm{p})$ defined in~\eqref{eq-RtUCB}.
\ENDFOR
\end{algorithmic}
\end{algorithm}

The next theorem gives the online regret guarantee for Algorithm~\ref{alg:ucb-joint-pricing}.
\begin{theorem}\label{thm-MLE-regret}
Under Assumption~\ref{assump-bounded-pricing} and \ref{assump-mim-price}, Algorithm~\ref{alg:ucb-joint-pricing} achieves 
\begin{align*}
    \mathrm{Reg}(T) 
    % &\lesssim  \frac{W+\log K}{L_0}d\sqrt{T}\log\left(\frac{W+\log K}{L_0}WT\right)+\frac{1}{\kappa}\frac{W+\log K}{L_0}d^2\log^2\left(\frac{W+\log K}{L_0}WT\right) \\
    &=\widetilde{\mathcal{O}}\left(\frac{W}{L_0}\left(d\sqrt{T}+\kappa^{-1}d^2\right)\right)
\end{align*}
\end{theorem}

Theorem~\ref{thm-MLE-regret} matches the $\Omega(d\sqrt{T}/L_0)$ lower bound in \cite{erginbas2025online}, thus is optimal up to logarithmic factors and reduces the exponential dependency on $W$ in previous best-known upper bounds.
% We note that Line~5 of Algorithm~\ref{alg:ucb-joint-pricing} is generally computationally intractable and developing computationally efficient algorithms that retain optimal regret guarantees remains an important direction for future research.

%-------------proof of lemma 3.2---------------%
\subsection{Proof of Lemma~\ref{lem-confidence-sequential-MLE}}

 For brevity, we use $q_{ti}(\bm\vartheta):=q_{t i}(S_t, \bm p_t\lvert \bm\vartheta)$ in the following proof.
The proof of Lemma \ref{lem-confidence-sequential-MLE} follows directly from Theorem 3.1 in \cite{lee2024unified}, which incorporated the Lipschitz constant for the MNL loss, i.e. 
\begin{align*}
    L_t=\max_{\|\bm\vartheta\|_2\le W}\|\nabla \mathcal{L}_t(\bm\vartheta)\|_2\le (t-1)\|\sum_{j\in S_t}(q_{ti}(\bm\vartheta)-y_{ti})\widetilde{\bm x}_{ti}(p_{ti})\|_2\le 2(t-1)\bar P
\end{align*}

\begin{lemma}[{\cite[Theorem 3.1]{lee2024unified}}]
Let $L_t:=$ $\max _{\mathrm{w} \in \mathcal{W}}\left\|\nabla \mathcal{L}_t(\mathrm{w})\right\|_2$ be the Lipschitz constant of $\mathcal{L}_t(\cdot)$, which may depend on $\left\{\left(x_\tau, p_\tau\right)\right\}_{\tau=1}^{t-1}$. Then, we have $\operatorname{Pr}[\forall t \geq$ $\left.1, \bm\vartheta^\star \in \mathcal{C}_t (\delta)\right] \geq 1-\delta$, where
\begin{align*}
\mathcal{C}_t (\delta):=\left\{\bm\vartheta \in \mathcal{W}: \mathcal{L}_t(\bm\vartheta)-\mathcal{L}_t (\hat{\bm\vartheta}_t ) \leq \alpha_t (\delta)^2=\log \frac{1}{\delta}+d \log \left(\max \left\{e, \frac{2 e B L_t}{d}\right\}\right)\right\}
\end{align*}
\end{lemma}

By setting $\delta=\frac{1}{T}$ and $L_t=2(t-1)\bar P$ in the lemma above, we obtain Lemma \ref{lem-confidence-sequential-MLE} directly.

%-------------proof of theorem 3.3---------------%
\subsection{Proof of Theorem~\ref{thm-MLE-regret}}
Based on Lemma \ref{lem-confidence-sequential-MLE}, we define the confidence set as
\begin{align*}
    \mathcal{C}_t:= \bigg\{\lVert \bm \vartheta \rVert_2 \leq W: \mathcal{L}_t(\bm \vartheta) - \mathcal{L}_t(\hat{\bm \vartheta}_t) \leq \alpha^2_t \bigg\}
\end{align*}
where 
\begin{align*}
    \alpha_t=\sqrt{\log T+ d\log( C(1+W\bar{P}t/d  )}
\end{align*}
for some constant $C>0$.
Additionally, we define the Hessian of the regularized loss at $\bm\vartheta$ as:
\begin{align}\label{eq:F.1}
H_t(\bm\vartheta):=\lambda \mathbf{I}_{2d}+\sum_{\tau=1}^{t-1} \nabla^2 \ell_\tau(\bm\vartheta), \quad \text { where } \lambda=\frac{1}{8 W^2}
\end{align}

Now, we present useful lemmas that will be used to prove Theorem \ref{thm-MLE-regret}.
\begin{lemma}\label{lemF.2}(Improved MLE confidence bound)
     For any $t \in[T]$, we define $\nu_t^\star$ such that $\frac{1}{2}\left\|\bm\vartheta^\star-\hat{\bm\vartheta}_t\right\|_{\nabla^2 \mathcal{L}_t\left(\nu_t^\star\right)}^2 = \left\|\bm\vartheta^\star-\hat{\bm\vartheta}_t\right\|^2_{\int_0^1(1-v) \nabla^2 \mathcal{L}_t\left(\hat{\bm\vartheta}_t+v\left(\bm\vartheta^\star-\hat{\bm\vartheta}_t\right)\right) \mathrm{d} v}$ and $H_t \left(\nu_t^\star\right):=\lambda  \mathbf{I}_{2d}+\nabla^2 \mathcal{L}_t\left(\nu_t^\star\right)=\lambda  \mathbf{I}_{2d}+\sum_{\tau=1}^{t-1} \nabla^2 \ell_\tau\left(\nu_t^\star\right)$. Let $\lambda =\frac{1}{8 W^2}$. Then, for any $t \geq 1$, if $\bm\vartheta^\star \in \mathcal{C}_t$ and Assumption \ref{assump-bounded} holds, then we have
\begin{align*}
\left\|\bm\vartheta^\star-\hat{\bm\vartheta}_t\right\|_{H_t \left(\nu_t^\star\right)}^2 \leq \underbrace{2 \alpha_t^2+1}_{=: \gamma_t ^2}=\mathcal{O}(d \log (W\bar Pt)).
\end{align*}
\end{lemma}
The proof follows directly from Lemma F.2 in \cite{lee2025improved}, with the only modification being that $L_t=2(t-1)\bar P$ in our setting.

\begin{lemma}[{\cite[Lemma F.3]{lee2025improved}}]\label{lemF.3}
    For any $t\in[T]$, $\bm\vartheta_1,\bm\vartheta_2\in \mathcal C_t$ and $w_{ti}\ge 0$, we have 
    \begin{align*}
        \sum_{i\in S_t}\left|q_{ti}(S_t,\bm p_t|\bm\vartheta_1)-q_{ti}(S_t,\bm p_t|\bm\vartheta_2) \right|w_{ti}\le 4\gamma_t\max_{i\in S_t}w_{ti}\max_{i\in S_t}\|\widetilde{\bm x}_{ti}(p_{ti})\|_{H_t(\nu_t^\star)^{-1}}
    \end{align*}
\end{lemma}

\begin{lemma}[{\cite[Lemma 4]{kim2023improved}}]\label{lemF.4}
    For $X,L>0$, let $x_1,\dots,x_T\in \mathbb R^{2d}$ be a sequence of vectors with $\|x_t\|_2\le X$ for all $t\in[T]$. Let $H_t:=\lambda\mathbf{I}_{2d}+\sum_{\tau=1}^{t-1}x_\tau x_\tau^\top$ for some $\lambda>0$. Let $\mathcal T\subseteq[T]$ be the set of indices where $\|x_t\|_{H_t^{-1}}^2\ge L$. Then,
    \begin{align*}
        |\mathcal T|\le\frac{2}{\log(1+L)}d\log\left(1+\frac{X^2}{\log(1+L)\lambda}\right).
    \end{align*}
\end{lemma}

\begin{lemma}[{\cite[Lemma D.7]{lee2025improved}}]\label{lemD.7}
    Define $H_t(\bm\vartheta):=\lambda\mathbf I_{2d}+\sum_{\tau\notin\mathcal T_0}\nabla^2\ell_\tau(\bm\vartheta)$. If $\left\|\widetilde{\bm x}_{\tau i}(p_{\tau i})\right\|_{H_\tau \left(\bm\vartheta\right)^{-1}}^2 \le \frac{1}{2} $ for all $i\in S_\tau$ and $\tau\notin\mathcal T_0$, then we have
    \begin{align*}
        \sum_{\tau\notin\mathcal{T}_0}\sum_{i\in S_\tau\cup\{0\}}q_{\tau i}(S_\tau,\bm p_\tau|\bm\vartheta) \left\|\widetilde{\bm x}_{\tau i}(p_{\tau i})- \E_\tau^{\bm\vartheta}\left[\widetilde{\bm x}_{\tau i}(p_{\tau i}) \right]\right\|_{H_\tau(\bm\vartheta)^{-1}}^2\le 2d\log\left(1+\frac{t}{d\lambda}\right).
    \end{align*}
\end{lemma}

\begin{lemma}[{\cite[Lemma E.2 and Lemma H.3]{lee2024nearly}}]\label{lemD.11}
   Let $H_t:=\lambda\mathbf I_{2d}+\sum_{\tau=1}^{t-1}\nabla^2\ell_\tau(\bm\vartheta_{\tau+1})$. If $\left\|\widetilde{\bm x}_{\tau i}(p_{\tau i})\right\|_{H_\tau \left(\bm\vartheta\right)^{-1}}^2 \le \frac{1}{2} $ for all $i\in S_\tau$ and $\tau\in[t]$, then we have
   \begin{align*}
       &(1)\quad \sum_{\tau=1}^t\sum_{i\in S_\tau} q_{\tau i}(S_\tau,\bm p_\tau|\bm\vartheta_{\tau+1})q_{\tau 0}(S_\tau,\bm p_\tau|\bm\vartheta_{\tau+1}) \left\|\widetilde{\bm x}_{\tau i}(p_{\tau i})\right\|_{H_\tau^{-1}}^2 \le 2d\log\left(1+\frac{t}{d\lambda}\right)\\
       &(2)\quad \sum_{\tau=1}^t\sum_{i\in S_\tau} q_{\tau i}(S_\tau,\bm p_\tau|\bm\vartheta_{\tau+1}) \left\|\widetilde{\bm x}_{\tau i}(p_{\tau i})-\E_\tau^{\bm\vartheta}\left[\widetilde{\bm x}_{\tau i}(p_{\tau i}) \right]\right\|_{H_\tau^{-1}}^2 \le 2d\log\left(1+\frac{t}{d\lambda}\right)\\
       &(3)\quad \sum_{\tau=1}^t\max_{i\in S_\tau} \left\|\widetilde{\bm x}_{\tau i}(p_{\tau i})\right\|_{H_\tau^{-1}}^2 \le \frac{2}{\kappa} d\log\left(1+\frac{t}{d\lambda}\right)\\
       &(4)\quad \sum_{\tau=1}^t\max_{i\in S_\tau} \left\|\widetilde{\bm x}_{\tau i}(p_{\tau i})-\E_\tau^{\bm\vartheta}\left[\widetilde{\bm x}_{\tau i} (p_{\tau i})\right]\right\|_{H_\tau^{-1}}^2 \le \frac{2}{\kappa} d\log\left(1+\frac{t}{d\lambda}\right)
   \end{align*}
\end{lemma}

% \begin{lemma}[{\cite[Lemma H.2]{lee2024nearly}}]\label{lemH.2}
% Let ${R}_t(S)=\frac{\sum_{i \in S} \exp \left(\alpha_{t i}\right) r_{t i}}{v_0+\sum_{j \in S} \exp \left(\alpha_{t j}\right)}$ and $S_t=\operatorname{argmax}_{S \in S} {R}_t(S)$. Assume $\alpha_{t i}^{\prime} \geqslant \alpha_{t i} \geqslant 0$ for all $i \in[N]$. Then, we have
% \begin{align*}
% {R}_t\left(S_t\right) \leqslant \frac{\sum_{i \in S_t} \exp \left(\alpha_{t i}^{\prime}\right) r_{t i}}{v_0+\sum_{j \in S_t} \exp \left(\alpha_{t j}^{\prime}\right)}
% \end{align*}
% \end{lemma}

\begin{lemma}[{\cite[Lemma E.3]{lee2024nearly}}]\label{lemE.3}
Define $Q: \mathbb{R}^K \rightarrow \mathbb{R}$, such that for any $\mathbf{u}=\left(u_1, \ldots, u_K\right) \in \mathbb{R}^K, Q(\mathbf{u})=$ $\sum_{i=1}^K \frac{\exp \left(u_i\right)}{v_0+\sum_{k=1}^K \exp \left(u_k\right)}$. Let $p_i(\mathbf{u})=\frac{\exp \left(u_i\right)}{v_0+\sum_{k=1}^K \exp \left(u_k\right)}$. Then, for all $i \in[K]$, we have
\begin{align*}
    \left|\frac{\partial^2 Q}{\partial i \partial j}\right| \leqslant \begin{cases}3 p_i(\mathbf{u}) & \text { if } i=j \\ 2 p_i(\mathbf{u}) p_j(\mathbf{u}) & \text { if } i \neq j\end{cases}
\end{align*}
\end{lemma}
%---------------Main proof--------------%
\subsubsection{Main proof of Theorem~\ref{thm-MLE-regret}}
\begin{proof}[\textbf{Proof of Theorem~\ref{thm-MLE-regret}}]
We first define the set of large elliptical potential rounds as follows:
\begin{align*}
\mathcal{T}_0:=\left\{t \in[T]:\left\|\widetilde{\bm x}_{t i}(p_{ti})\right\|_{H_t \left(\nu_t^\star\right)^{-1}}^2 \ge \frac{1}{2}, \quad \forall i \in S_t\right\} .
\end{align*}

Let $\mathrm{UCB}_{t i}(p_i)=\widetilde{\bm x}_{t i}(p_i)^{\top} \hat{\bm\vartheta}_t+\gamma_t  \left\|\widetilde{\bm x}_{t i}\right\|_{H_t \left(\nu_t^{\star}\right)^{-1}}$ and $\overline{\mathrm{UCB}}_{t i}(p_i):=\widetilde{\bm x}_{t i}^{\top} \bm\vartheta^{\star}+2 \gamma_t  \left\|\widetilde{\bm x}_{t i}(p_i)\right\|_{H_t \left(\nu_t^{\star}\right)^{-1}}$. Then, by Lemma \ref{lemF.2}, for all $i \in[N]$ and $t \ge 1$, we have
\begin{align}\label{eq:27}
u_{ti}^\text{UCB}(p_i)-\widetilde{\bm x}_{t i}( p_i)^{\top} \bm\vartheta^{\star} &= \max_{\bm \vartheta \in \mathcal{C}_{t-1}} \tilde{\bm{x}}_{tj}( p_i)^\top \bm\vartheta-\widetilde{\bm x}_{t i}( p_i)^{\top} \bm\vartheta^{\star} \nonumber\\&\le \mathrm{UCB}_{t i}( p_i)-\widetilde{\bm x}_{t i}( p_i)^{\top} \bm\vartheta^{\star} \le 2 \gamma_t  \left\|\widetilde{\bm x}_{t i}( p_i)\right\|_{H_t \left(\nu_t^{\star}\right)^{-1}}
\end{align}
which implies $u_{ti}^\text{UCB}(p_i) \le \overline{\mathrm{UCB}}_{t i}(p_i)$. Thus, by Proposition~\ref{prop-non-uniform-reward-optimal-assortment}, we get
\begin{align}\label{eq:F.2}
 R_t^{\text{UCB}}(S_t,\bm p_t)\le \widetilde{R}_t \left(S_t,\bm p_t\right):=\frac{\sum_{i \in S} \exp \left(\overline{\mathrm{UCB}}_{t i}( p_{ti})\right) p_{ti}}{1+\sum_{j \in S} \exp \left(\overline{\mathrm{UCB}}_{t j}( p_{tj})\right)}.
\end{align}
Now, we decompose the regret into two phases:
\begin{align} \label{eq:F.3}
    \mathrm{Reg}(T)&=\sum_{t\in\mathcal T_0} \E \left[R_t(S_t^\star,\bm p_t^\star|\bm\vartheta^\star)-R_t(S_t,\bm p_t|\bm\vartheta^\star)\right] + \sum_{t\notin\mathcal T_0} \E \left[R_t(S_t^\star,\bm p_t^\star|\bm\vartheta^\star)-R_t(S_t,\bm p_t|\bm\vartheta^\star)\right] \nonumber\\
    &\le  P |\mathcal T_0|+ \sum_{t\notin\mathcal T_0}\E \left[R_t(S_t^\star,\bm p_t^\star|\bm\vartheta^\star)-R_t(S_t,\bm p_t|\bm\vartheta^\star)\right] \nonumber\\
    &\le   P \frac{2}{\log(3/2)}d\log\left(1+\frac{\bar P^2}{\log(3/2)\lambda}\right)+ \sum_{t\notin\mathcal T_0}\E \left[R_t(S_t^\star,\bm p_t^\star|\bm\vartheta^\star)-R_t(S_t,\bm p_t|\bm\vartheta^\star )\right],
\end{align}
where the last inequality follows from Lemma~\ref {lemF.4}.

Given an assortment and price $(S_t,\bm p_t)$, we define a function $Q: \mathbb{R}^{K_t} \rightarrow \mathbb{R}$ as 
\begin{align*}
    Q(\mathbf u)=\sum_{i\in S_t}\frac{\exp(u_i)p_{ti}}{1+\sum_{j\in S_t}\exp(u_j)}.
\end{align*}
Furthermore, we denote $\bm u_t=\left(\overline{\mathrm{UCB}}_{t i_1}, \ldots, \overline{\mathrm{UCB}}_{t i_{K_t}}\right)^{\top}$ and $\bm u_t^{\star}=\left(x_{t i_1}^{\top} \bm\vartheta^{\star}, \ldots, x_{t i_{K_t}}^{\top} \bm\vartheta^{\star}\right)^{\top}$.
Then, we have 
\begin{align}
    \sum_{t\notin\mathcal T_0} \left[R_t(S_t^\star,\bm p_t^\star|\bm\vartheta^\star)-R_t(S_t,\bm p_t|\bm\vartheta^\star)\right] &\le \sum_{t\notin\mathcal T_0} \left[R_t^{\text{UCB}}(S_t,\bm p_t)-R_t(S_t,\bm p_t|\bm\vartheta^\star)\right] \nonumber\\
    & \le \sum_{t\notin\mathcal T_0} \left[\widetilde{R}_t(S_t,\bm p_t)-R_t(S_t,\bm p_t|\bm\vartheta^\star)\right] \nonumber\\
    & =  \sum_{t\notin\mathcal T_0} \left[Q(\bm u_t)-Q(\bm u_t^\star)\right] \nonumber\\
    & = \underbrace{\sum_{t\notin\mathcal T_0}\nabla Q(\bm u_t^\star)^\top(\bm u_t-\bm u_t^\star)}_{(\mathbf I)}+ \underbrace{\frac{1}{2} \sum_{t\notin\mathcal T_0}(\bm u_t-\bm u_t^\star)^\top \nabla^2 Q(\bar{\bm u}_t)(\bm u_t-\bm u_t^\star)}_{(\mathbf{II})} \label{eq:F.4}
\end{align}
where the first inequality holds follows from Lemma 4 in \cite{oh2021multinomial}, the second inequality follows from the inequality \eqref{eq:F.2}, and in the last equality, $\bar{\bm u}_t=c\bm u_t+(1-c)\bm u_t^\star$ for some $c\in[0,1]$.

\paragraph{Bounding $(\mathbf{I})$}
First, we bound the term $(\mathbf{I})$ in Equation \eqref{eq:F.4}. For simplicity, let $q_{ti}(\bm\vartheta)=\frac{\exp(\widetilde{\bm x}_{ti}(p_{ti})^\top\bm\vartheta )}{1+\sum_{j\in S_t}\exp(\widetilde{\bm x}_{tj}(p_{tj})^\top\bm\vartheta )}$, $\E_t^{\bm\vartheta}[\widetilde{\bm x}_{ti}(p_{ti})]=\sum_{i\in S_t}q_{ti}(\bm\vartheta)\widetilde{\bm x}_{ti}(p_{ti})$ . Then, we have
\begin{align} \label{eq:31}
& \sum_{t\notin\mathcal T_0} \nabla Q\left(\bm u_t^{\star}\right)^{\top}\left(\bm u_t-\bm u_t^{\star}\right) \nonumber\\ & =\sum_{t\notin\mathcal T_0}\sum_{i \in S_t} \frac{\exp \left(u_{ti}^\star\right) p_{t i}}{1+\sum_{k \in S_t} \exp \left(u_{tk}^\star\right)}\left(u_{t i}-u_{t i}^{\star}\right)- \frac{\sum_{j \in S_t}\exp \left(u_{tj}^\star\right) p_{t j} \sum_{i \in S_t} \exp \left(u_{ti}^\star\right)}{\left(1+\sum_{k \in S_t} \exp \left(u_{tk}^\star\right)\right)^2}\left(u_{t i}-u_{t i}^{\star}\right) \nonumber \\ & 
=\sum_{t\notin\mathcal T_0}\sum_{i\in S_t}p_{ti} q_{ti}(\bm\vartheta^\star)  (u_{t i}-u_{t i}^{\star})- \left( \sum_{j\in S_t} p_{tj} q_{tj}(\bm\vartheta^\star) \right)  \left(\sum_{i\in S_t}  q_{ti}(\bm\vartheta^\star)(u_{t i}-u_{t i}^{\star})\right)\nonumber \\
&\le 2\gamma_T  P \sum_{t\notin\mathcal T_0} \sum_{i\in S_t}q_{ti}(\bm\vartheta^\star)\left(\|\widetilde{\bm x}_{ti}(p_{ti})\|_{H_t(\nu_t^\star)^{-1}}-\sum_{j\in S_t}q_{tj}(\bm\vartheta^\star)\|\widetilde{\bm x}_{tj}(p_{tj})\|_{H_t(\nu_t^\star)^{-1}} \right) \nonumber\\
&= 2\gamma_T P \sum_{t\notin\mathcal T_0} \E_t^{\bm\vartheta^\star}\left[\|\widetilde{\bm x}_{ti}(p_{ti})\|_{H_t(\nu_t^\star)^{-1}} -\E_t^{\bm\vartheta^\star}\left[\|\widetilde{\bm x}_{ti}(p_{ti})\|_{H_t(\nu_t^\star)^{-1}}  \right] \right]\nonumber\\
&\le 2\gamma_T P \sum_{t\notin\mathcal T_0} \E_t^{\bm\vartheta^\star}\left[\left\|\widetilde{\bm x}_{ti}(p_{ti})- \E_t^{\bm\vartheta^\star}\left[\widetilde{\bm x}_{ti} (p_{ti})\right]\right\|_{H_t(\nu_t^\star)^{-1}} \right]
\end{align}
where the first inequality follows from Equation \eqref{eq:27} and the bound $p_{ti}\le P$, and the last equality holds due to the fact that $\|\bm a\|-\|\bm b\|\le\|\bm a-\bm b\|$ for any vectors $\bm a,\bm b\in\mathbb R^{2d}$.

We can decompose the last term of Equation \eqref{eq:31} as follows:
\begin{align}\label{eq:F.5}
    &\E_t^{\bm\vartheta^\star}\left[\|\widetilde{\bm x}_{ti}(p_{ti})- \E_t^{\bm\vartheta^\star}\left[\widetilde{\bm x}_{ti}(p_{ti}) \right]\|_{H_t(\nu_t^\star)^{-1}} \right] = \sum_{t\notin\mathcal T_0}\sum_{i\in S_t}\sqrt{q_{ti}(\bm\vartheta^\star)q_{ti}(\nu_t^\star)} \left\|\widetilde{\bm x}_{ti}(p_{ti})- \E_t^{\nu_t^\star}\left[\widetilde{\bm x}_{ti}(p_{ti}) \right]\right\|_{H_t(\nu_t^\star)^{-1}} \nonumber \\
    &+ \sum_{t\notin\mathcal T_0}\sum_{i\in S_t}\left(\sqrt{q_{ti}(\bm\vartheta^\star)}-\sqrt{q_{ti}(\nu_t^\star)}\right) \sqrt{q_{ti}(\bm\vartheta^\star)}\left\|\widetilde{\bm x}_{ti}(p_{ti})- \E_t^{\nu_t^\star}\left[\widetilde{\bm x}_{ti}(p_{ti}) \right]\right\|_{H_t(\nu_t^\star)^{-1}} \nonumber\\
    &+ \sum_{t\notin\mathcal T_0}\sum_{i\in S_t}q_{ti}(\bm\vartheta^\star) \left(\left\|\widetilde{\bm x}_{ti}(p_{ti})- \E_t^{\bm\vartheta^\star}\left[\widetilde{\bm x}_{ti}(p_{ti}) \right]\right\|_{H_t(\nu_t^\star)^{-1}}-\left\|\widetilde{\bm x}_{ti}(p_{ti})- \E_t^{\nu_t^\star}\left[\widetilde{\bm x}_{ti}(p_{ti}) \right]\right\|_{H_t(\nu_t^\star)^{-1}} \right)
\end{align}
By applying Cauchy-Schwarz inequality and Lemma \ref{lemD.7}, the first term in Equation \eqref{eq:F.5} can be bounded by
\begin{align*}
    &\sum_{t\notin\mathcal T_0}\sum_{i\in S_t}\sqrt{q_{ti}(\bm\vartheta^\star)q_{ti}(\nu_t^\star)} \left\|\widetilde{\bm x}_{ti}(p_{ti})- \E_t^{\nu_t^\star}\left[\widetilde{\bm x}_{ti} (p_{ti})\right]\right\|_{H_t(\nu_t^\star)^{-1}} \\
    &\le \sqrt{\sum_{t\notin\mathcal T_0}\sum_{i\in S_t}q_{ti}(\bm\vartheta^\star)} \sqrt{\sum_{t\notin\mathcal T_0}\sum_{i\in S_t}q_{ti}(\nu_t^\star)\left\|\widetilde{\bm x}_{ti}(p_{ti})- \E_t^{\nu_t^\star}\left[\widetilde{\bm x}_{ti}(p_{ti}) \right]\right\|_{H_t(\nu_t^\star)^{-1}}^2}\\
    &\le \sqrt{T}\sqrt{2d\log(1+\frac{T}{d\lambda})}.
\end{align*}
Additionally, the second  term in Equation \eqref{eq:F.5} can be bounded by
\begin{align*}
    &\sum_{t\notin\mathcal T_0}\sum_{i\in S_t}\left(\sqrt{q_{ti}(\bm\vartheta^\star)}-\sqrt{q_{ti}(\nu_t^\star)}\right) \sqrt{q_{ti}(\bm\vartheta^\star)}\left\|\widetilde{\bm x}_{ti}(p_{ti})- \E_t^{\nu_t^\star}\left[\widetilde{\bm x}_{ti}(p_{ti}) \right]\right\|_{H_t(\nu_t^\star)^{-1}}\\
    &=\sum_{t\notin\mathcal T_0}\sum_{i\in S_t} \frac{|q_{ti}(\bm\vartheta^\star)-q_{ti}(\nu_t^\star)|}{\sqrt{q_{ti}(\bm\vartheta^\star)}+\sqrt{q_{ti}(\nu_t^\star)}}
\sqrt{q_{ti}(\bm\vartheta^\star)}\left\|\widetilde{\bm x}_{ti}(p_{ti})- \E_t^{\nu_t^\star}\left[\widetilde{\bm x}_{ti} (p_{ti})\right]\right\|_{H_t(\nu_t^\star)^{-1}}\\
    &\le \sum_{t\notin\mathcal T_0}\sum_{i\in S_t} |q_{ti}(\bm\vartheta^\star)-q_{ti}(\nu_t^\star)|\cdot \left\|\widetilde{\bm x}_{ti}(p_{ti})- \E_t^{\nu_t^\star}\left[\widetilde{\bm x}_{ti}(p_{ti}) \right]\right\|_{H_t(\nu_t^\star)^{-1}}\\
    &\le 4\gamma_T \sum_{t\notin\mathcal T_0}\max_{i\in S_t} \left\|\widetilde{\bm x}_{ti}(p_{ti}) - \E_t^{\nu_t^\star}\left[\widetilde{\bm x}_{ti}(p_{ti}) \right]\right\|_{H_t(\nu_t^\star)^{-1}}\max_{i\in S_t} \left\|\widetilde{\bm x}_{ti}(p_{ti})\right\|_{H_t(\nu_t^\star)^{-1}}\\
    &\le 4\gamma_T\sqrt{\sum_{t\notin\mathcal T_0}\max_{i\in S_t} \left\|\widetilde{\bm x}_{ti}(p_{ti})- \E_t^{\nu_t^\star}\left[\widetilde{\bm x}_{ti}(p_{ti}) \right]\right\|_{H_t(\nu_t^\star)^{-1}}^2 } \sqrt{\sum_{t\notin\mathcal T_0}\max_{i\in S_t} \left\|\widetilde{\bm x}_{ti}(p_{ti})\right\|_{H_t(\nu_t^\star)^{-1}}^2 }\\
    &\le \frac{4}{\sqrt\kappa}\gamma_T\sqrt{\sum_{t\notin\mathcal T_0}\sum_{i\in S_t} q_{ti}(\nu_t^\star) \left\|\widetilde{\bm x}_{ti}(p_{ti})- \E_t^{\nu_t^\star}\left[\widetilde{\bm x}_{ti}(p_{ti}) \right]\right\|_{H_t(\nu_t^\star)^{-1}}^2 } \sqrt{\sum_{t\notin\mathcal T_0}\max_{i\in S_t} \left\|\widetilde{\bm x}_{ti}(p_{ti})\right\|_{H_t(\nu_t^\star)^{-1}}^2 }\\
    &\le \frac{8}{\kappa}\gamma_T d\log(1+\frac{T}{d\lambda})
\end{align*}
where the first inequality holds because $\frac{\sqrt{q_{ti}(\bm\vartheta^\star)}}{\sqrt{q_{ti}(\bm\vartheta^\star)}+\sqrt{q_{ti}(\nu_t^\star)}}\le 1$, the second inequality follows from Lemma
\ref{lemF.3}, the third inequality follows from Cauchy-Schwarz inequality, the fourth inequality holds due to the definition of $\kappa$, and the last one follows from Lemma \ref{lemD.7} and \ref{lemD.11}.

Finally, we bound the last term in \eqref{eq:F.5} as follows:
\begin{align*}
    &\sum_{t\notin\mathcal T_0}\sum_{i\in S_t}q_{ti}(\bm\vartheta^\star) \left(\left\|\widetilde{\bm x}_{ti}(p_{ti})- \E_t^{\bm\vartheta^\star}\left[\widetilde{\bm x}_{ti}(p_{ti}) \right]\right\|_{H_t(\nu_t^\star)^{-1}}-\left\|\widetilde{\bm x}_{ti}(p_{ti})- \E_t^{\nu_t^\star}\left[\widetilde{\bm x}_{ti}(p_{ti}) \right]\right\|_{H_t(\nu_t^\star)^{-1}} \right)\\
    &\le \sum_{t\notin\mathcal T_0}\sum_{i\in S_t}q_{ti}(\bm\vartheta^\star) \left\| \E_t^{\bm\vartheta^\star}\left[\widetilde{\bm x}_{ti}(p_{ti}) \right]- \E_t^{\nu_t^\star}\left[\widetilde{\bm x}_{ti}(p_{ti}) \right]\right\|_{H_t(\nu_t^\star)^{-1}} \\
    &= \sum_{t\notin\mathcal T_0}\sum_{i\in S_t}q_{ti}(\bm\vartheta^\star) \left\| \sum_{j\in S_t}(q_{tj}(\nu_t^\star)-q_{tj}(\bm\vartheta^\star))\widetilde{\bm x}_{tj}(p_{tj}) \right\|_{H_t(\nu_t^\star)^{-1}} \\
    &\le \sum_{t\notin\mathcal T_0}\sum_{j\in S_t}|q_{tj}(\bm\vartheta^\star) -q_{tj}(\nu_t^\star)|\left\| \widetilde{\bm x}_{tj}(p_{tj}) \right\|_{H_t(\nu_t^\star)^{-1}} \\
    &\le 4\gamma_T \sum_{t\notin\mathcal T_0}\max_{j\in S_t} \left\| \widetilde{\bm x}_{tj}(p_{tj}) \right\|_{H_t(\nu_t^\star)^{-1}}^2\\
    &\le \frac{8}{\kappa}\gamma_T d\log\left(1+\frac{T}{d\lambda}\right).
\end{align*}
where the first inequality holds because $\|\bm a\|-\|\bm b\|\le\|\bm a-\bm b\|$ for any vectors $\bm a,\bm b\in\mathbb R^{2d}$, the second inequality holds due to the fact that $\sum_{i\in S_t}q_{ti}(\bm\vartheta^\star)\le 1$ and $\|\sum_ia_ix_i\|_H\le \sum_i|a_i|\|x_i\|_H$ for any scalars $a_i$ and vectors $x_i$, the third inequality follows from Lemma \ref{lemF.3}, and the last inequality follows from Lemma \ref{lemD.11}.

Thus, combining the results above, we can bound term $(\mathbf{I})$ in \eqref{eq:F.4} by
\begin{align}\label{eq:F.6}
    \sum_{t\notin\mathcal T_0} \nabla Q\left(\bm u_t^{\star}\right)^{\top}\left(\bm u_t-\bm u_t^{\star}\right)\leq 2\sqrt2\gamma_T  P\sqrt{dT\log\left(1+\frac{T}{d\lambda}\right)}+\frac{32}{\kappa}\gamma_T^2  P d\log\left(1+\frac{T}{d\lambda}\right).
\end{align}

\paragraph{Bounding $(\mathbf{II})$}
Now, we provide the upper bound for term $(\mathbf{II})$ in \eqref{eq:F.4}. Given an assortment and price $(S_t,\bm p_t)$, we define a function $\bar Q: \mathbb{R}^{K_t} \rightarrow \mathbb{R}$ as $\bar Q(\mathbf u)=\sum_{i\in S_t}\frac{\exp(u_i)}{1+\sum_{j\in S_t}\exp(u_j)}.$ Since $p_{ti}\in[0,  P]$, $\left|\frac{\partial^2Q}{\partial i\partial j}\right| \le \left|\frac{\partial^2\bar Q}{\partial i\partial j}\right|$. Hence, we have
\begin{align*}
    \frac{1}{2} \sum_{t\notin\mathcal T_0}(\bm u_t-\bm u_t^\star)^\top \nabla^2 Q(\bar{\bm u}_t)(\bm u_t-\bm u_t^\star) &\le \frac{1}{2} \sum_{t\notin\mathcal T_0} \sum_{i\in S_t}\sum_{j\in S_t}(u_{ti}-u_{ti}^\star)\frac{\partial^2Q}{\partial i\partial j}(u_{tj}-u_{tj}^\star)\\
    &\le \frac{  P}{2} \sum_{t=1}^T \sum_{i\in S_t}\sum_{j\in S_t}|u_{ti}-u_{ti}^\star|\left|\frac{\partial^2\bar Q}{\partial i\partial j}\right||u_{tj}-u_{tj}^\star|
\end{align*}
For simplicity, we denote $q_i(\bar{\bm u}_t)=\frac{\exp(\bar u_{ti})}{1+\sum_{j\in S_t}\exp(\bar u_{tj})}$. Then, we have
\begin{align*}
    &\frac{  P}{2} \sum_{t=1}^T \sum_{i\in S_t}\sum_{j\in S_t}|u_{ti}-u_{ti}^\star|\left|\frac{\partial^2\bar Q}{\partial i\partial j}\right||u_{tj}-u_{tj}^\star|\\
    &=\frac{P}{2} \sum_{t=1}^T \sum_{i\in S_t}\sum_{j\in S_t,j\neq i}|u_{ti}-u_{ti}^\star|\left|\frac{\partial^2\bar Q}{\partial i\partial j}\right||u_{tj}-u_{tj}^\star|  +  \frac{P}{2} \sum_{t=1}^T \sum_{i\in S_t}|u_{ti}-u_{ti}^\star|\left|\frac{\partial^2\bar Q}{\partial i^2}\right||u_{tj}-u_{tj}^\star|\\
    &\le  P \sum_{t=1}^T \sum_{i\in S_t}\sum_{j\in S_t,j\neq i}|u_{ti}-u_{ti}^\star|\left|q_i(\bar{\bm u}_t)q_j(\bar{\bm u}_t)\right||u_{tj}-u_{tj}^\star|  +  \frac{3P}{2} \sum_{t=1}^T \sum_{i\in S_t}(u_{ti}-u_{ti}^\star)^2q_i(\bar{\bm u}_t)\\
    &\le  P \sum_{t=1}^T \sum_{i\in S_t}\sum_{j\in S_t}|u_{ti}-u_{ti}^\star|\left|q_i(\bar{\bm u}_t)q_j(\bar{\bm u}_t)\right||u_{tj}-u_{tj}^\star|  +  \frac{3P}{2} \sum_{t=1}^T \sum_{i\in S_t}(u_{ti}-u_{ti}^\star)^2q_i(\bar{\bm u}_t)\\
    &\le \frac{P}{2}\sum_{t=1}^T \sum_{i\in S_t}\sum_{j\in S_t}(u_{ti}-u_{ti}^\star)^2q_i(\bar{\bm u}_t)q_j(\bar{\bm u}_t) + \frac{P}{2}\sum_{t=1}^T\sum_{i\in S_t}\sum_{j\in S_t}(u_{tj}-u_{tj}^\star)^2q_i(\bar{\bm u}_t)q_j(\bar{\bm u}_t) \\
    &\quad + \frac{3P}{2} \sum_{t=1}^T \sum_{i\in S_t}(u_{ti}-u_{ti}^\star)^2q_i(\bar{\bm u}_t)\\
    &\le \frac{5P}{2} \sum_{t=1}^T \sum_{i\in S_t}(u_{ti}-u_{ti}^\star)^2q_i(\bar{\bm u}_t).
\end{align*}
where the first inequality follows from Lemma \ref{lemE.3}, and the third inequality follows from AM-GM inequality.

Hence, the term $(\mathbf{II})$ in \eqref{eq:F.4} can be bounded by
\begin{align}\label{eq:F.7}
    \frac{1}{2} \sum_{t\notin\mathcal T_0}(\bm u_t-\bm u_t^\star)^\top \nabla^2 Q(\bar{\bm u}_t)(\bm u_t-\bm u_t^\star) &\le \frac{5P}{2} \sum_{t=1}^T \sum_{i\in S_t}(u_{ti}-u_{ti}^\star)^2q_i(\bar{\bm u}_t)
    \nonumber\\& \le 10P \sum_{t=1}^T \sum_{i\in S_t}q_i(\bar{\bm u}_t)\gamma_t^2\|\widetilde{\bm x}_{ti}(p_{ti})\|^2_{H_t(\nu^\star_t)^{-1}}
    \nonumber\\& \le 10P \gamma_t^2\sum_{t=1}^T \max_{i\in S_t}\|\widetilde{\bm x}_{ti}(p_{ti})\|^2_{H_t(\nu^\star_t)^{-1}}
    \nonumber\\& \le \frac{20}{\kappa}P \gamma_t^2 d\log \left(1+\frac{T}{d\lambda}\right)
\end{align}
where the last inequality follows from Lemma \ref{lemD.11}.

Therefore, combining results in \eqref{eq:F.6}, \eqref{eq:F.7} and \eqref{eq:F.3}, and setting $\lambda=\frac{1}{8W^2}$ and $\gamma_T=\mathcal{O}\left(\sqrt{d\log(WPT)}\right)$, we obtain
\begin{align*}
    Reg(T)  &\le P \frac{2}{\log(3/2)}d\log\left(1+\frac{\bar P^2}{\log(3/2)\lambda}\right)  +   2\sqrt2\gamma_TP\sqrt{dT\log\left(1+\frac{T}{d\lambda}\right)}+\frac{52}{\kappa}\gamma_T^2P d\log\left(1+\frac{T}{d\lambda}\right)\\
    &=\widetilde{\mathcal{O}}\left(\frac{W}{L_0}d\sqrt{T} +\frac{1}{\kappa}\frac{W }{L_0}d^2  \right)
\end{align*}
\end{proof}

%--------------------------------------------------------------%
\subsection{Bayesian regret via the online UCB analysis}

We now prove the Bayesian regret bound. The key idea is to use the optimistic revenue \(R_t^{\mathrm{UCB}}\) in \eqref{eq-RtUCB} as a bridge between the optimal revenue under the true parameter and the revenue collected by Thompson sampling.
   
\begin{proof}[\textbf{Proof of Theorem~\ref{thm-TS-regret}}] We first divide the Bayesian regret into two parts.
\begin{align*}
    BR_Q(T) & = \mathbb{E}_{\bm\vartheta\sim Q}\left[\sum_{t=1}^T \left(R(S^\star,\bm p^\star|\bm\vartheta) - R(S_t,\bm p_t|\bm\vartheta) \right) \right]\\
    &=  \E_{\bm\vartheta\sim Q}\left[\sum_{t=1}^T\E[R(S^\star,\bm p^\star|\bm\vartheta) - R_t^{\text{UCB}}(S^\star,\bm p^\star) + R_t^{\text{UCB}}(S_t,\bm p_t) - R(S_t,\bm p_t|\bm\vartheta) | \mathcal{F}_{t-1}] \right]\\
    &=  \E_{\bm\vartheta\sim Q} \underbrace{\left[\sum_{t=1}^T \left(R(S^\star,\bm p^\star|\bm\vartheta) - R_t^{\text{UCB}}(S^\star,\bm p^\star)\right)\right]}_{I_1} + \E_{\bm\vartheta\sim Q}\underbrace{\left[\sum_{t=1}^T \left(R_t^{\text{UCB}}(S_t,\bm p_t) - R(S_t,\bm p_t|\bm\vartheta)\right)\right]}_{I_2}  
\end{align*}
where $\mathcal{F}_{t-1}$ is the filtration generated by all historical observations up to time $t-1$, and the second line holds because $(S^\star,\bm p^\star)$ and $(S_t,\bm p_t)$ are identically distributed given $\mathcal{F}_{t-1}$.

Lemma~\ref{prop-non-uniform-reward-optimal-assortment} and Lemma~\ref{lem-confidence-sequential-MLE} imply that $R(S^\star,\bm p^\star|\bm\vartheta) \le R_t^{\text{UCB}}(S^\star,\bm p^\star)$, i.e., $I_1\le 0$. 

Similar to the proof of Theorem \ref{thm-MLE-regret}, we bound $I_2$ term by
\begin{align*}
   I_2
   &= \sum_{t\in\mathcal{T}_0}\left[ R_t^{\text{UCB}}(S_t,\bm p_t) - R(S_t,\bm p_t|\bm\vartheta)\right] + \sum_{t\notin\mathcal{T}_0}\left[ R_t^{\text{UCB}}(S_t,\bm p_t) - R(S_t,\bm p_t|\bm\vartheta)\right]\\
   &\le  P|\mathcal{T}_0| + \sum_{t\notin\mathcal{T}_0}\left[ R_t^{\text{UCB}}(S_t,\bm p_t) - R(S_t,\bm p_t|\bm\vartheta)\right]\\
   &\le  P \frac{2}{\log(3/2)}d\log\left(1+\frac{\bar P^2}{\log(3/2)\lambda}\right) + 2\sqrt2\gamma_T P\sqrt{dT\log\left(1+\frac{T}{d\lambda}\right)}+\frac{52}{\kappa}\gamma_T^2 P d\log\left(1+\frac{T}{d\lambda}\right)\\
   % &=\widetilde{\mathcal{O}}\left(\frac{W+\log K}{L_0}d\sqrt{T}\log\left(\frac{W+\log K}{L_0}WT\right)+\frac{1}{\kappa}\frac{W+\log K}{L_0}d^2\log^2\left(\frac{W+\log K}{L_0}WT\right)  \right)\\
    &=\widetilde{\mathcal{O}}\left(\frac{W}{L_0}d\sqrt{T} + \frac{1}{\kappa}\frac{W }{L_0}d^2 \right)
\end{align*}

Thus, we obtain the Bayesian regret
\begin{align*}
    &BR_Q(T)\le \E_{\bm\vartheta\sim Q}[I_1+I_2]=\widetilde{\mathcal{O}}\left(\frac{W}{L_0}d\sqrt{T} + \frac{1}{\kappa}\frac{W }{L_0}d^2 \right)
    % &\le \mathcal{O}\left(\frac{W+\log K}{L_0}d\sqrt{T}\log\left(\frac{W+\log K}{L_0}WT\right)+\frac{1}{\kappa}\frac{W+\log K}{L_0}d^2\log^2\left(\frac{W+\log K}{L_0}WT\right)  \right).
\end{align*}
\end{proof}
%% ---- END appendix_section6.tex ----

\end{document}